\documentclass[11pt,reqno]{amsart}
\usepackage{amsmath,amssymb}
\usepackage{mathrsfs}
\usepackage{amsfonts}
\usepackage[dvipsnames]{xcolor}
\usepackage{graphicx}
\usepackage{float}  
\usepackage{cite}
\usepackage{cases}
\usepackage{indentfirst}
\allowdisplaybreaks

\newtheorem{theorem}{Theorem}[section]
\newtheorem{lemma}[theorem]{Lemma}

\newtheorem{remark}[theorem]{Remark}
\newtheorem{proposition}[theorem]{Proposition}

\newtheorem{assumption}[theorem]{Assumption}
\numberwithin{equation}{section}

\makeatletter

\begin{document}
\author{Yasheng Lyu}

\title{\textbf{On the Dirichlet problem for the degenerate $k$-Hessian equation} }

\address{School of Mathematics and Statistics, Xi'an Jiaotong University, Xi'an, Shaanxi 710049, People's Republic of China}

\email{lvysh21@stu.xjtu.edu.cn}
%**********************************************************************************

\begin{abstract}
This paper investigates the existence of a global $C^{1,1}$ solution to the Dirichlet problem for the $k$-Hessian equation with a nonnegative right-hand side $f$, focusing on the required conditions for $f$. 
The conditions $f^{1/(k-1)}\in C^{1,1}(\overline{\Omega_{0}})$ and $f^{3/(2k-2)}\in C^{2,1}(\overline{\Omega_{0}})$, together with $f\geq0$ in a domain $\Omega_{0}\Supset\Omega$, are optimal, as demonstrated by classical counterexamples. 
For the Monge-Amp{\`e}re equation ($k=n$), we establish the existence under the optimal condition $f^{3/(2n-2)}\in C^{2,1}(\overline{\Omega_{0}})$ together with $f\geq0$ in $\Omega_{0}$.
For the general $k$-Hessian equation, we establish the existence under the condition $f\geq0$ in $\Omega_{0}$ together with one of the following three conditions:
\begin{align*}
&(\romannumeral1)\quad f^{1/(k-1)}\in C^{1,1}(\overline{\Omega_{0}}),\ \ \inf_{\Omega}\Delta u\geq1,\ \ 2\leq k\leq n-1;\\
&(\romannumeral2)\quad f^{3/(2k-2)}\in C^{2,1}(\overline{\Omega_{0}}),\ \ \inf_{\Omega}\Delta u\geq1,\ \ 5\leq k\leq n-1;\\
&(\romannumeral3)\quad f^{3/(2k)}\in C^{2,1}(\overline{\Omega_{0}}),\ \ 2\leq k\leq n-1.
\end{align*}
\end{abstract}


 \date{}
\maketitle

\pagestyle{myheadings}
\markboth{$~$ \hfill{\uppercase{Yasheng Lyu}}\hfill $~$}{$~$ \hfill{\uppercase{On the degenerate $k$-Hessian equation}}\hfill $~$}

\section{Introduction}

In this paper, we consider the Dirichlet problem for the degenerate $k$-Hessian equation 
\begin{equation}\label{eqn1.1}
\begin{cases}
\sigma_{k}\big[D^{2}u\big]=f\geq0\quad \text{in}\ \Omega,\\
u=\varphi\quad \text{on}\ \partial\Omega,
\end{cases}
\end{equation}
where $\Omega$ is a bounded domain in $\mathbb{R}^{n}$, $2\leq k\leq n$, and $D^{2}u$ denotes the Hessian matrix of the function $u$.  
For an $n\times n$ real symmetric matrix $A$, we denote by $\sigma_{k}[A]$ the  $k$-th elementary symmetric polynomial of the eigenvalues $\lambda(A)=(\lambda_{1},\dots,\lambda_{n})$; that is, 
\[
\sigma_{k}[A]=\sigma_{k}(\lambda(A)):=\sum_{1\leq i_{1}<\cdots<i_{k}\leq n}\lambda_{i_{1}}\cdots\lambda_{i_{k}} 
\]
for $k=1,2,\dots,n$.
Equivalently, $\sigma_{k}[A]$ is the sum of all $k\times k$ principal minors of $A$.

Following the work of Caffarelli-Nirenberg-Spruck \cite{CNS1985}, a function $u\in C^{2}(\Omega)\cap C^{0}(\overline{\Omega})$ is called $k$-admissible if its Hessian eigenvalues satisfy  
\[
\lambda\big(D^{2}u(x)\big)\in\Gamma_{k},\quad \forall x\in\Omega,
\]
where the G{\aa}rding cone is defined by 
\[
\Gamma_{k}:=\left\{\lambda\in\mathbb{R}^{n}:\ \sigma_{j}(\lambda)>0\ \text{for all}\ 1\leq j\leq k\right\}.
\]
They proved that  
\[
\frac{\partial\sigma_{k}}{\partial\lambda_{i}}>0\quad \text{in}\ \Gamma_{k},\quad\forall1\leq i\leq n;
\]
that $\sigma_{k}^{1/k}(\lambda)$ is concave in $\Gamma_{k}$; 
and that the operator $\sigma_{k}^{1/k}[A]$ is elliptic and concave for any real symmetric matrix $A$ with $\lambda(A)\in\Gamma_{k}$.  
A hypersurface in $\mathbb{R}^{n}$ is called $k$-convex (for $k\in\{1,2,\dots,n-1\}$) if its principal curvatures $\kappa=(\kappa_{1},\kappa_{2},\dots,\kappa_{n-1})$ satisfy $\sigma_{j}(\kappa)\geq0$ for all $j\leq k$ everywhere. 
It is strictly $k$-convex if there exists a constant $\delta>0$ such that $\sigma_{j}(\kappa)\geq\delta$ uniformly on the hypersurface for all $j\leq k$.
As shown in \cite{CNS1985}, a necessary condition for the existence of a $C^{2}(\overline{\Omega})$ solution to \eqref{eqn1.1} is that the boundary $\partial\Omega$ be strictly $(k-1)$-convex when regarded as a hypersurface in $\mathbb{R}^{n}$.

The $k$-Hessian equation is called degenerate if the nonnegative function $f$ is allowed to vanish at some points in $\overline{\Omega}$, and non-degenerate if $\inf_{\Omega}f>\delta$ for some positive constant $\delta$. 
This paper focuses on the degenerate case. 
The condition $f^{1/k}\in C^{1,1}(\overline{\Omega})$ is natural when exploiting the concavity of the operator.
However, the condition $f^{1/k}\in C^{1,1}$ may not be optimal.
Regarding the optimal condition, a counterexample for the Monge-Amp{\`e}re equation was given by Wang \cite{Wang1995}:
\begin{equation}\label{eqn7.6}
\begin{cases}
\det D^{2}u=\bar{f}(x):=\eta\left(\frac{x_{n}}{|x'|^{\alpha}}\right)\left|x'\right|^{\beta}\quad \text{in}\ B_{1}, \\
u=0\quad \text{on}\ \partial B_{1}, 
\end{cases}
\end{equation}
where $\alpha>1$, $\beta>0$, and 
\[
\eta(t):=
\begin{cases}
e^{-1/(1-t^{2})},\quad |t|<1,\\
0,\quad |t|\geq1. 
\end{cases}
\]
Wang \cite{Wang1995} proved that if $\beta<2(n-1)(\alpha-1)$, then \eqref{eqn7.6} admits no $C^{1,1}$ solution. 
Taking $\beta=2(n-1)(\alpha-1)-1$, one has $\bar{f}\geq0$ in $B_{2}$, and 
\[
\bar{f}^{\frac{1}{n-1-\frac{2n-1}{2\alpha}}}\in C^{1,1}(\overline{B_{2}}),\quad \bar{f}^{\frac{1}{n-1}}\in C^{1,1-\frac{2n-1}{\alpha(n-1)}}(\overline{B_{2}}),\quad \forall\alpha\geq4, 
\]
and 
\[
\bar{f}^{\frac{1}{\frac{2(n-1)}{3}-\frac{2n-1}{3\alpha}}}\in C^{2,1}(\overline{B_{2}}),\quad 
\bar{f}^{\frac{3}{2(n-1)}}\in C^{2,1-\frac{6n-3}{2\alpha(n-1)}}(\overline{B_{2}}),\quad \forall\alpha\geq6. 
\]
Ivochkina-Trudinger-Wang \cite{Wang2004} pointed out that a modification of the counterexample \eqref{eqn7.6} also applies to the $k$-Hessian
equation for $2\leq k\leq n-1$, although no proof was provided.
A detailed proof of this fact was given by Dinew-Pli{\'s}-Zhang \cite{Zhang2019};  namely, the problem 
\[
\begin{cases}
\sigma_{k}\big[D^{2}u\big]=\bar{f}\quad \text{in}\ B_{1}, \\
u=0\quad \text{on}\ \partial B_{1}, 
\end{cases}
\]
has no $C^{1,1}$ solution whenever $\beta<2(k-1)(\alpha-1)$, for all $2\leq k\leq n-1$.
Consequently, for any $k\in\{2,3,\dots,n\}$, 
\[
f^{\frac{1}{k-1}}\in C^{1,1}\quad \text{and}\quad f^{\frac{3}{2(k-1)}}\in C^{2,1} 
\]
are optimal. 

For the Monge-Amp{\`e}re equation, Guan \cite{Guan1997} reduced the estimate of second-order derivatives to the boundary under the optimal condition $f^{1/(n-1)}\in C^{1,1}(\overline{\Omega})$, and obtained the existence of a convex solution in $C^{1,1}(\overline{\Omega})$ under the homogeneous boundary condition ($\varphi=0$).  
Subsequently, by introducing a completely new method, Guan-Trudinger-Wang \cite{Guan1999} established a global $C^{2}$ estimate for the Monge-Amp{\`e}re equation under the optimal condition $f^{1/(n-1)}\in C^{1,1}(\overline{\Omega})$, which in turn yielded the existence of a convex solution in $C^{1,1}(\overline{\Omega})$.
For the case $2\leq k\leq n-1$ of equation \eqref{eqn1.1}, Ivochkina-Trudinger-Wang \cite{Wang2004} proposed the following open problem:    
\begin{equation}\label{eqn1.5}
\text{existence of a global $C^{1,1}$ solution under $f^{1/(k-1)}\in C^{1,1}$}.  
\end{equation} 
Dong \cite{Dong2006} solved problem \eqref{eqn1.5} under the homogeneous boundary condition.  
For general boundary conditions, Jiao-Wang \cite{Jiao2024} derived a global $C^{2}$ estimate for convex solutions to problem \eqref{eqn1.5}.  
However, since smooth approximations of convex viscosity solutions do not necessarily preserve convexity, they \cite{Jiao2024} were unable to establish the existence of a $k$-admissible solution in $C^{1,1}(\overline{\Omega})$ for equation \eqref{eqn1.1}. 
There is an extensive literature on degenerate $k$-Hessian equations; see, for instance, \cite{CNS1986,Hong1994,Krylov1987,Krylov1988,Trudinger1987,Trudinger1983,Trudinger1984,Guan1997a,Guan1998,Guan2021,Xu2014,Li1999,Wang2001}, although this list is far from exhaustive.

A common strategy is to derive global $C^{2}$ a priori estimate for the non-degenerate case that are independent of the positive lower bound of $f$. 
Then the existence of a $k$-admissible solution in $C^{1,1}(\overline{\Omega})$ for the degenerate case is established by an approximation argument using solutions to non-degenerate problems.  
To establish global $C^{2}$ a priori estimate independent of the positive lower bound of $f$, the primary challenge is to derive the second-order normal derivative estimate on the boundary.

For the Monge-Amp{\`e}re equation, using the convexity of the admissible solutions and of the domain, as well as the affine invariance of the equation,  Guan-Trudinger-Wang \cite{Guan1999} introduced a method to estimate the second-order normal derivatives on the boundary.
Recently, Jiao-Wang \cite{Jiao2024} extended this method to convex solutions of the $k$-Hessian equation in uniformly convex domains. 
Applying the method of Guan-Trudinger-Wang \cite{Guan1999}, we establish the existence of a global $C^{1,1}$ solution for the Monge-Amp{\`e}re equation under the optimal condition. 

\begin{theorem}\label{thm10.1}
Let $\partial\Omega\in C^{3,1}$ be uniformly convex,  
$\varphi\in C^{3,1}(\partial\Omega)$, 
$f^{3/(2n-2)}\in C^{2,1}(\overline{\Omega_{0}})$ and $f\geq0$ in $\Omega_{0}$, where $\Omega\Subset\Omega_{0}$.    
Then there exists a unique convex solution in $C^{1,1}(\overline{\Omega})$ for  \eqref{eqn1.1} with $k=n$.
\end{theorem}

For general $k$-Hessian equations, we employ the method of Ivochkina-Trudinger-Wang \cite{Wang2004}, which is completely different from that of \cite{Guan1999}, to establish the boundary second-order normal derivative estimate.
Ivochkina-Trudinger-Wang \cite{Wang2004} provided a PDE-based proof of Krylov’s results \cite{Krylov1994a,Krylov1994b,Krylov1995a,Krylov1995b}, and, roughly speaking, their ideas are similar. 
These results and ideas were clarified very well in \cite{Wang2004}, and we quote them below. 
\begin{quote}
A global upper bound, independent of the positive lower bound of $f$, for the second-order derivatives of admissible solutions to the Dirichlet problem of Hessian equations, and more general Bellman equations, was obtained by Krylov \cite{Krylov1990}, 
by a probabilistic argument, and by an analytic proof in a series of papers Krylov \cite{Krylov1994a,Krylov1994b,Krylov1995a,Krylov1995b}. 
Roughly speaking, Krylov’s proof  \cite{Krylov1994a,Krylov1994b,Krylov1995a,Krylov1995b} consists of two steps. 
One is the \textit{weakly interior estimate}, that is for any positive constants $\varepsilon,\ \delta>0$, there exists a constant $C_{\varepsilon,\delta}$, depending on $n$, $\varphi,\ \partial\Omega,\ \|u\|_{C^{1}(\overline{\Omega})},\ \|f\|_{C^{1,1}(\overline{\Omega})}$, and in particular the upper bound of the boundary second-order tangential and tangential-normal derivatives of $u$, such that
\begin{equation}\label{eqn1.2}
\sup_{\Omega_{\delta}}\left|D^{2}u\right|\leq\varepsilon\sup_{\partial\Omega}\left|D^{2}u\right|+C_{\varepsilon,\delta},
\end{equation}
where $\Omega_{\delta}:=\{x\in\Omega:\ \operatorname{dist}(x,\partial\Omega)>\delta\}$.
The other one is the \textit{boundary estimate in terms of the interior one}, that is for any $\delta>0$, there exists a constant $C_{\delta}$, depending in addition on $C_{\varepsilon_{0},\delta}$ in \eqref{eqn1.2}, with arbitrarily given  $\varepsilon_{0}>0$, such that
\begin{equation}\label{eqn1.3}
\sup_{\partial\Omega}\left|D^{2}u\right|\leq C_{\delta}\left(1+\sup_{\partial\Omega_{\delta}}\left|D^{2}u\right|\right).
\end{equation} 
\end{quote} 

Firstly, we establish the \textit{weakly interior estimate} for degenerate $k$-Hessian equations.

\begin{theorem}\label{thm1.1}
Let $u\in C^{3,1}(\overline{\Omega})$ be a $k$-admissible solution of equation \eqref{eqn1.1}, 
$\partial\Omega\in C^{3,1}$ be strictly $(k-1)$-convex, 
$\varphi\in C^{3,1}(\partial\Omega)$, $\inf_{\Omega_{0}}f>0$ where  $\Omega\Subset\Omega_{0}$.  
Let either $f^{1/(k-1)}\in C^{1,1}(\overline{\Omega_{0}})$ hold or  $f^{3/(2k-2)}\in C^{2,1}(\overline{\Omega_{0}})$ and $k\geq5$ hold. 
Then, for any constant $\varepsilon>0$, there exists a constant $C_{\varepsilon}>0$ such that
\begin{equation}\label{eqn1.7}
\left|D^{2}u(x)\right|\leq \frac{1}{d(x)}\left(\varepsilon\sup_{\partial\Omega}\left|D^{2}u\right|+C_{\varepsilon}\right),\quad \forall x\in\Omega, 
\end{equation}
where $d(x):=\operatorname{dist}(x,\partial\Omega)$, and the constant $C_{\varepsilon}$ depends on $\varepsilon$, $\Omega$, $\|\varphi\|_{C^{3}(\partial\Omega)}$,  $\operatorname{dist}(\Omega,\partial\Omega_{0})$, and either $\|f^{1/(k-1)}\|_{C^{1,1}(\overline{\Omega_{0}})}$ or  $\|f^{3/(2k-2)}\|_{C^{2,1}(\overline{\Omega_{0}})}$, but is independent of $\inf_{\Omega_{0}}f$. 
\end{theorem} 

Secondly, the \textit{boundary estimate in terms of the interior one} for degenerate $k$-Hessian equations is also obtained, but under an additional a priori condition: 
there exists a small positive constants $r_{0}$ such that 
\begin{equation}\label{eqn1.6}
\Delta u\geq1\quad \text{in}\ \Omega\cap B_{r_{0}}(y),
\end{equation}
for all $y\in\partial\Omega$ at which $\Delta u(y)\gg1$.
Then the boundary second-order normal derivative estimate is established.

\begin{theorem}\label{thm10.3}
Let $\partial\Omega\in C^{3,1}$ be strictly $(k-1)$-convex, 
$\varphi\in C^{3,1}(\partial\Omega)$, and $f\geq0$ in $\Omega_{0}$ where $\Omega\Subset\Omega_{0}$.   
Let one of the following three conditions hold:
\begin{align*}
&(\romannumeral1)\quad f^{1/(k-1)}\in C^{1,1}(\overline{\Omega_{0}}),\ \ 2\leq k\leq n-1,\ \ \text{and}\ \eqref{eqn1.6}\ \text{holds};\\
&(\romannumeral2)\quad f^{3/(2k-2)}\in C^{2,1}(\overline{\Omega_{0}}),\ \ 5\leq k\leq n-1,\ \ \text{and}\ \eqref{eqn1.6}\ \text{holds};\\
&(\romannumeral3)\quad f^{3/(2k)}\in C^{2,1}(\overline{\Omega_{0}}),\ \ 2\leq k\leq n-1.
\end{align*}
Then there exists a unique $k$-admissible solution in $C^{1,1}(\overline{\Omega})$ for \eqref{eqn1.1}.
\end{theorem}

\begin{remark}[On condition $(\romannumeral1)$]
If condition \eqref{eqn1.6} in $(\romannumeral1)$ could be removed, then the longstanding open problem \eqref{eqn1.5} would be resolved.
When $\sup_{\partial\Omega}\Delta u$ is bounded from above, the $C^{2}$ estimate follows readily. 
We hope that Theorem \ref{thm10.3} may provide some insight into the resolution of the open problem \eqref{eqn1.5}.
Moreover, condition \eqref{eqn1.6} is automatically satisfied if $f$ admits a positive lower bound near $\partial\Omega$.  
\end{remark}

\begin{remark}[On condition $(\romannumeral3)$]
As pointed out by Guan-Trudinger-Wang \cite{Guan1999}, for applications to problems in differential geometry, it is desirable to impose no restriction on the nonnegative function $f$ apart from smoothness. 
Guan \cite{Guan1997} established the existence of a global $C^{1,1}$ solution for the Monge-Amp{\`e}re equation with homogeneous boundary condition in dimension $n=3$, assuming $0\leq f\in C^{3,1}(\mathbb{R}^{3})$. 
In comparison with the condition $f^{1/(k-1)}\in C^{1,1}$, it is interesting that the exponent of $f$ can be improved when higher regularity of $f$ is available, for $k\geq4$. 
\end{remark}

\begin{remark}
There is no implication between the conditions $f^{1/k}\in C^{1,1}(\overline{\Omega_{0}})$ and $f^{1/(k-q)}\in C^{2,1}(\overline{\Omega_{0}})$ for any $0<q<k$.  
This can be seen from the following two counterexamples.
One counterexample is $f^{1/(k-q)}:=|x|^{2}\in C^{\infty}(\overline{\Omega_{0}})$ with $0\in\partial\Omega$.
It is clear that $f^{1/k}=|x|^{2(k-q)/k}\notin C^{1,1}(\overline{\Omega_{0}})$.
The other counterexample is 
\[
f^{1/k}:=\left(1+g(x)\right) e^{-1/|x-x_{0}|} 
\]
in $\overline{\Omega_{0}}$, where $0$ and $x_{0}$ are two distinct points on $\partial\Omega$, and  
\begin{equation*}
g(x):=
\begin{cases}
|x|^{4}\left(1+\sin(1/|x|)\right),\quad x\neq0,\\
0,\quad x=0.
\end{cases}
\end{equation*}
One verifies that $g\in W^{2,\infty}$, hence $g\in C^{1,1}(\overline{\Omega_{0}})$, while $D^{2}g$ is not continuous at $0$. 
Thus $f^{1/k}\in C^{1,1}(\overline{\Omega_{0}})$, but $D^{2}(f^{1/(k-q)})$ fails to be continuous at $0$. 
\end{remark}

The rest of the paper is organized as follows. 
Section 2 presents preliminary results.
Section 3 establishes the global $C^{1}$ estimate and the boundary second-order tangential and tangential-normal estimates, and reduces the estimate of second-order derivatives to the boundary. 
Section 4 deals with the Monge-Amp{\`e}re equation under the optimal condition, proving Theorem \ref{thm10.1}. 
Section 5 is devoted to the \textit{weakly interior estimate}.
Section 6 develops the \textit{boundary estimate in terms of the interior one} to finish Theorem \ref{thm10.3}. 

Throughout the paper, the dependence of constants on the dimension $n$ and the number $k$ in \eqref{eqn1.1} is omitted, and a constant depending only on $n$ and $k$ is called universal. 
This section ends with a list of notations.
\par
\vspace{2mm}
\noindent\textbf{Notation.}
\par
\vspace{2mm}
\noindent 1. $e_{i}=(0,\dots,1,\dots,0)$ denotes the $i$-th standard coordinate vector.

\noindent 2. $x'=(x_{1},\dots,x_{n-1})$ and $x=(x',x_{n})$. 

\noindent 3. $B_{r}(x_{0})=\{x\in\mathbb{R}^{n}:\ |x-x_{0}|\leq r\}$ and $x_{0}$ is omitted when $x_{0}=0$. 

\noindent 4. $\operatorname{dist}(E,F)=\text{distance from $E$ to $F$}$,\quad $\forall E,\ F\subset\mathbb{R}^{n}$.

\noindent 5. $\mathbb{S}^{n-1}$ denotes the unit sphere in $\mathbb{R}^{n}$. 

\noindent 6. $\mathbf{I}_{n}$ denotes the $n\times n$ identity matrix. 

\noindent 7. $A'=A^{T}$ denotes the transpose of the matrix $A$. 

\noindent 8. $A^{1/2}$ denotes the square root of positive semi-definite matrix $A$.

\noindent 9. $\nabla_{x'}=(\partial_{1},\partial_{2},\dots,\partial_{n-1})$.

\noindent 10. $D^{2}_{x'x'}$ denotes a matrix with elements $\partial_{ij}$ where $1\leq i,\ j\leq n-1$.   

\noindent 11. $\nu=\nu(x)$ denotes the unit interior normal vector of $\partial\Omega$ at $x\in\partial\Omega$.

\section{Preliminaries}

The following pointwise inequalities are fundamental in this paper.  

\begin{lemma}\label{thm2.2}
Let $g$ be a function satisfying $g>0$ in $\Omega_{0}$, where $\Omega\Subset\Omega_{0}$.
We have  

$(\romannumeral1)$ If $g\in C^{1,1}(\overline{\Omega_{0}})$, then 
\begin{equation}\label{eqn7.1}
\frac{\left|\nabla g(x)\right|^{2}}{g(x)}\leq K,\quad \forall x\in\overline{\Omega},
\end{equation}
where $K$ depends on $\|g\|_{C^{1,1}(\overline{\Omega_{0}})}$ and $\operatorname{dist}(\Omega,\partial\Omega_{0})$, but is independent of $\inf_{\Omega_{0}} g$.

$(\romannumeral2)$ If $g\in C^{2,1}(\overline{\Omega_{0}})$, then for any $\alpha<1/2$, 
\begin{equation}\label{eqn7.2}
\partial_{ee}g(x)-\alpha\frac{\left|\partial_{e}g(x)\right|^{2}}{g(x)}\geq-Kg^{\frac{1}{3}}(x),\quad \forall x\in\overline{\Omega},\ e\in\mathbb{S}^{n-1},
\end{equation}
where $K$ depends on $\alpha$, $\|g\|_{C^{2,1}(\overline{\Omega_{0}})}$, and  $\operatorname{dist}(\Omega,\partial\Omega_{0})$, but is independent of $\inf_{\Omega_{0}} g$.
\end{lemma}
\begin{proof}
In fact, we only need to consider the case of one dimension.
Let $h(x)$ be a function in $C^{2,1}(\mathbb{R})$ and $h>0$ on $\mathbb{R}$. 
Denote $\beta:=h(0)>0$. 

\par
\vspace{2mm}
\textbf{Step 1:} \emph{we prove conclusion $(\romannumeral1)$.} 
 
It holds that  
\begin{equation*}
\begin{cases}
0\leq h\big(\sqrt{\beta}\big)=\beta+h'(0)\sqrt{\beta}+O(\beta)\\
0\leq h\big(-\sqrt{\beta}\big)=\beta-h'(0)\sqrt{\beta}+O(\beta).
\end{cases}
\end{equation*} 
Thus 
\[
\frac{\left|h'(0)\right|}{\sqrt{\beta}}\leq1+K,
\]
where $K=\|h\|_{C^{1,1}(\mathbb{R})}$. 
By virtue of the above, \eqref{eqn7.1} are valid at points where $(g(x))^{1/2}<\operatorname{dist}(\Omega,\partial\Omega_{0})$. 
And \eqref{eqn7.1} hold clearly at points where $(g(x))^{1/2}\geq\operatorname{dist}(\Omega,\partial\Omega_{0})$.  

\par
\vspace{2mm}
\textbf{Step 2:} \emph{we prove conclusion $(\romannumeral2)$.} 

It holds that  
\[
0\leq h\big(\beta^{t_{1}}\big)+h\big(-\beta^{t_{1}}\big)=2\beta+h''(0)\beta^{2t_{1}}+O\big(\beta^{3t_{1}}\big), 
\]
where $t_{1}>0$ is to be determined. 
Thus 
\[
h''(0)\geq-2\beta^{1-2t_{1}}+O\big(\beta^{t_{1}}\big).
\]
Taking $t_{1}=1/3$ yields 
\begin{equation}\label{eqn7.3}
h''(0)\geq-C\beta^{\frac{1}{3}},
\end{equation}
where $C$ depends only on $\|h\|_{C^{2,1}(\mathbb{R})}$. 

It holds that  
\begin{equation*}
\begin{cases}
0\leq h\big(\beta^{t_{2}}\big)=\beta+h'(0)\beta^{t_{2}}+\frac{1}{2}h''(0)\beta^{2t_{2}}+O\big(\beta^{3t_{2}}\big)\\
0\leq h\big(-\beta^{t_{2}}\big)=\beta-h'(0)\beta^{t_{2}}+\frac{1}{2}h''(0)\beta^{2t_{2}}+O\big(\beta^{3t_{2}}\big),
\end{cases}
\end{equation*}
where $t_{2}>0$ is to be determined. 
Thus  
\[
\frac{\left|h'(0)\right|}{\sqrt{\beta}}\leq\beta^{\frac{1}{2}-t_{2}}+\frac{1}{2}h''(0)\beta^{t_{2}-\frac{1}{2}}+O\big(\beta^{2t_{2}-\frac{1}{2}}\big).
\]
Taking $t_{2}=1/3$ yields 
\begin{equation}\label{eqn7.4}
\frac{\left|h'(0)\right|}{\sqrt{\beta}}\leq\frac{1}{2}h''(0)\beta^{-\frac{1}{6}}+C\beta^{\frac{1}{6}}, 
\end{equation}
where $C$ depends only on $\|h\|_{C^{2,1}(\mathbb{R})}$. 

\par
\vspace{2mm}
\textbf{Case 1}: \emph{$h''(0)\leq K_{1}\beta^{1/3}$, where $K_{1}>0$ is to be determined.} 
It follows from \eqref{eqn7.4} that 
\[
\frac{\left|h'(0)\right|^{2}}{\beta}\leq\left(K_{1}/2+C\right)^{2}\beta^{\frac{1}{3}}.
\]
Combining this inequality and \eqref{eqn7.3}, we get 
\[
h''(0)-\alpha\frac{\left|h'(0)\right|^{2}}{h(0)}\geq-\left( C+\alpha(K_{1}/2+C)^{2}\right) h^{\frac{1}{3}}(0)\quad \text{in Case 1}. 
\]

\par
\vspace{2mm}
\textbf{Case 2}: \emph{$h''(0)>K_{1}\beta^{1/3}$.}
It holds that 
\begin{equation*}
\begin{cases}
0\leq h\big(t_{3}\sqrt{\beta}\big)=\beta+h'(0)t_{3}\sqrt{\beta}+\frac{1}{2}h''(0)\left(t_{3}\sqrt{\beta}\right)^{2}+O\big(\big(t_{3}\sqrt{\beta}\big)^{3}\big)\\
0\leq
h\big(-t_{3}\sqrt{\beta}\big)=\beta-h'(0)t_{3}\sqrt{\beta}+\frac{1}{2}h''(0)\left(t_{3}\sqrt{\beta}\right)^{2}+O\big(\big(t_{3}\sqrt{\beta}\big)^{3}\big), 
\end{cases}
\end{equation*}
where $t_{3}>0$ is to be determined. 
Thus 
\[
\frac{\left|h'(0)\right|}{\sqrt{\beta}}\leq\frac{1}{t_{3}}+\frac{1}{2}h''(0)t_{3}+Ct_{3}^{2}\beta^{\frac{1}{2}},
\]
where $C$ depends only on $\|h\|_{C^{2,1}(\mathbb{R})}$. 
Taking $t_{3}=\sqrt{2}\left(h''(0)\right)^{-1/2}$, which implies 
\[
t_{3}\sqrt{\beta}=\sqrt{2}\left(h''(0)\right)^{-\frac{1}{2}}\sqrt{\beta}<\sqrt{2}K_{1}^{-\frac{1}{2}}\beta^{\frac{1}{3}}, 
\]
we have 
\begin{align*}
\frac{\left|h'(0)\right|}{\sqrt{\beta}}&\leq\sqrt{2}\sqrt{h''(0)}+2C\left(h''(0)\right)^{-\frac{3}{2}}\beta^{\frac{1}{2}}\sqrt{h''(0)}\\
&<\left(\sqrt{2}+2CK_{1}^{-\frac{3}{2}}\right)\sqrt{h''(0)}.
\end{align*}
Given that $\alpha<1/2$, choose $K_{1}$ sufficiently large such that 
\[
\sqrt{2}+2CK_{1}^{-\frac{3}{2}}\leq\frac{1}{\sqrt{\alpha}}.
\]
It follows that 
\[
h''(0)-\alpha\frac{\left|h'(0)\right|^{2}}{h(0)}\geq0\quad \text{in Case 2}.
\]

Combining Case 1 and Case 2 gives  
\[
h''(0)-\alpha\frac{\left|h'(0)\right|^{2}}{h(0)}\geq-Kh^{\frac{1}{3}}(0), 
\]
where $K$ depends on $\|h\|_{C^{2,1}(\mathbb{R})}$ and $\alpha$. 
By virtue of the above, \eqref{eqn7.2} are valid at points in $\overline{\Omega}$ satisfying 
\[
\max\left\{1,\sqrt{2}K_{1}^{-1/2}\right\}g^{\frac{1}{3}}(x)<\operatorname{dist}(\Omega,\partial\Omega_{0}). 
\] 
And \eqref{eqn7.2} hold clearly at other points in $\overline{\Omega}$. 
The proof of Lemma \ref{thm2.2} is complete. 
\end{proof}

\begin{remark}\label{thm2.9}
In general, it is impossible to weaken the condition $g>0$ in $\Omega_{0}$ to $g>0$ in $\Omega$. 
A counterexample was provided by B{\l}ocki \cite{Blocki2003}: the function  $g(x):=x$ on $[0,1]$ is smooth but $\sqrt{g}\notin C^{0,1}([0,1])$.  
However, the condition $g>0$ in $\Omega$ suffices to estimate the derivatives of $g$ in directions tangential to $\partial\Omega$. 
\end{remark}

\begin{remark}
In general, it is impossible to improve the exponent $\alpha$ to any $\alpha>1/2$. 
Consider the family of functions $g_{\beta}(x):=(x+\beta^{1/2})^{2}$, where $\beta>0$ is a parameter.
It is clear that $g_{\beta}>0$ and $g_{\beta}\in C^{\infty}(\mathbb{R})$. 
Computing \eqref{eqn7.2} at $x=0$ and letting $\beta\rightarrow0$ yield that $\alpha\leq1/2$. 
\end{remark}

By convention, we set $\sigma_{0}(\lambda)\equiv 1$.
Some fundamental properties of $\sigma_k$ are listed below, and one can refer to \cite{Spruck2005,Hardy1952,Lieberman1996,Wang2009}. 

\begin{proposition}\label{thm2.1}
For any $k\in\{2,3,\dots,n\}$ and any symmetric matrix $A$ with $\lambda(A)\in\Gamma_{k}$, the following hold:
\begin{equation}\label{eqn2.1}
\sigma_{k}^{\frac{1}{k}}[A]\leq C_{1}\sigma_{k-1}^{\frac{1}{k-1}}[A];
\end{equation}
\begin{equation}\label{eqn2.2}
\sigma_{k-1}[A]\geq C_{2}\sigma_{1}^{\frac{1}{k-1}}[A]\sigma_{k}^{\frac{k-2}{k-1}}[A];
\end{equation}
\begin{equation}\label{eqn2.3}
\sum_{i=1}^{n}\sigma_{k}^{ii}[A]=(n-k+1)\sigma_{k-1}[A],
\end{equation}
where $C_{1}>0$ and $C_{2}>0$ are universal. 
\end{proposition}

\par
\vspace{2mm}
The equation \eqref{eqn1.1} is equivalent to 
\begin{equation}\label{eqn2.5}
F\big[D^{2}u\big]=g^{\frac{p}{k}}\quad \text{in}\ \Omega, 
\end{equation}
where 
\[
F[A]:=\sigma_{k}^{\frac{1}{k}}[A]\quad \text{and}\quad g:=f^{\frac{1}{p}}, 
\]
with 
\begin{equation}\label{eqn7.7}
p:=
\begin{cases}
k-1,\quad \text{if}\ f^{1/(k-1)}\in C^{1,1}(\overline{\Omega_{0}}),\\
(2k-2)/3,\quad \text{if}\ f^{3/(2k-2)}\in C^{2,1}(\overline{\Omega_{0}}). 
\end{cases}
\end{equation}
Define $F^{ij}[A]:=\partial F/\partial A_{ij}$ and $F^{ij,st}[A]:=\partial^{2} F/\partial A_{ij}\partial A_{st}$. 
Differentiating equation \eqref{eqn2.5} in the direction $e\in\mathbb{S}^{n-1}$ gives
\begin{equation}\label{eqn2.7}
F^{ij}\big[D^{2}u\big]u_{eij}=\frac{p}{k}g^{\frac{p}{k}-\frac{1}{2}}\frac{\partial_{e}g}{\sqrt{g}}=\left(\frac{p}{k}g^{\frac{p}{k-1}-\frac{1}{2}}\frac{\partial_{e}g}{\sqrt{g}}\right)f^{\frac{-1}{k(k-1)}}
\end{equation}
for all $e\in\mathbb{S}^{n-1}$. 
A second differentiation combined with the concavity of $F$ yields 
\[
\partial_{ee}\left(g^{\frac{p}{k}}\right)=F^{ij}\big[D^{2}u\big]u_{eeij}+F^{ij,st}\big[D^{2}u\big]u_{eij}u_{est}\leq F^{ij}\big[D^{2}u\big]u_{eeij}, 
\]
which implies 
\begin{equation}\label{eqn2.16}
F^{ij}\big[D^{2}u\big]u_{eeij}\geq\frac{p}{k}g^{\frac{p}{k}-1}\left(\partial_{ee}g-\frac{k-p}{k}\frac{\left|\partial_{e}g\right|^{2}}{g}\right)
\end{equation}
for all $e\in\mathbb{S}^{n-1}$. 

Define a matrix-valued function $G$ with entries 
\begin{equation}\label{eqn2.6}
G^{ij}(x):=\frac{F^{ij}\big[D^{2}u(x)\big]}{\operatorname{tr}F^{ij}\big[D^{2}u(x)\big]},\quad \forall x\in\Omega. 
\end{equation}

\begin{lemma}\label{thm2.3}
Let $u\in C^{3,1}(\overline{\Omega})$ be a $k$-admissible solution of equation \eqref{eqn1.1}, and $\inf_{\Omega_{0}}f>0$ where  $\Omega\Subset\Omega_{0}$.  
Let $f^{1/(k-1)}\in C^{1,1}(\overline{\Omega_{0}})$ hold or $f^{3/(2k-2)}\in C^{1,1}(\overline{\Omega_{0}})$ hold. 
Then  
\begin{equation}\label{eqn2.11}
\operatorname{tr}F^{ij}\big[D^{2}u\big]\geq\frac{1}{C}\max\left\{1,\sigma_{1}^{\frac{1}{k-1}}\big[D^{2}u\big]f^{\frac{-1}{k(k-1)}}\right\};
\end{equation}
\begin{equation}\label{eqn2.12}
0\leq\operatorname{tr}\left(GD^{2}u\right)\leq K;
\end{equation}
\begin{equation}\label{eqn2.13}
\left|\operatorname{tr}\left(GD^{2}u_{\xi}\right)\right|\leq K\left|\xi\right|\sigma_{1}^{\frac{-1}{k-1}}\big[D^{2}u\big],\quad \forall \xi\in\mathbb{R}^{n}, 
\end{equation}
in $\Omega$, where $C>0$ is universal, and $K$ depends on  $\operatorname{dist}(\Omega,\partial\Omega_{0})$, and either $\|f^{1/(k-1)}\|_{C^{1,1}(\overline{\Omega_{0}})}$ or  $\|f^{3/(2k-2)}\|_{C^{1,1}(\overline{\Omega_{0}})}$, but is independent of $\inf_{\Omega_{0}} f$. 
\end{lemma}
\begin{proof}
\textbf{Step 1.}
It follows from \eqref{eqn2.3} that  
\begin{align*}
\operatorname{tr}F^{ij}\big[D^{2}u\big]&=\frac{1}{k}\sigma_{k}^{\frac{1-k}{k}}\big[D^{2}u\big]\sum_{i=1}^{n}\sigma_{k}^{ii}\big[D^{2}u\big]\\
&=
\frac{n-k+1}{k}\sigma_{k}^{\frac{1-k}{k}}\big[D^{2}u\big]\sigma_{k-1}\big[D^{2}u\big].
\end{align*}
By \eqref{eqn2.1} and \eqref{eqn2.2}, we get \eqref{eqn2.11}. 

\par
\vspace{2mm}
\textbf{Step 2.} 
It follows from \eqref{eqn2.3} and the Euler theorem for homogeneous functions that
\[
\operatorname{tr}\left(GD^{2}u\right)=\frac{1}{n-k+1}\frac{\sigma_{k}^{ij}\big[D^{2}u\big]u_{ij}}{\sigma_{k-1}\big[D^{2}u\big]}=\frac{k}{n-k+1}\frac{\sigma_{k}\big[D^{2}u\big]}{\sigma_{k-1}\big[D^{2}u\big]}\geq0. 
\]
By \eqref{eqn2.1} and \eqref{eqn2.5}, one derives 
\[
\operatorname{tr}\left(GD^{2}u\right)\leq C\sigma_{k}^{\frac{1}{k}}\big[D^{2}u\big]=Cf^{\frac{1}{k}}\leq K,
\]
where $C$ is universal, and $K$ depends on $\|f\|_{C^{0}(\overline{\Omega})}$. 

\par
\vspace{2mm}
\textbf{Step 3.} 
The combination of \eqref{eqn2.11} and \eqref{eqn2.1} yields that 
\[
\left|\operatorname{tr}\left(GD^{2}u_{e}\right)\right|\leq C\sigma_{1}^{\frac{-1}{k-1}}\big[D^{2}u\big]g^{\frac{p}{k(k-1)}}\left|F^{ij}\big[D^{2}u\big]u_{eij}\right|
\] 
for all $e\in\mathbb{S}^{n-1}$.
Then it follows from \eqref{eqn2.7}, \eqref{eqn7.7}, and \eqref{eqn7.1} that 
\[
\left|\operatorname{tr}\left(GD^{2}u_{e}\right)\right|\leq C\sigma_{1}^{\frac{-1}{k-1}}\big[D^{2}u\big]\frac{p}{k}g^{\frac{p}{k-1}-\frac{1}{2}}\frac{\left|\nabla g\right|}{\sqrt{g}}\leq K\sigma_{1}^{\frac{-1}{k-1}}\big[D^{2}u\big] 
\]
for all $e\in\mathbb{S}^{n-1}$, where the constant $K$ depends on $\operatorname{dist}(\Omega,\partial\Omega_{0})$ and either $\|f^{1/(k-1)}\|_{C^{1,1}(\overline{\Omega_{0}})}$ or  $\|f^{3/(2k-2)}\|_{C^{1,1}(\overline{\Omega_{0}})}$. 
Thus \eqref{eqn2.13} is valid. 
\end{proof}

\begin{lemma}\label{thm2.5}
Let $u\in C^{3,1}(\overline{\Omega})$ be a $k$-admissible solution of equation \eqref{eqn1.1}, and $\inf_{\Omega_{0}}f>0$ where  $\Omega\Subset\Omega_{0}$.  
Let either $f^{1/(k-1)}\in C^{1,1}(\overline{\Omega_{0}})$ hold or  $f^{3/(2k-2)}\in C^{2,1}(\overline{\Omega_{0}})$ and $k\geq5$ hold.  
Then 
\begin{equation}\label{eqn2.14}
\operatorname{tr}\left(GD^{2}u_{\xi\xi}\right)\geq -K\left|\xi\right|^{2}\sigma_{1}^{\frac{-1}{k-1}}\big[D^{2}u\big],\quad \forall \xi\in\mathbb{R}^{n}, 
\end{equation}
in $\Omega$, where $K$ depends on  $\operatorname{dist}(\Omega,\partial\Omega_{0})$, and either $\|f^{1/(k-1)}\|_{C^{1,1}(\overline{\Omega_{0}})}$ or  $\|f^{3/(2k-2)}\|_{C^{2,1}(\overline{\Omega_{0}})}$, but is independent of $\inf_{\Omega_{0}}f$. 
\end{lemma}
\begin{proof} 
When $f^{1/(k-1)}\in C^{1,1}(\overline{\Omega_{0}})$, it follows from \eqref{eqn2.16}, \eqref{eqn2.11}, and \eqref{eqn7.1} that 
\begin{align*}
\operatorname{tr}\left(GD^{2}u_{ee}\right)&\geq-C\sigma_{1}^{\frac{-1}{k-1}}\big[D^{2}u\big]g^{\frac{p}{k(k-1)}} \frac{k-1}{k}g^{-\frac{1}{k}}\left|\partial_{ee}g-\frac{1}{k}\frac{\left|\partial_{e}g\right|^{2}}{g}\right|\\
&\geq-K_{1}\sigma_{1}^{\frac{-1}{k-1}}\big[D^{2}u\big]  
\end{align*}
for all $e\in\mathbb{S}^{n-1}$, where $C$ is universal, and $K_{1}$
depends on $\operatorname{dist}(\Omega,\partial\Omega_{0})$ and  $\|f^{1/(k-1)}\|_{C^{1,1}(\overline{\Omega_{0}})}$. 
When $f^{3/(2k-2)}\in C^{2,1}(\overline{\Omega_{0}})$ and $k\geq5$, which implies $(k+2)/3k<1/2$, it follows from \eqref{eqn2.16} and \eqref{eqn7.2} that 
\begin{align*}
\operatorname{tr}\left(GD^{2}u_{ee}\right)&\geq\frac{1}{\operatorname{tr}F^{ij}\big[D^{2}u\big]}\frac{p}{k}g^{-\frac{2}{3k}-\frac{1}{3}}\left(\partial_{ee}g-\frac{k+2}{3k}\frac{\left|\partial_{e}g\right|^{2}}{g}\right)\\
&\geq-\frac{1}{\operatorname{tr}F^{ij}\big[D^{2}u\big]}K_{2}g^{-\frac{2}{3k}} 
\end{align*}
for all $e\in\mathbb{S}^{n-1}$, where $K_{2}$ depends on $\operatorname{dist}(\Omega,\partial\Omega_{0})$ and $\|f^{3/(2k-2)}\|_{C^{2,1}(\overline{\Omega_{0}})}$. 
Combining this and \eqref{eqn2.11} yields 
\[
\operatorname{tr}\left(GD^{2}u_{ee}\right)\geq  -C\sigma_{1}^{\frac{-1}{k-1}}\big[D^{2}u\big]g^{\frac{p}{k(k-1)}}K_{2}g^{-\frac{2}{3k}}=-CK_{2}\sigma_{1}^{\frac{-1}{k-1}}\big[D^{2}u\big], 
\]
where $C$ is universal.  
Therefore \eqref{eqn2.14} holds.
\end{proof}

\par
\vspace{2mm}
Below, we construct an auxiliary function which depends only on $\Omega$. 

\begin{lemma}\label{thm2.4}
Let $\Omega\subset\mathbb{R}^n$ be a domain of class $C^{3,1}$ with $\partial\Omega$ strictly $(k-1)$-convex. 
Then there exists a function $\psi\in C^{3,1}(\overline{\Omega})$ satisfying 
\begin{equation}\label{eqn2.15}
\begin{cases}
\operatorname{tr}\left(GD^{2}\psi\right)\leq-1&\quad\text{in}\ \Omega,\\
\psi>0&\quad\text{in}\ \Omega,\\
\psi=0&\quad\text{on}\ \partial\Omega,\\
\frac{D\psi(x)}{|D\psi(x)|}=\nu(x)&\quad\text{on}\ \partial\Omega,\\
|D\psi|\geq1&\quad\text{on}\ \partial\Omega. 
\end{cases}
\end{equation} 
\end{lemma}
\begin{proof}
Let $\psi^{*}$ be the function constructed in Section 2 of Caffarelli-Nirenberg-Spruck \cite{CNS1985}, which satisfies 
\begin{equation*}
\begin{cases}
\lambda\big(D^{2}\psi^{*}(x)\big)\in\Gamma_{k},&\quad\forall x\in\overline{\Omega},\\
\psi^{*}<0&\quad\text{in}\ \Omega,\\
\psi^{*}=0&\quad\text{on}\ \partial\Omega,\\
\frac{D\psi^{*}(x)}{|D\psi^{*}(x)|}=-\nu(x)&\quad\text{on}\ \partial\Omega,\\
\left|D\psi^{*}\right|\geq\varepsilon_{0}&\quad\text{on}\ \partial\Omega,
\end{cases}
\end{equation*}
for some constant $\varepsilon_{0}>0$ depending on $\Omega$.
By the compactness of $\overline{\Omega}$, there exists a small constant $\varepsilon_{1}>0$ such that 
\[
\left\{\lambda\big(D^{2}\psi^{*}(x)\big):\ x\in\overline{\Omega}\right\}-\varepsilon_{1}\lambda(\mathbf{I}_{n})\Subset\Gamma_{k}. 
\]
Using the concavity of $F$, we derive
\begin{align*}
F^{ij}\big[D^{2}u\big]\left(K\big(\psi^{*}_{ij}-\varepsilon_{1}\delta_{ij}\big)-u_{ij}\right)&\geq F\big[K\big(D^{2}\psi^{*}-\varepsilon_{1}\mathbf{I}_{n}\big)\big]-F\big[D^{2}u\big]\\
&=K\sigma_{k}^{\frac{1}{k}}\big[D^{2}\psi^{*}-\varepsilon_{1}\mathbf{I}_{n}\big]-f^{\frac{1}{k}}\\
&>0,
\end{align*}
where $K$ is sufficiently large.
Consequently, by $\eqref{eqn2.12}$, we have 
\[
F^{ij}\big[D^{2}u\big]\psi^{*}_{ij}\geq \varepsilon_{1}\operatorname{tr}F^{ij}\big[D^{2}u\big]+\frac{1}{K}F^{ij}\big[D^{2}u\big]u_{ij}\geq \varepsilon_{1}\operatorname{tr}F^{ij}\big[D^{2}u\big]. 
\]
Thus  
\[
\operatorname{tr}\left(GD^{2}\psi^{*}\right)\geq\varepsilon_{1}.
\]
Choosing $\psi=-t\psi^{*}$ with $t>0$ sufficiently large, we obtain \eqref{eqn2.15}.
\end{proof}

\section{Reduction to the boundary second-order normal estimate}

It follows from Caffarelli-Nirenberg-Spruck \cite{CNS1985} that 
\begin{equation}\label{eqn7.23}
\sup_{\overline{\Omega}}|u|+\sup_{\partial\Omega}|Du|\leq C,
\end{equation}
where $C$ depends on $\Omega$, $\|f\|_{C^{0}(\overline{\Omega})}$, and $\|\varphi\|_{C^{3}(\partial\Omega)}$, but is independent of $\inf_{\Omega_{0}}f$.   

\begin{lemma}\label{thm7.1}
Let $u\in C^{3,1}(\overline{\Omega})$ be a $k$-admissible solution of equation \eqref{eqn1.1}, 
$\partial\Omega\in C^{3,1}$ be strictly $(k-1)$-convex, 
$\varphi\in C^{3,1}(\partial\Omega)$, 
$\inf_{\Omega_{0}}f>0$, and $f^{3/(2k-2)}\in C^{1,1}(\overline{\Omega_{0}})$, where $\Omega\Subset\Omega_{0}$.      
Then  
\[
\sup_{\Omega}|Du|\leq \sup_{\partial\Omega}|Du|+C,
\]
where $C$ depends on $\Omega$, $\operatorname{dist}(\Omega,\partial\Omega_{0})$, and $\|f^{3/(2k-2)}\|_{C^{1,1}(\overline{\Omega_{0}})}$, but is independent of $\inf_{\Omega_{0}}f$. 
\end{lemma}
\begin{proof}
Without loss of generality, let $\Omega\subset\mathbb{R}^{n}\setminus B_{1}$.  
Suppose
\[
\sup_{x\in\overline{\Omega},e\in\mathbb{S}^{n-1}}\left(u_{e}+K|x|^{2}\right)
\] 
is attained at an interior point $x_{0}\in\Omega$ in the direction $e=e_{1}$, where $K>0$ is a constant to be determined. 
Then at the point $x_{0}$, 
\[
u_{1j}+2Kx_{j}=0,\quad \forall1\leq j\leq n, 
\]
and 
\begin{equation}\label{eqn7.13}
F^{ij}\big[D^{2}u\big]\left(u_{1ij}+2K\delta_{ij}\right)\leq0.
\end{equation}
Since $\lambda(D^{2}u)\in\Gamma_{2}$, we derive at the point $x_{0}$, 
\begin{equation}\label{eqn7.14}
0\leq \sigma_{1}^{2}\big[D^{2}u\big]-\sum_{j=1}^{n}u_{1j}^{2}\leq\sigma_{1}^{2}\big[D^{2}u\big]-4K^{2}. 
\end{equation}
By \eqref{eqn2.7} and \eqref{eqn7.1}, we obtain  
\begin{equation}\label{eqn4.27}
F^{ij}\big[D^{2}u\big]u_{1ij}=\frac{p}{k}\frac{\partial_{1}g}{\sqrt{g}}f^{\frac{1}{4(k-1)}-\frac{1}{k(k-1)}}\geq-K_{1}f^{\frac{-1}{k(k-1)}} 
\end{equation}
in $\Omega$, where $K_{1}$ depends on $\operatorname{dist}(\Omega,\partial\Omega_{0})$ and  $\|f^{3/(2k-2)}\|_{C^{1,1}(\overline{\Omega_{0}})}$. 

It follows from \eqref{eqn7.13}, \eqref{eqn4.27}, \eqref{eqn2.11}, and \eqref{eqn7.14} that at the point $x_{0}$,
\begin{align*}
0&\geq\frac{2K}{C}\sigma_{1}^{\frac{1}{k-1}}\big[D^{2}u\big]f^{\frac{-1}{k(k-1)}}-K_{1}f^{\frac{-1}{k(k-1)}}\\
&\geq\left(\frac{2K}{C}(2K)^{\frac{1}{k-1}}-K_{1}\right)f^{\frac{-1}{k(k-1)}},
\end{align*}
where $C>0$ is universal. 
Taking $K$ large enough, we get a contradiction. 
The proof of Lemma \ref{thm7.1} is complete.  
\end{proof} 

\begin{lemma}\label{thm7.2}
Let $u\in C^{3,1}(\overline{\Omega})$ be a $k$-admissible solution of equation \eqref{eqn1.1}, 
$\partial\Omega\in C^{3,1}$ be strictly $(k-1)$-convex, 
$\varphi\in C^{3,1}(\partial\Omega)$, 
$\inf_{\Omega_{0}}f>0$, and $f^{3/(2k-2)}\in C^{1,1}(\overline{\Omega_{0}})$, where $\Omega\Subset\Omega_{0}$.  
Then 
\[
\sup_{\partial\Omega}\left(|u_{\nu\eta}|+|u_{\eta\eta}|\right)\leq C, 
\]
where $\eta$ is any unit tangential vector; and $C$ depends on $\Omega$, $\|\varphi\|_{C^{3}(\partial\Omega)}$, $\|u\|_{C^{1}(\overline{\Omega})}$, $\operatorname{dist}(\Omega,\partial\Omega_{0})$, and $\|f^{3/(2k-2)}\|_{C^{1,1}(\overline{\Omega_{0}})}$, but is independent of $\inf_{\Omega_{0}}f$.  
\end{lemma}
\begin{proof}
Let the origin be a boundary point satisfying $e_{n}=\nu(0)$, and $\partial\Omega$ be represented by
\begin{equation}\label{eqn7.15}
x_{n}=\rho\big(x'\big)
\end{equation}
on $\partial\Omega\cap\{|x'|\leq r_{0}\}$ for some $C^{3,1}$ function $\rho$ and constant $r_{0}>0$.
It is clear that $\nabla_{x'}\rho(0)=0$. 
Let $\varphi(0)=0$ and $\nabla\varphi(0)=0$ by subtracting an affine function. 
Differentiating $u(x',\rho(x'))=\varphi(x',\rho(x'))$ on $\partial\Omega$ twice, one gets 
\[
u_{ij}(0)=\varphi_{ij}(0)-u_{n}(0)\rho_{ij}(0),\quad \forall1\leq i,\ j\leq n-1.
\] 
Thus the second-order tangential estimate is obtained.

For any $x\in\partial\Omega\cap\{|x'|\leq r_{0}\}$, define an operator 
\[
T_{\alpha}:=\partial_{\alpha}+\sum_{\beta=1}^{n-1}\rho_{\alpha\beta}(0)(x_{\beta}\partial_{n}-x_{n}\partial_{\beta}),\quad \forall1\leq\alpha\leq n-1. 
\]  
It is standard that 
\begin{equation}\label{eqn7.17}
T_{\alpha}(u-\varphi)=O\Big(\big|x'\big|^{2}\Big),\quad \forall1\leq\alpha\leq n-1, 
\end{equation}
on $\partial\Omega\cap\{|x'|\leq r_{0}\}$.
It follows from \eqref{eqn2.7} and \eqref{eqn7.1} that 
\begin{align}\label{eqn7.21}
\left|F^{ij}[D^{2}u]\left(T_{\alpha}(u-\varphi)\right)_{ij}\right|&\leq\left|T_{\alpha}(g^{p})\right|+C_{1}\operatorname{tr}F^{ij}\big[D^{2}u\big]\nonumber\\
&\leq C_{2}f^{\frac{-1}{k(k-1)}}+C_{1}\operatorname{tr}F^{ij}\big[D^{2}u\big]
\end{align}
in $\Omega$, where $C_{1}$ depends on $\|\rho\|_{C^{3,1}}$, $\|\varphi\|_{C^{3}(\partial\Omega)}$; $C_{2}$ depends on $\|\rho\|_{C^{3,1}}$,  $\operatorname{dist}(\Omega,\partial\Omega_{0})$, and  $\|f^{3/(2k-2)}\|_{C^{1,1}(\overline{\Omega_{0}})}$.

Define  
\[
\omega_{r}:=\left\{x\in\Omega:\ \rho\big(x'\big)<x_{n}<\rho\big(x'\big)+r^{4},\ \big|x'\big|<r\right\}
\]
for all $r\leq r_{0}$.
Define a function 
\[
v:=\big(x_{n}-\rho\big(x'\big)\big)^{2}-\beta_{1}\big(x_{n}-\rho\big(x'\big)\big)-\beta_{2}\big|x'\big|^{2}\quad \text{in}\ \omega_{r_{0}}, 
\]
where positive constants $\beta_{1}$ and $\beta_{2}$ are to be determined.  
It holds that 
\[
D^{2}v=D^{2}\big[x_{n}^{2}+\beta_{1}\rho\big(x'\big)\big]+D^{2}\big[\rho^{2}\big(x'\big)-2x_{n}\rho\big(x'\big)\big]-\beta_{2}D^{2}\Big[\big|x'\big|^{2}\Big]. 
\]
Since $\partial\Omega\in C^{3,1}$ is strictly $(k-1)$-convex, one can fix $\beta_{1}$  sufficiently small such that $\lambda(D^{2}[x_{n}^{2}+\beta_{1}\rho](0))$ belongs to a compact subset of $\Gamma_{k}$ independent of $0\in\partial\Omega$.  
By 
\[
|x|\rightarrow0,\ \rho\big(x'\big)\rightarrow0,\ \nabla_{x'}\rho\big(x'\big)\rightarrow0,\ D^{2}_{x'x'}\rho\big(x'\big)\rightarrow D^{2}_{x'x'}\rho(0)\quad \text{as}\ r\rightarrow0,
\]
there exist small positive constants $r_{1}\leq\min\{r_{0},\beta_{1}\}$, $\beta_{2}$, and $\delta_{1}$ such that 
\[
\left\{\lambda\big(D^{2}v-\delta_{1}\mathbf{I}_{n}\big):\ x\in\omega_{r_{1}}\right\}\ \ \text{is a perturbation of}\ \  \lambda\big(D^{2}\big[x_{n}^{2}+\beta_{1}\rho\big](0)\big),
\]
and is thus contained in a compact subset of $\Gamma_{k}$.
We emphasize that $\beta_{1}$, $r_{1}$, $\beta_{2}$, $\delta_{1}$ all depend only on $\Omega$.  
Using the concavity of $F$,   
\begin{align*}
F^{ij}\big[D^{2}u\big]\left(C_{3}(v_{ij}-\delta_{1}\delta_{ij})-u_{ij}\right)\geq C_{3}\sigma_{k}^{\frac{1}{k}}\big[D^{2}v-\delta_{1}\mathbf{I}_{n}\big]-f^{\frac{1}{k}}>0
\end{align*}
in $\omega_{r_{1}}$, where $C_{3}$ is sufficiently large. 
Then by virtue of \eqref{eqn2.12} and \eqref{eqn2.11}, we derive 
\begin{equation}\label{eqn7.22}
F^{ij}\big[D^{2}u\big]v_{ij}\geq \delta_{1}\operatorname{tr}F^{ij}\big[D^{2}u\big]\geq\frac{\delta_{1}}{C}\sigma_{1}^{\frac{1}{k-1}}\big[D^{2}u\big]f^{\frac{-1}{k(k-1)}}\quad \text{in}\ \omega_{r_{1}},
\end{equation}
where $C>0$ is universal.  

We claim that 
\begin{equation}\label{eqn7.16}
\pm T_{\alpha}(u-\varphi)\leq -Kv\quad \text{in}\ \overline{\omega_{r_{1}}},\quad \forall1\leq\alpha\leq n-1,
\end{equation}
where the constant $K>0$ is to be determined. 

\par
\vspace{2mm}
\textbf{Step 1:} \emph{verification of \eqref{eqn7.16} on $\partial\omega_{r_{1}}$.}
The boundary $\partial\omega_{r_{1}}$ consists of three parts: $\partial_{1}\omega_{r_{1}}\cup\partial_{2}\omega_{r_{1}}\cup\partial_{3}\omega_{r_{1}}$, where $\partial_{1}\omega_{r_{1}}$ and $\partial_{2}\omega_{r_{1}}$ are respectively the graph parts of $\rho$ and $\rho+r_{1}^{4}$, and $\partial_{3}\omega_{r_{1}}$ is the boundary part on $\{|x'|=r_{1}\}$. 
On $\partial_{1}\omega_{r_{1}}$, using \eqref{eqn7.17} one gets 
\[
-Kv=K\beta_{2}\big|x'\big|^{2}\geq \pm T_{\alpha}(u-\varphi)
\]
by taking $K$ large, which depends on $\Omega$, $\|\varphi\|_{C^{1}(\partial\Omega)}$, and $\|u\|_{C^{1}(\overline{\Omega})}$.
It is clear that 
\[
\pm T_{\alpha}(u-\varphi)\leq C_{4},
\]
where $C_{4}$ depends on $\|\rho\|_{C^{2}}$,  $\|\varphi\|_{C^{1}(\partial\Omega)}$, and $\|u\|_{C^{1}(\overline{\Omega})}$. 
It holds that
\[
-Kv\geq K\left(\beta_{1}r_{1}^{4}-r_{1}^{8}\right)\geq\frac{\beta_{1}}{2}Kr_{1}^{4}\quad \text{on}\ \partial_{2}\omega_{r_{1}}, 
\]
and
\[
-Kv\geq K\beta_{2}r_{1}^{4}\quad \text{on}\ \partial_{3}\omega_{r_{1}}.
\]
Then we obtain $-Kv\geq \pm T_{\alpha}(u-\varphi)$ on $\partial_{2}\omega_{r_{1}}\cup\partial_{3}\omega_{r_{1}}$ by taking $K$ large, which depends on $\Omega$, $\|\varphi\|_{C^{1}(\partial\Omega)}$, and $\|u\|_{C^{1}(\overline{\Omega})}$. 

\par
\vspace{2mm}
\textbf{Step 2.} 
Suppose, for contradiction, that \eqref{eqn7.16} does not hold. 
Then the maximum  
\[
\max_{\overline{\omega_{r_{1}}}}\left(\pm T_{\alpha}(u-\varphi)+Kv\right)
\]
is attained at an interior point $x_{0}\in\omega_{r_{1}}$. 
At the point $x_{0}$, we have 
\begin{equation}\label{eqn7.18}
0=\left(\pm T_{\alpha}(u-\varphi)+Kv\right)_{n}=\left(\pm T_{\alpha}u\right)_{n}-\left[-Kv_{n}+\left(\pm T_{\alpha}\varphi\right)_{n}\right]
\end{equation}
and
\begin{equation}\label{eqn7.19}
F^{ij}\big[D^{2}u\big]\left(\pm T_{\alpha}(u-\varphi)+Kv\right)_{ij}\leq0.
\end{equation}
It follows from \eqref{eqn7.18} that 
\begin{align*}
\left|\left(T_{\alpha}u\right)_{n}\right|\geq-Kv_{n}-\left|\left(T_{\alpha}\varphi\right)_{n}\right|&=K\left(\beta_{1}-2x_{n}+2\rho\big(x'\big)\right)-\left|\left(T_{\alpha}\varphi\right)_{n}\right|\\
&\geq \frac{\beta_{1}}{2}K-\left|\left(T_{\alpha}\varphi\right)_{n}\right|
\end{align*}
at the point $x_{0}$. 
By taking $K$ large, which depends on $\Omega$ and  $\|\varphi\|_{C^{2}(\partial\Omega)}$, we get 
\[
\big|D^{2}u(x_{0})\big|\geq1, 
\]
which together with $\lambda(D^{2}u)\in\Gamma_{2}$ yields \
\begin{equation}\label{eqn7.20}
\sigma_{1}\big[D^{2}u(x_{0})\big]\geq\big|D^{2}u(x_{0})\big|\geq1.  
\end{equation}
Using \eqref{eqn7.19}, \eqref{eqn7.21}, \eqref{eqn7.22}, and \eqref{eqn7.20}, we obtain at the point $x_{0}$, 
\[
-\left(C_{2}f^{\frac{-1}{k(k-1)}}+C_{1}\operatorname{tr}F^{ij}\big[D^{2}u\big]\right)+K\left(\frac{\delta_{1}}{2}\operatorname{tr}F^{ij}\big[D^{2}u\big]+\frac{\delta_{1}}{2C}f^{\frac{-1}{k(k-1)}}\right)\leq0, 
\]
which leads to a contradiction by taking $K$ large, which depends on $\Omega$, $\|\varphi\|_{C^{3}(\partial\Omega)}$, $\operatorname{dist}(\Omega,\partial\Omega_{0})$, and  $\|f^{3/(2k-2)}\|_{C^{1,1}(\overline{\Omega_{0}})}$.

Therefore, claim \eqref{eqn7.16} holds with $K$ depending on $\Omega$, $\|\varphi\|_{C^{3}(\partial\Omega)}$, $\|u\|_{C^{1}(\overline{\Omega})}$,  $\operatorname{dist}(\Omega,\partial\Omega_{0})$, and  $\|f^{3/(2k-2)}\|_{C^{1,1}(\overline{\Omega_{0}})}$.
It follows immediately that 
\[
\pm\left(T_{\alpha}(u-\varphi)\right)_{n}(0)\leq -Kv_{n}(0)=K\beta_{1}
\]
for all $1\leq\alpha\leq n-1$.
The proof of Lemma \ref{thm7.2} is complete. 
\end{proof}
\par
\vspace{2mm}

Inspired by Guan \cite{Guan1997}, Guan-Trudinger-Wang \cite{Guan1999}, and Jiao-Wang \cite{Jiao2024}, we reduce the estimate of second-order derivatives to the boundary. 

\begin{lemma}\label{thm1.2}
Let $u\in C^{3,1}(\overline{\Omega})$ be a $k$-admissible solution of equation \eqref{eqn1.1},  
$\inf_{\Omega_{0}}f>0$, and $f^{3/(2k-2)}\in C^{2,1}(\overline{\Omega_{0}})$, where $\Omega\Subset\Omega_{0}$.  
Then 
\[
\sup_{\Omega}\left|D^{2}u\right|\leq\sup_{\partial\Omega}\left|D^{2}u\right|+C,
\]
where $C$ depends on $\Omega$, $\operatorname{dist}(\Omega,\partial\Omega_{0})$, and $\|f^{3/(2k-2)}\|_{C^{2,1}(\overline{\Omega_{0}})}$, but is independent of $\inf_{\Omega_{0}}f$.  
\end{lemma}
\begin{proof}
Define a function 
\[
v:=\Delta u+|x|^{2}.  
\]
Suppose $v$ attains its maximum at an interior point $x_{0}\in\Omega$, for otherwise we are done. 
Then at the point $x_{0}$, we have 
\begin{equation}\label{eqn6.14}
\partial_{x_{i}}(\Delta u)+2x_{i}=0,\quad \forall1\leq i\leq n,
\end{equation}
and  
\[
0\geq\sigma_{k}^{ij}\big[D^{2}u\big]v_{ij},
\]
which together with \eqref{eqn2.3} and \eqref{eqn2.2} implies 
\begin{equation}\label{eqn6.15}
0\geq\sigma_{k}^{ij}\big[D^{2}u\big](\Delta u)_{ij}+C\sigma_{1}^{\frac{1}{k-1}}\big[D^{2}u\big]\sigma_{k}^{\frac{k-2}{k-1}}\big[D^{2}u\big], 
\end{equation}
where $C>0$ is universal. 
Differentiating equation \eqref{eqn1.1}, we derive 
\[
\sigma_{k}^{ij}\big[D^{2}u\big](\Delta u)_{ij}+\sum_{l=1}^{n}\sigma_{k}^{ij,st}\big[D^{2}u\big]u_{ijl}u_{stl}=\Delta (g^{p}).
\]
Combining this equality with \eqref{eqn6.15}, we obtain at the point $x_{0}$, 
\begin{align}\label{eqn6.13}
0&\geq C(\Delta u)^{\frac{1}{k-1}}f^{1-\frac{1}{k-1}}+pg^{p-1}\sum_{l=1}^{n}\left[(p-1)\frac{|\partial_{x_{l}}g|^{2}}{g}+\partial_{x_{l}x_{l}}g\right]\\
&\hspace{4.5mm}-\sum_{l=1}^{n}\sigma_{k}^{ij,st}\big[D^{2}u\big]u_{ijl}u_{stl}.\nonumber 
\end{align}
It follows from Lemma 3.2 of Guan-Li-Li \cite{Guan2012} that computing 
\[
0\geq\left(\left(\frac{\sigma_{k}}{\sigma_{1}}\right)^{1/(k-1)}\right)^{ij,st}u_{ije}u_{ste} 
\]
yields 
\begin{align*}
\frac{\sigma_{k}^{ij,st}u_{ije}u_{ste}}{\sigma_{k}}&\leq \frac{k-2}{k-1}\left(\frac{\sigma_{k}^{ij}u_{ije}}{\sigma_{k}}\right)^{2}\\
&\hspace{4.5mm}+\frac{2}{k-1}\frac{\sigma_{k}^{ij}u_{ije}}{\sigma_{k}}\frac{\sigma_{1}^{ij}u_{ije}}{\sigma_{1}}-\frac{k}{k-1}\left(\frac{\sigma_{1}^{ij}u_{ije}}{\sigma_{1}}\right)^{2}
\end{align*}
for all $e\in\mathbb{S}^{n-1}$, where operators take values at arbitrary matrix $A$ with $\lambda(A)\in\Gamma_{k}$.
Thus by \eqref{eqn6.14}, we infer at the point $x_{0}$, 
\begin{align}
-\sum_{l=1}^{n}\sigma_{k}^{ij,st}\big[D^{2}u\big]u_{ijl}u_{stl}&\geq -\sum_{l=1}^{n}\frac{k-2}{k-1}\frac{p^{2}g^{p-1}|\partial_{x_{l}}g|^{2}}{g} \label{eqn6.12}\\
&\hspace{4.5mm}-\sum_{l=1}^{n}\left[\frac{2}{k-1}\frac{pg^{p-1}\partial_{x_{l}}g2x_{l}}{\Delta u}-\frac{kf}{k-1}\left(\frac{2x_{l}}{\Delta u}\right)^{2}\right]\nonumber. 
\end{align}
Using \eqref{eqn6.13}, \eqref{eqn6.12}, \eqref{eqn7.2}, and \eqref{eqn7.1}, we get at the point $x_{0}$,  
\begin{align*}
0&\geq C(\Delta u)^{\frac{1}{k-1}}f^{1-\frac{1}{k-1}}+pg^{p-1}\sum_{l=1}^{n}\left[-\frac{1}{3}\frac{|\partial_{x_{l}}g|^{2}}{g}+\partial_{x_{l}x_{l}}g\right]\\
&\hspace{4.5mm}-\sum_{l=1}^{n}\left[\frac{2}{k-1}\frac{pg^{p-1}\partial_{x_{l}}g2x_{l}}{\Delta u}-\frac{kf}{k-1}\left(\frac{2x_{l}}{\Delta u}\right)^{2}\right]\\
&\geq C(\Delta u)^{\frac{1}{k-1}}f^{1-\frac{1}{k-1}}-K_{1}f^{1-\frac{1}{k-1}}\\
&\hspace{4.5mm}-f^{1-\frac{1}{k-1}}\sum_{l=1}^{n}\left[\frac{2}{k-1}\frac{K_{2}f^{\frac{1}{4(k-1)}}2|x_{l}|}{\Delta u}-\frac{kf^{\frac{1}{k-1}}}{k-1}\left(\frac{2x_{l}}{\Delta u}\right)^{2}\right], 
\end{align*}
where $K_{1}$ and $K_{2}$ depend on $\operatorname{dist}(\Omega,\partial\Omega_{0})$ and $\|f^{3/(2k-2)}\|_{C^{2,1}(\overline{\Omega_{0}})}$. 
Therefore, 
\[
\Delta u(x_{0})\leq K, 
\]
where $K$ depends on $\Omega$, $\operatorname{dist}(\Omega,\partial\Omega_{0})$, and $\|f^{3/(2k-2)}\|_{C^{2,1}(\overline{\Omega_{0}})}$.  
The proof of Lemma \ref{thm1.2} is complete. 
\end{proof} 
\par
\vspace{2mm}

Combining \eqref{eqn7.23}, Lemmas \ref{thm7.1}, \ref{thm7.2} and \ref{thm1.2}, we reduce the $C^{2}$ estimate to the boundary second-order normal derivative estimate under the condition $f^{3/(2k-2)}\in C^{2,1}(\overline{\Omega_{0}})$ with $f\geq0$ in $\Omega_{0}$. 
Under the homogeneous boundary condition, Dong \cite{Dong2006} established the required boundary second-order normal derivative estimate.
Thus the $C^{2}$ estimate independent of $\inf_{\Omega_{0}}f$ is obtained.
Then the existence is established by an approximation argument.

\begin{theorem}\label{thm1.9}
Let $\partial\Omega\in C^{3}$ be strictly $(k-1)$-convex, 
$\varphi=0$, $f^{3/(2k-2)}\in C^{2,1}(\overline{\Omega_{0}})$, and $f\geq0$ in $\Omega_{0}$, where $\Omega\Subset\Omega_{0}$.     
Then there exists a unique $k$-admissible solution in $C^{1,1}(\overline{\Omega})$ for \eqref{eqn1.1}.
\end{theorem}

We emphasize here that results corresponding to Lemmas \ref{thm7.1}, \ref{thm7.2}, \ref{thm1.2}, and Theorem \ref{thm1.9} were established by Jiao-Wang \cite{Jiao2024} under the condition $f^{1/(k-1)}\in C^{1,1}(\overline{\Omega})$ with $f\geq0$ in $\Omega$.

\section{Monge-Amp{\`e}re equation}

In this section, we assume 
\[
f^{\frac{3}{2n-2}}\in C^{2,1}(\overline{\Omega}),\quad f\geq0\ \text{in}\ \Omega,
\]
and apply the method of Guan-Trudinger-Wang \cite{Guan1999} to establish the boundary second-order normal derivative estimate.
We emphasize here that for the Monge-Amp{\`e}re equation, Lemmas \ref{thm7.1} and \ref{thm7.2} hold under the weaker condition $\inf_{\Omega}f>0$ and $f^{3/(2n-2)}\in C^{1,1}(\overline{\Omega})$. 
Indeed, Lemma \ref{thm7.1} follows directly from the convexity of the admissible solution, and Lemma \ref{thm7.2} can be adapted by differentiating with respect to a tangential vector filed, with the same argument as in the proof of later Lemma \ref{thm6.4}.  

Let the origin be a boundary point satisfying $\nu(0)=e_{n}$. 
By an affine transformation of $x'$, let $\partial\Omega$ be represented by
\begin{equation}\label{eqn6.10}
x_{n}=\rho\big(x'\big)=\frac{1}{2}\big|x'\big|^{2}+\text{cubic of $x'$}+O\Big(\big|x'\big|^{4}\Big)
\end{equation}
on $\partial\Omega\cap\{|x'|\leq r_{0}\}$ for some $C^{3,1}$ function $\rho$ and constant $r_{0}>0$. 
Then rotate $x'$ such that $D^{2}_{x'x'}u(0)$ is a diagonal matrix with 
\[
0<u_{11}(0)\leq u_{22}(0)\leq\dots\leq u_{n-1,n-1}(0).
\] 
By subtracting an affine function, let  
\[
u(0)=0\quad\text{and} \quad\nabla u(0)=0,
\]
which yields by the convexity of $u$, 
\[
u(x)\geq0,\quad\forall x\in\overline{\Omega}. 
\]
Differentiating the boundary condition $u(x',\rho(x'))=\varphi(x',\rho(x'))$ yields 
\[
D^{2}_{x'x'}u(0)=D^{2}_{x'x'}\varphi(0)+D^{2}_{x'x'}\rho(0)\partial_{n}\varphi(0). 
\]
For convenience, denote
\[
b_{i}:=u_{ii}(0),\quad \forall 1\leq i\leq n-1. 
\]
By Taylor expansion with respect to $x'$ at the origin, we have 
\begin{equation}\label{eqn9.2}
\varphi\big(x',\rho\big(x'\big)\big)=\frac{1}{2}\sum_{i=1}^{n-1}b_{i}x_{i}^{2}+\text{cubic of $x'$}+O\Big(\big|x'\big|^{4}\Big)
\end{equation}
on $\partial\Omega\cap\{|x'|\leq r_{0}\}$. 

The following lemma is the inequality (3.5) in Guan-Trudinger-Wang \cite{Guan1999}. 

\begin{lemma}[\cite{Guan1999}]\label{thm6.2}
Let $h\in C^{3,1}(\mathbb{R})$ be a nonnegative function on $\mathbb{R}$ satisfying  
\[
h(t)=\alpha t^{2}+\beta t^{3}+R(t),\quad \forall t\in\mathbb{R}, 
\]
where $\alpha\geq0$, $\beta\in\mathbb{R}$, and the remainder term satisfies  $|R(t)|\leq\gamma|t|^{4}$ for some $\gamma>0$. 
Then
\[
|\beta|\leq(1+\gamma)\sqrt{\alpha}. 
\]
\end{lemma}
\begin{proof}
Lemma \ref{thm6.2} is obtained by combining two facts $h(\sqrt{\alpha})\geq0$ and $h(-\sqrt{\alpha})\geq0$. 
\end{proof}

\subsection{Lower bounds for boundary second-order tangential derivatives}

\begin{lemma}\label{thm6.3}
There exists a constant $\gamma_{0}>0$ such that 
\[
\displaystyle\prod_{i=1}^{n-1}b_{i}\geq \gamma_{0}f(0),
\]
where $\gamma_{0}$ depends on $\Omega$, $\|\varphi\|_{C^{3,1}(\partial\Omega)}$, and $\|f^{3/(2n-2)}\|_{C^{2,1}(\overline{\Omega})}$, but is independent of $\inf_{\Omega}f$.  
\end{lemma}
\begin{proof}
\textbf{Step 1.} 
Define 
\[
M:=\frac{1}{b_{1}}. 
\]
Make a dilation $x\rightarrow y=T(x)$ defined by 
\begin{equation}\label{eqn9.1}
y_{i}=M_{i}x_{i}\quad \text{for}\ i=1,2,\dots,n, 
\end{equation}
where 
\[
M_{i}:=\sqrt{b_{i}}M\ \big(\geq\sqrt{M}\big)\ \ \text{for}\ i=1,2,\dots,n-1,\quad \text{and}\quad M_{n}:=M.  
\]
Without loss of generality, we assume $M\gg1$, which depends on the $C^{3,1}$-norm of $\partial\Omega$. 
Thus, the transformation $T$ can be regarded as a blow-up transformation near the origin.  
Define 
\[
v(y):=M^{2}u(x). 
\]
Then 
\[
\det D^{2}v=\frac{f(T^{-1}y)}{\prod_{i=1}^{n-1}b_{i}}=:\tilde{f}(y)\quad \text{in}\ \widetilde{\Omega}:=T(\Omega).  
\]
It follows from \eqref{eqn6.10} and \eqref{eqn9.1} that  $\partial\widetilde{\Omega}$ is represented by 
\begin{equation}\label{eqn6.16}
y_{n}=\tilde{\rho}\big(y'\big)=\frac{1}{2}y_{1}^{2}+\frac{1}{2}\sum_{i=2}^{n-1}d_{i}y_{i}^{2}+\frac{1}{\sqrt{M}}O\Big(\big|y'\big|^{3}\Big)
\end{equation}
on $\partial\widetilde{\Omega}\cap\{|y'|\leq\sqrt{M}r_{0}\}$, where  
\[
d_{i}\leq1\ \ \text{for}\ i=2,\dots,n-1,\quad \big|D^{3}\tilde{\rho}\big|\leq\frac{C_{1}}{\sqrt{M}},\quad \big|D^{4}\tilde{\rho}\big|\leq\frac{C_{1}}{M},  
\]
and $C_{1}$ depends only on $\partial\Omega$. 
Combining \eqref{eqn9.2} and \eqref{eqn9.1} yields   
\begin{equation}\label{eqn6.11}
v=\widetilde{\varphi}\big(y',\tilde{\rho}\big(y'\big)\big)=\frac{1}{2}\big|y'\big|^{2}+\text{cubic of $y'$}+O\Big(\big|y'\big|^{4}\Big)
\end{equation}
on $\partial\widetilde{\Omega}\cap\{|y'|\leq\sqrt{M}r_{0}\}$, where
\[
\big|D_{y}^{4}\widetilde{\varphi}\big|\leq\big|D_{x}^{4}\varphi\big|\leq C. 
\]   
By applying Lemma \ref{thm6.2}, we conclude that the coefficients of the third-order terms in \eqref{eqn6.11} are also bounded.
Hence both $\partial\widetilde{\Omega}$ and $\widetilde{\varphi}$ are $C^{3,1}$-smooth on $\partial\widetilde{\Omega}\cap\{|y'|\leq\sqrt{M}r_{0}\}$, and their $C^{3,1}$-norms are independent of $M$. 

\par
\vspace{2mm}
\textbf{Step 2.}
Define a set 
\[
\omega:=\left\{y\in\widetilde{\Omega}:\ y_{n}<1,\ |y_{i}|<1,\ \,\forall2\leq i\leq n-1\right\}.  
\] 
We analyze the geometry of $\omega$. 
Clearly, the set $\omega$ is convex. 
Looking at $\partial\omega\cap\{y_{n}=1\}$, by \eqref{eqn6.16} we have 
\[
1\geq\frac{1}{2}y_{1}^{2}+\frac{1}{\sqrt{M}}O\Big(\big|y'\big|^{3}\Big).  
\]
Since $M\gg1$, any point with $|y_{1}|=2$ must be outside of $\omega$.   
The connectedness of $\omega$ yields $|y_{1}|\leq 2$, so the set $\omega$ is contained in the ball $B_{\sqrt{n+3}}$.   
It is trivial that 
\begin{equation}\label{eqn9.5}
y_{\star}:=\left(0,\frac{1}{4n},\dots,\frac{1}{4n},\frac{1}{2}\right)\in\omega,  
\end{equation}
and 
\[
\operatorname{dist}(y_{\star},\partial\omega)\geq\min\left\{\frac{1}{2},1-\frac{1}{4n},\operatorname{dist}\Big(y_{\star},\partial\omega\cap\partial\widetilde{\Omega}\Big)\right\}. 
\]
Then, since $\tilde{\rho}(y')\leq|y'|^{2}$, the ball $B_{r_{\star}}(y_{\star})$ is contained in $\omega$, where 
\begin{equation}\label{eqn9.6}
r_{\star}:=\min\left\{\frac{1}{2},1-\frac{1}{4n},\min\left\{\frac{1}{4n},\frac{1}{2}-\frac{1}{4n}\right\}\right\}. 
\end{equation}
Using the convexity of $v$, we derive that the maximum of $v$ in $\overline{\omega}$ is attained on $\partial\omega\cap\partial\widetilde{\Omega}$. 
Thus, by \eqref{eqn6.11} we obtain  
\begin{equation}\label{eqn6.19}
0\leq v\leq\sup_{\partial\omega\cap\partial\widetilde{\Omega}}\widetilde{\varphi}\leq C_{2}\quad \text{in}\ \overline{\omega},  
\end{equation}
where $C_{2}$ depends only on $\|\varphi\|_{C^{3,1}(\partial\Omega)}$.  

\par
\vspace{2mm}
\textbf{Step 3.}
We claim that there exists a constant $K>0$, independent of $M$, such that for any ball $B_{\frac{r_{\star}}{4n}}(y_{*})$ contained in $B_{r_{\star}}(y_{\star})$,
\begin{equation}\label{eqn6.18}
\inf_{B_{\frac{r_{\star}}{4n}}(y_{*})}\tilde{f}<K. 
\end{equation}
Suppose, for contradiction, that \eqref{eqn6.18} does not hold.
Then 
\[
\det D^{2}v\geq K\quad \text{in}\  B_{\frac{r_{\star}}{4n}}(y_{*}). 
\]
Since \eqref{eqn6.19} the function 
\[
\frac{K^{1/n}}{2}\left((y-y_{*})^{2}-\frac{r_{\star}}{4n}\right)+C_{2}
\]
is an upper barrier of $v$ in $B_{\frac{r_{\star}}{4n}}(y_{*})$. 
Thus 
\[
v(y_{*})\leq-K^{\frac{1}{n}}\frac{r_{\star}}{8n}+C_{2}<0
\]
by taking $K$ large, which depends only on $\|\varphi\|_{C^{3,1}(\partial\Omega)}$. 
But this contradicts with $v\geq0$.  
Therefore, we obtain claim \eqref{eqn6.18}. 

Define 
\[
\tilde{g}:=\tilde{f}^{\frac{3}{2n-2}}\quad \text{in}\ \widetilde{\Omega}. 
\]
We have 
\begin{align}\label{eqn6.17}
\left\|D_{y}^{3}\tilde{g}\right\|_{L^{\infty}(\widetilde{\Omega})}&\leq \left(\prod_{i=1}^{n-1}b_{i}\right)^{\frac{-3}{2n-2}}b_{1}^{\frac{3}{2}}\left\|D_{x}^{3}\left(f^{\frac{3}{2n-2}}\right)\right\|_{L^{\infty}(\Omega)}\nonumber\\
&\leq\left\|f^{\frac{3}{2n-2}}\right\|_{C^{2,1}(\overline{\Omega})}\leq C. 
\end{align}
From \eqref{eqn6.18}, there is a point $\tilde{y}\in B_{\frac{r_{\star}}{4n}}(y_{\star})$ satisfying $\tilde{f}(\tilde{y})\leq K$. 
The fact 
\[
0\leq \tilde{g}\Big(\tilde{y}+\frac{r_{\star}}{2}e\Big)+\tilde{g}\Big(\tilde{y}-\frac{r_{\star}}{2}e\Big)\leq2K^{\frac{3}{2(n-1)}}+\partial_{ee}\tilde{g}(\tilde{y})\left(\frac{r_{\star}}{2}\right)^{2}+O\big(r_{\star}^{3}\big) 
\]
yields 
\begin{equation}\label{eqn6.20}
\partial_{ee}\tilde{g}(\tilde{y})\geq-\frac{C}{r_{\star}^{2}}\left(K^{\frac{3}{2(n-1)}}+1\right),\quad \forall e\in\mathbb{S}^{n-1}.  
\end{equation}
Combining \eqref{eqn6.17} and \eqref{eqn6.20} yields that there exists a constant $K_{1}$ such that 
\[
D^{2}\tilde{g}+K_{1}\mathbf{I}_{n}\quad \text{is positive definite in $\omega$},  
\]
that is $\tilde{g}$ is semi-convex in $\omega$. 
Define 
\[
\bar{g}:=\tilde{g}+\frac{K_{1}}{2}|y|^{2}\quad \text{in}\ \overline{\omega}. 
\]
Then $\bar{g}$ is convex in $\omega$, and there exists a constant $\overline{K}>0$ such that 
\begin{equation}\label{eqn6.21}
\inf_{B_{\frac{r_{\star}}{4n}}(y_{*})}\bar{g}< \overline{K} 
\end{equation}
for any ball $B_{\frac{r_{\star}}{4n}}(y_{*})$ contained in $B_{r_{\star}}(y_{\star})$. 
By \eqref{eqn6.21} and the convexity of $\bar{g}$, it follows (see Figure \ref{figure 1}) that  
\[
|D\bar{g}|\leq \frac{2\overline{K}}{r_{\star}/4}\quad \text{in}\ B_{\frac{r_{\star}}{4}}(y_{\star}),
\]
and then by the mean value theorem,
\[
\partial_{ee}\bar{g}(y_{e})\leq \frac{4\overline{K}}{r_{\star}/4}\frac{1}{r_{\star}/4},\quad \forall e\in\mathbb{S}^{n-1},
\]
for some point $y_{e}\in\omega$ depending on $e$. 
Combining this and \eqref{eqn6.17}, we obtain 
\[
\|\tilde{g}\|_{C^{2,1}(\overline{\omega})}\leq K_{2}, 
\]
where $K_{2}$ depends on $\Omega$, $\|\varphi\|_{C^{3,1}(\partial\Omega)}$, and $\|f^{3/(2n-2)}\|_{C^{2,1}(\overline{\Omega})}$. 
\begin{figure}[htbp]
\centering
%\hspace*{-1.8cm}
\includegraphics[width=0.5\textwidth]{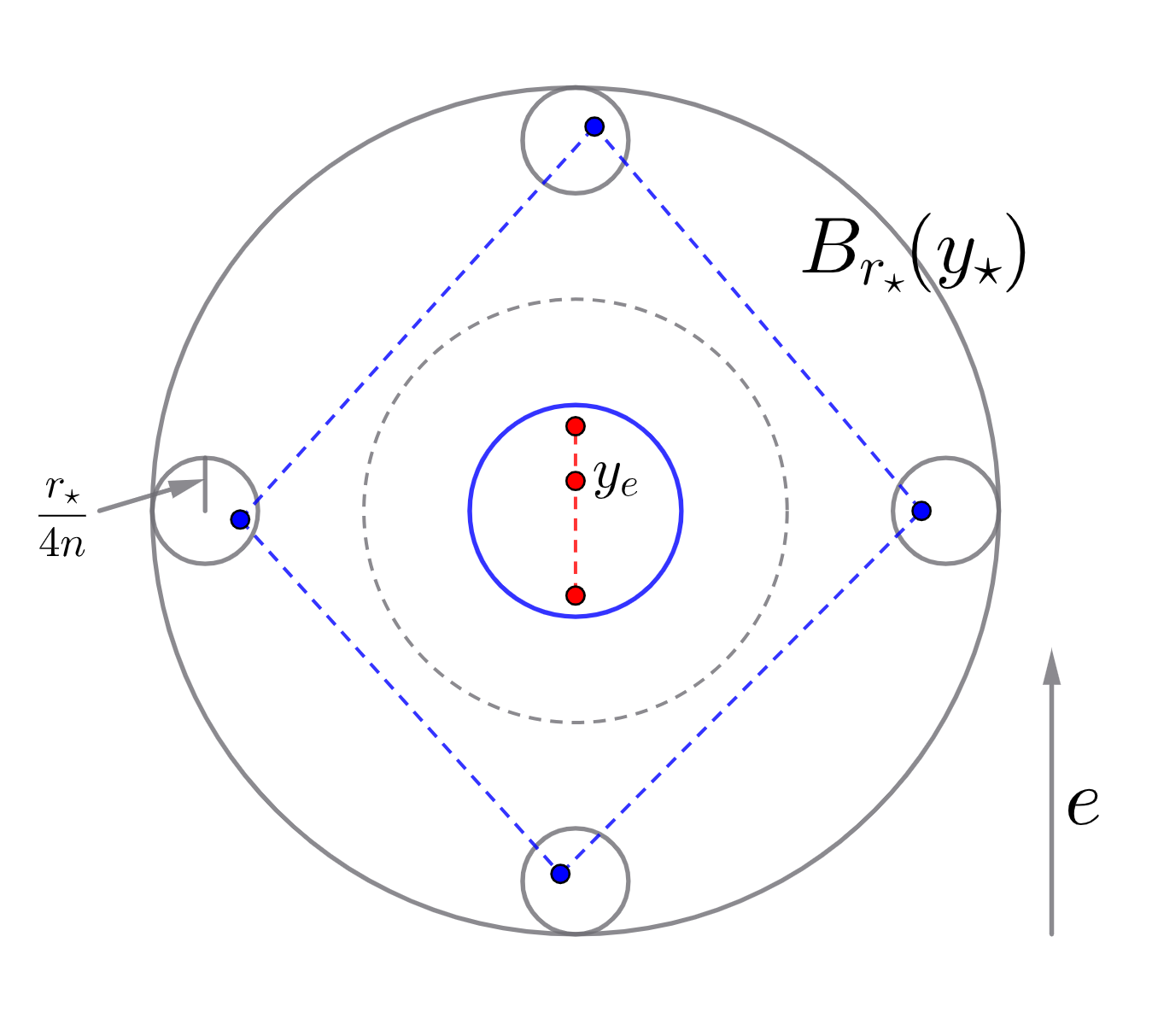}
\caption{Use the convexity of $\bar{g}$ to get a uniform bound on its first-order derivatives in $B_{\frac{r_{\star}}{4}}(y_{\star})$, and further a uniform bound on its second-order derivatives at some points is obtained by the mean value theorem.}
\label{figure 1}
\end{figure}

Therefore, 
\[
K_{2}^{\frac{2(n-1)}{3}}\geq\tilde{f}(0)=\frac{f(0)}{\prod_{i=1}^{n-1}b_{i}}. 
\]
This completes the proof of Lemma \ref{thm6.3}. 
\end{proof}

\subsection{Lower bounds for boundary tangential-normal derivatives}

\begin{lemma}\label{thm6.4}
For any $x\in\partial\Omega\cap\{|x'|\leq r_{0}\}$ satisfying $|x|\leq\frac{1}{2}\sqrt{b_{n-1}}$, we have 
\[
\left|u_{\nu\eta}(x)\right|\leq C\sqrt{b_{n-1}}, 
\]
where $\eta$ is any unit tangential vector to $\partial\Omega$ at $x$, and $C>0$ depends on $\Omega$, $\|\varphi\|_{C^{3,1}(\partial\Omega)}$, and $\|f^{3/(2n-2)}\|_{C^{2,1}(\overline{\Omega})}$, but is independent of $\inf_{\Omega}f$. 
\end{lemma}
\begin{proof}
Define 
\[
M:=\frac{1}{b_{n-1}}. 
\] 
Make a dilation $x\rightarrow y=T(x)$ defined by 
\begin{equation}\label{eqn9.3}
y'=\sqrt{M}x'\quad \text{and}\quad 
y_{n}=Mx_{n}. 
\end{equation}
Without loss of generality, we assume $M\gg1$, which depends on the $C^{3,1}$-norm of $\partial\Omega$. 
Thus, the transformation $T$ can be regarded as a blow-up transformation near the origin.  
Define 
\[
v(y):=M^{2}u(x). 
\]
Then 
\[
\det D^{2}v=M^{n-1}f\big(T^{-1}y\big)=:\tilde{f}(y)\quad \text{in}\ \widetilde{\Omega}:=T(\Omega).  
\]
It follows from \eqref{eqn6.10} and \eqref{eqn9.3} that  $\partial\widetilde{\Omega}$ is represented by 
\begin{equation}\label{eqn6.23}
y_{n}=\tilde{\rho}\big(y'\big)=\frac{1}{2}\big|y'\big|^{2}+\frac{1}{\sqrt{M}}O\Big(\big|y'\big|^{3}\Big)
\end{equation}
on $\partial\widetilde{\Omega}\cap\{|y'|\leq\sqrt{M}r_{0}\}$, where  
\[
\big|D^{3}\tilde{\rho}\big|\leq\frac{C_{1}}{\sqrt{M}},\quad \big|D^{4}\tilde{\rho}\big|\leq\frac{C_{1}}{M},  
\]
and $C_{1}$ depends only on $\partial\Omega$. 
Combining \eqref{eqn9.2} and \eqref{eqn9.3} yields   
\begin{equation}\label{eqn9.4}
v=\widetilde{\varphi}\big(y',\tilde{\rho}\big(y'\big)\big)=\frac{1}{2}\sum_{i=1}^{n-1}\tilde{b}_{i}y_{i}^{2}+\text{cubic of $y'$}+O\Big(\big|y'\big|^{4}\Big)
\end{equation}
on $\partial\widetilde{\Omega}\cap\{|y'|\leq\sqrt{M}r_{0}\}$, where
\[
\tilde{b}_{i}\leq1\ \ \text{for}\ i=1,2,\dots,n-1,\quad \text{and}\quad \big|D_{y}^{4}\widetilde{\varphi}\big|\leq\big|D_{x}^{4}\varphi\big|\leq C. 
\]   
By applying Lemma \ref{thm6.2}, we conclude that the coefficients of the third-order terms in \eqref{eqn9.4} are also bounded. 
Hence both $\partial\widetilde{\Omega}$ and $\widetilde{\varphi}$ are $C^{3,1}$-smooth on $\partial\widetilde{\Omega}\cap\{|y'|\leq\sqrt{M}r_{0}\}$, and their $C^{3,1}$-norms are independent of $M$. 

Define a set 
\[
\omega:=\left\{y\in\widetilde{\Omega}:\ y_{n}<1\right\}.  
\] 
Looking at $\partial\omega\cap\{y_{n}=1\}$, by \eqref{eqn6.23} we have 
\[
1=\frac{1}{2}\big|y'\big|^{2}+\frac{1}{\sqrt{M}}O\Big(\big|y'\big|^{3}\Big).    
\]
Since $M\gg1$, any point with $|y'|=2$ must be outside of $\omega$.   
So the set $\omega$ is contained in the ball $B_{2}$ by the convexity of $\omega$.  
Similar to Step 2 in the proof of Lemma \ref{thm6.3}, the ball $B_{r_{\star}}(y_{\star})$ is contained in $\omega$, where $y_{\star}$ and $r_{\star}$ are given in \eqref{eqn9.5} and \eqref{eqn9.6}, respectively. 
Furthermore, using the convexity of $v$ and \eqref{eqn9.4}, we derive 
\[
0\leq v\leq\sup_{\partial\omega\cap\partial\widetilde{\Omega}}\widetilde{\varphi}\leq C_{2}\quad \text{in}\ \overline{\omega}, 
\]
where $C_{2}$ depends only on $\|\varphi\|_{C^{3,1}(\partial\Omega)}$. 

Define 
\[
\tilde{g}:=\tilde{f}^{\frac{3}{2n-2}}\quad \text{in}\ \widetilde{\Omega}. 
\]
By the same argument as in Step 3 of the proof of Lemma \ref{thm6.3}, we get 
\[
\left\|\tilde{g}\right\|_{C^{2,1}(\overline{\omega})}\leq K,  
\]
where $K$ depends on $\Omega$, $\|\varphi\|_{C^{3,1}(\partial\Omega)}$, and $\|f^{3/(2n-2)}\|_{C^{2,1}(\overline{\Omega})}$. 
Note that $\partial\widetilde{\Omega}$ is uniformly convex on $\partial\widetilde{\Omega}\cap\partial\omega$ by \eqref{eqn6.23}. 
The first-order normal estimate on the boundary follows from Theorem 17.21 of Gilbarg-Trudinger \cite{GT} that 
\begin{equation}\label{eqn6.22}
\left|v_{\tilde{\nu}}(y)\right|\leq C_{3},\quad \forall y\in\partial\widetilde{\Omega}\cap\left\{y_{n}\leq7/8\right\},  
\end{equation}
where $\tilde{\nu}=\tilde{\nu}(y)$ denotes the unit interior normal vector to $\partial\widetilde{\Omega}$ at $y$, and $C_{3}$ depends on $\Omega$, $\|\varphi\|_{C^{3,1}(\partial\Omega)}$, and $\|f^{3/(2n-2)}\|_{C^{2,1}(\overline{\Omega})}$. 
Using the convexity of $v$ and the uniform convexity of  $\partial\widetilde{\Omega}$, it follows from \eqref{eqn6.22} that  
\[
|Dv|\leq C_{4}\quad \text{in}\ \widetilde{\Omega}\cap\left\{y_{n}\leq3/4\right\}, 
\]
where $C_{4}$ depends on $\Omega$, $\|\varphi\|_{C^{3,1}(\partial\Omega)}$, and $\|f^{3/(2n-2)}\|_{C^{2,1}(\overline{\Omega})}$. 
In view of Remark \ref{thm2.9}, to weaken the condition to $\inf_{\Omega}\tilde{g}>0$ when applying Lemma \ref{thm2.2}, we adapt the proof of Lemma \ref{thm7.2} by differentiating with respect to a tangential vector filed to establish the boundary tangential-normal estimate.  
The convexity of $v$ and the uniform convexity of $\partial\widetilde{\Omega}$ play key roles in this argument.
Define tangential vector fields 
\[
\tilde{\eta}^{l}(y):=e_{l}+\tilde{\rho}_{l}\big(y'\big)e_{n}\quad \text{for}\  l=1,2,\dots,n-1,   
\]
in $\overline{\omega}$. 
Consider functions 
\[
\Psi^{l}(y):=\pm\tilde{\eta}^{l}\cdot\nabla(v-\widetilde{\varphi})+A|y|^{2}-B_{A}y_{n}\quad \text{in}\ \overline{\omega},  
\]
for all $l=1,2,\dots,n-1$, where positive constants $A$ and $B_{A}$ are to be determined. 
We can obtain 
\[
|v_{nl}(0)|\leq C,\quad \forall 1\leq l\leq n-1, 
\]
and similarly 
\[
\left|v_{\tilde{\nu}\tilde{\eta}^{l}}(y)\right|\leq C,\quad \forall y\in\partial\widetilde{\Omega}\cap\left\{y_{n}\leq1/2\right\},\quad \forall 1\leq l\leq n-1. 
\]
where $C$ depends on $\Omega$, $\|\varphi\|_{C^{3,1}(\partial\Omega)}$, and $\|f^{3/(2n-2)}\|_{C^{2,1}(\overline{\Omega})}$. 
Pulling back to the $x$-coordinates, we finish the proof of Lemma \ref{thm6.4}. 
Since a more complicated version of these arguments will be needed later, in the absence of uniform convexity of $\partial\widetilde{\Omega}$, we omit the details here.  
\end{proof}

\par
\vspace{2mm}
We proceed from Lemma \ref{thm6.4} by induction to obtain a refinement. 
Define
\begin{equation}\label{eqn6.43}
\eta^{l}=\eta^{l}(x):=e_{l}+\rho_{l}\big(x'\big)e_{n},\quad \forall x\in\partial\Omega\cap\left\{\big|x'\big|\leq r_{0}\right\},  
\end{equation}
for all $l=1,2,\dots,n-1$. 
In fact, $\eta^{l}(x)$ is the tangential direction of $\partial\Omega$, which lies in the two-dimensional plane parallel to the $x_{l}$- and $x_{n}$-axes and passes through the point $x$. 
Our induction hypothesis is the following. 

\begin{assumption}\label{thm6.5}
For some fixed $k\in\{1,2,\dots,n-2\}$ and each $i=k+1,\dots,n-1$, there exists a  constant $\theta_{i}>0$ such that 
for any $x\in\partial\Omega\cap\{|x'|\leq r_{0}\}$ satisfying  $|x|\leq\theta_{i}\sqrt{b_{i}}$, we have 
\[
\left|u_{\nu\eta^{j}}(x)\right|\leq C\sqrt{b_{i}},\quad \forall1\leq j\leq i,
\]
where constants $\theta_{i}$ and $C$ depend on $\Omega$, $\|\varphi\|_{C^{3,1}(\partial\Omega)}$, and $|f^{3/(2n-2)}|_{C^{2,1}(\overline{\Omega})}$, but is independent of $\inf_{\Omega}f$.  
\end{assumption}

\begin{lemma}\label{thm6.6}
Let Assumption \ref{thm6.5} hold for some fixed $k\in\{1,2,\dots,n-2\}$. 
Then there exists a constant $\theta_{k}>0$ such that for any $x\in\partial\Omega\cap\{|x'|\leq r_{0}\}$ satisfying  $|x|\leq\theta_{k}\sqrt{b_{k}}$, we have 
\begin{equation}\label{eqn6.24}
\left|u_{\nu\eta^{j}}(x)\right|\leq C\sqrt{b_{k}},\quad \forall1\leq j\leq k,
\end{equation}
where constants $\theta_{k}$ and $C$ depend on $\Omega$, $\|\varphi\|_{C^{3,1}(\partial\Omega)}$, and $|f^{3/(2n-2)}|_{C^{2,1}(\overline{\Omega})}$, but are independent of $\inf_{\Omega}f$. 
\end{lemma}
\begin{proof}
Define 
\[
M:=\frac{1}{b_{k}}. 
\] 
Make a dilation $x\rightarrow y=T(x)$ defined by 
\begin{equation}\label{eqn6.42}
y_{i}:=M_{i}x_{i}\quad \text{for}\ i=1,2,\dots,n, 
\end{equation}
where 
\[
M_{i}:=
\begin{cases}
\sqrt{M},\quad i=1,\dots,k,\\
\sqrt{b_{i}}M\ (\geq\sqrt{M}),\quad i=k+1,\dots,n-1,\\
M,\quad i=n. 
\end{cases}
\]
Without loss of generality, we assume $M\gg1$, which depends on the $C^{3,1}$-norm of $\partial\Omega$. 
Thus, the transformation $T$ can be regarded as a blow-up transformation near the origin. 
Define 
\[
v(y):=M^{2}u(x). 
\]
Then 
\[
\det D^{2}v=\frac{M^{k}f(T^{-1}y)}{\prod_{i=k+1}^{n-1}b_{i}}=:\tilde{f}(y)\quad \text{in}\ \widetilde{\Omega}:=T(\Omega).  
\]
Denote 
\[
\hat{y}:=(y_{1},\dots,y_{k})\quad \text{and}\quad \tilde{y}:=(y_{k+1},\dots,y_{n-1}). 
\]
It follows from \eqref{eqn6.10} and \eqref{eqn6.42} that  $\partial\widetilde{\Omega}$ is represented by 
\begin{equation}\label{eqn6.26}
y_{n}=\tilde{\rho}\big(y'\big)=\frac{1}{2}|\hat{y}|^{2}+\frac{1}{2}\sum_{i=k+1}^{n-1}d_{i}y_{i}^{2}+\frac{1}{\sqrt{M}}O\Big(\big|y'\big|^{3}\Big)
\end{equation}
on $\partial\widetilde{\Omega}\cap\{|y'|\leq\sqrt{M}r_{0}\}$, where  
\[
d_{i}=\frac{b_{k}}{b_{i}}\ (\leq1)\ \ \forall k+1\leq i\leq n-1,\ \  \big|D^{3}\tilde{\rho}\big|\leq\frac{C_{1}}{\sqrt{M}},\ \  \big|D^{4}\tilde{\rho}\big|\leq\frac{C_{1}}{M},  
\]
and $C_{1}$ depends only on $\partial\Omega$. 
Combining \eqref{eqn9.2} and \eqref{eqn6.42} yields   
\begin{equation}\label{eqn6.28}
v=\widetilde{\varphi}\big(y',\tilde{\rho}\big(y'\big)\big)=\frac{1}{2}|\tilde{y}|^{2}+\frac{1}{2}\sum_{i=1}^{k}\tilde{b}_{i}y_{i}^{2}+\text{cubic of $y'$}+O\Big(\big|y'\big|^{4}\Big)
\end{equation}
on $\partial\widetilde{\Omega}\cap\{|y'|\leq\sqrt{M}r_{0}\}$, where
\[
\tilde{b}_{i}=\frac{b_{i}}{b_{k}}\ (\leq1)\ \ \text{for}\ i=1,2,\dots,k, \quad\text{and} \quad \big|D_{y}^{4}\widetilde{\varphi}\big|\leq\big|D_{x}^{4}\varphi\big|\leq C. 
\] 
By applying Lemma \ref{thm6.2}, we conclude that the coefficients of the third-order terms in \eqref{eqn6.28} are also bounded. 
Hence both $\partial\widetilde{\Omega}$ and $\widetilde{\varphi}$ are $C^{3,1}$-smooth on $\partial\widetilde{\Omega}\cap\{|y'|\leq\sqrt{M}r_{0}\}$, and their $C^{3,1}$-norms are independent of $M$. 

Define a set 
\[
\omega:=\left\{y\in\widetilde{\Omega}:\ y_{n}<\beta^{2},\ |y_{i}|<\beta,\ \,\forall k+1\leq i\leq n-1\right\}, 
\]
where $\beta$ is a small positive constant to be determined.  
We analyze the geometry of $\omega$. 
Clearly, the set $\omega$ is convex. 
Looking at $\partial\omega\cap\{y_{n}=\beta^{2}\}$, by \eqref{eqn6.26} we have 
\[
\beta^{2}\geq\frac{1}{2}|\hat{y}|^{2}+\frac{1}{\sqrt{M}}O\Big(\big|y'\big|^{3}\Big).
\]
Since $M\gg1$, any point with $|\hat{y}|=2\beta$ must be outside of $\omega$.   
So the set $\omega$ is contained in the ball $B_{\sqrt{n+3}\beta}$ by the convexity of $\omega$.  
It is trivial that 
\[
y_{\star}:=\left(0,\dots,0,\frac{\beta}{4n},\dots,\frac{\beta}{4n},\frac{\beta^{2}}{2}\right)\in\omega,  
\]
and 
\[
\operatorname{dist}(y_{\star},\partial\omega)\geq\min\left\{\frac{\beta^{2}}{2},\beta-\frac{\beta}{4n},\operatorname{dist}\Big(y_{\star},\partial\omega\cap\partial\widetilde{\Omega}\Big)\right\}. 
\]
Then, since $\tilde{\rho}(y')\leq|y'|^{2}$, the ball $B_{r_{\star}}(y_{\star})$ is contained in $\omega$, where 
\[
r_{\star}:=\min\left\{\frac{\beta^{2}}{2},\beta-\frac{\beta}{4n},\min\left\{\frac{\beta}{4n},\frac{\beta^{2}}{2}-\frac{\beta^{2}}{4n}\right\}\right\}. 
\]
Using the convexity of $v$ and \eqref{eqn6.28}, we derive 
\begin{equation}\label{eqn6.34}
0\leq v\leq\sup_{\partial\omega\cap\partial\widetilde{\Omega}}\widetilde{\varphi}\leq C_{2}\quad \text{in}\ \overline{\omega}, 
\end{equation}
where $C_{2}$ depends only on $\|\varphi\|_{C^{3,1}(\partial\Omega)}$. 
Define 
\[
\tilde{g}:=\tilde{f}^{\frac{3}{2n-2}}\quad \text{in}\ \widetilde{\Omega},\quad \text{and}\quad p=\frac{2n-2}{3}. 
\]
By the same argument as in Step 3 of the proof of Lemma \ref{thm6.3}, we get 
\[
\left\|\tilde{g}\right\|_{C^{2,1}(\overline{\omega})}\leq C_{3},  
\]
where $C_{3}$ depends on $\Omega$, $\|\varphi\|_{C^{3,1}(\partial\Omega)}$, and $\|f^{3/(2n-2)}\|_{C^{2,1}(\overline{\Omega})}$. 

The main difficulty is that we cannot control the convexity of $\partial\widetilde{\Omega}$ in \eqref{eqn6.26}. 
An observation is that the directions, in which the convexity is under control, with respect to $\partial\widetilde{\Omega}$ and $\widetilde{\varphi}$ complement each other. 

\par
\vspace{2mm}
\textbf{Step 1:} \emph{we prove $|Dv|\leq C$ in $\omega$.} 
Define  
\[
\underline{v}(y):=\frac{1}{2}\sigma\big|y'\big|^{2}+\frac{1}{2}Ky_{n}^{2}-Ky_{n}, 
\]
where constants $\sigma>0$ and $K>0$ are to be determined. 
Let $K$ large depending on $\sigma$ such that 
\begin{equation}\label{eqn6.39}
\det D^{2}\underline{v}=\sigma^{n-1}K\geq2^{n}\sup_{\omega}\tilde{f}. 
\end{equation}

We now verify that 
\begin{equation}\label{eqn6.25}
\underline{v}\leq v\quad \text{on}\ \partial\omega. 
\end{equation}
The boundary $\partial\omega$ consists of three parts: $\partial_{1}\omega\cup\partial_{2}\omega\cup\partial_{3}\omega$, where $\partial_{1}\omega:=\partial\omega\cap\partial\widetilde{\Omega}$, $\partial_{2}\omega:=\partial\omega\cap\{y_{n}=\beta^{2}\}$, and $\partial_{3}\omega:=\partial\omega\cap\{|y_{i}|=\beta\ \text{for some}\ i=k+1,\dots,n-1\}$.
It follows from \eqref{eqn6.28} that on $\partial_{1}\omega$,  
\[
v=\frac{1}{2}|\tilde{y}|^{2}+\frac{1}{2}\sum_{i=1}^{k}\tilde{b}_{i}y_{i}^{2}+O\Big(\big|y'\big|^{3}\Big)\geq \frac{1}{4}|\tilde{y}|^{2}-|\hat{y}|^{2}
\]
by taking $\beta$ small.  
By \eqref{eqn6.26}, we have 
\[
y_{n}=\frac{1}{2}|\hat{y}|^{2}+\frac{1}{2}\sum_{i=k+1}^{n-1}d_{i}y_{i}^{2}+\frac{1}{\sqrt{M}}O\Big(\big|y'\big|^{3}\Big)
\geq\frac{1}{4}|\hat{y}|^{2}-\beta|\tilde{y}|^{2}
\]
on $\partial_{1}\omega$, since $M\gg1$. 
Hence 
\[
\underline{v}\leq\frac{1}{2}\sigma\big|y'\big|^{2}-\frac{1}{2}Ky_{n}\leq\left(\frac{1}{2}\sigma+\frac{1}{2}K\beta\right)\left|\tilde{y}\right|^{2}+\left(\frac{1}{2}\sigma-\frac{1}{8}K\right)\left|\hat{y}\right|^{2}\leq v
\]
on $\partial_{1}\omega$, by taking $\sigma<1/4$ and $\beta$ small depending on $K$. 
It holds that 
\[
\underline{v}\leq\frac{1}{2}\sigma\big|y'\big|^{2}-\frac{1}{2}K\beta^{2}\leq\frac{1}{2}\beta^{2}((n+3)\sigma-K)\leq0\leq v\quad \text{on}\ \partial_{2}\omega. 
\]
We only consider the piece  $\partial'_{3}\omega:=\partial\omega\cap\{y_{n-1}=\beta\}$, since other pieces of $\partial_{3}\omega$ can be handled in the same way. 
Let $\varepsilon_{0}>0$ be a small parameter to be determined later. 
It holds that  
\[
\underline{v}\leq\frac{1}{2}\beta^{2}((n+3)\sigma-K\varepsilon_{0})\leq0\leq v\quad \text{on}\ \partial'_{3}\omega\cap\left\{y_{n}\geq\varepsilon_{0}\beta^{2}\right\} 
\] 
by taking $\sigma$ small depending on $\varepsilon_{0}$. 

We claim that   
\begin{equation}\label{eqn6.27}
v\geq\frac{1}{4}\beta^{2}\quad \text{on}\ \partial'_{3}\omega\cap\left\{y_{n}<\varepsilon_{0}\beta^{2}\right\}.   
\end{equation}
It follows from \eqref{eqn6.28} that  
\[
v\geq\frac{1}{2}\beta^{2}+O\big(\beta^{3}\big)\geq\frac{1}{4}\beta^{2}\quad \text{on}\ \partial'_{3}\omega\cap\left\{y_{n}<\varepsilon_{0}\beta^{2}\right\}\cap\partial\widetilde{\Omega},
\]
by taking $\beta$ small. 
It remains to consider points $p=(\hat{p},\tilde{p},p_{n})$ on $\partial'_{3}\omega\cap\{y_{n}<\varepsilon_{0}\beta^{2}\}\cap\widetilde{\Omega}$, where $\hat{p}:=(p_{1},\dots,p_{k})$ and $\tilde{p}:=(p_{k+1},\dots,p_{n-1})$. 
Define 
\[
\delta:=\left|\nabla_{\hat{y}}\tilde{\rho}(0,\tilde{p})\right|\left|\hat{p}\right| 
\]
and 
\[
p^{*}:=(0,\tilde{p},p_{n}+\delta),\quad p^{0}:=(0,\tilde{p},\tilde{\rho}(0,\tilde{p}))\in\partial\widetilde{\Omega}.
\]
From \eqref{eqn6.26}, we deduce 
\[
\left|\nabla_{\hat{y}}\tilde{\rho}(0,\tilde{p})\right|=O\big(\beta^{2}\big). 
\]
Then $\delta=O(\beta^{3})$. 
Since $\widetilde{\Omega}$ is convex, 
\[
p^{*}_{n}=p_{n}+\delta> p^{0}_{n}+\nabla_{\hat{y}}\tilde{\rho}(0,\tilde{p})\cdot\hat{p}+\delta\geq p^{0}_{n},
\]
which implies $p^{*}\in\widetilde{\Omega}$.
Extend the segment $\overline{pp^{*}}$ until it intersects  $\partial\widetilde{\Omega}$ at a point denoted by $\bar{p}$ (see Figure \ref{figure 2}). 
\begin{figure}[htbp]
\centering
%\hspace*{-1.8cm}
\includegraphics[width=0.95\textwidth]{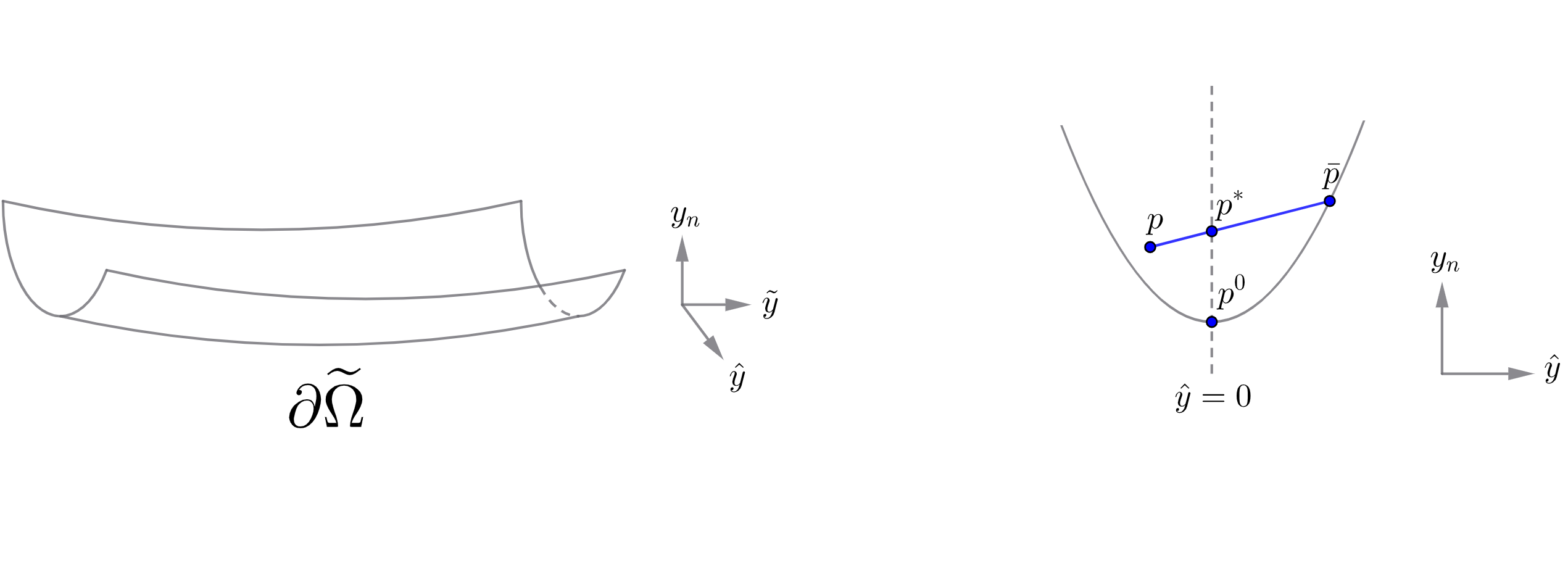}
\caption{The convexity of $\partial\widetilde{\Omega}$ in the direction $\tilde{y}$ is not under control. Cut $\widetilde{\Omega}$ by the hyperplane $\tilde{y}=\tilde{p}$ through the point $p$, and relate it to points on the boundary $\partial\widetilde{\Omega}$.}
\label{figure 2}
\end{figure}
From \eqref{eqn6.26}, the graph of $\partial\widetilde{\Omega}$ near $p^{0}$ in the direction $\hat{y}$ is a perturbation of the paraboloid $|\hat{y}|^{2}/2$, by taking $\beta$ small. 
Hence we have 
\begin{equation}\label{eqn6.30}
\frac{1}{2}|p^{*}-p|\leq|\bar{p}-p^{*}|. 
\end{equation}
Since the slope of $\overline{pp^{*}}$ is   $|\nabla_{\hat{y}}\tilde{\rho}(0,\tilde{p})|=O(\beta^{2})$, we solve 
\[
\frac{1}{4}t^{2}\leq\left(\varepsilon_{0}\beta^{2}+O\big(\beta^{3}\big)\right)+O\big(\beta^{2}\big)t
\] 
to deduce $t=O(\beta)$. 
Thus 
\begin{equation}\label{eqn6.29}
\bar{p}_{n}\leq \left(\varepsilon_{0}\beta^{2}+O\big(\beta^{3}\big)\right)+O\big(\beta^{3}\big)\leq2\varepsilon_{0}\beta^{2}
\end{equation}
by taking $\beta$ small depending on $\varepsilon_{0}$. 
Using \eqref{eqn6.26} and \eqref{eqn6.29} we derive 
\begin{equation}\label{eqn6.31}
\frac{1}{2}\sum_{i=1}^{k}\left|\bar{p}_{i}\right|^{2}\leq\bar{p}_{n}+O\big(\beta^{3}\big)\leq3\varepsilon_{0}\beta^{2}. 
\end{equation}
For any $T^{-1}x=y=(0,\tilde{y},\tilde{\rho}(0,\tilde{y}))\in\partial\widetilde{\Omega}\cap B_{\theta_{k+1}}(0)$ with $\theta_{k+1}$ given in Assumption \ref{thm6.5}, we have 
\[
\hat{x}=0\quad \text{and}\quad |x_{i}|=\left|\frac{y_{i}}{M_{i}}\right|\leq\theta_{k+1}\frac{b_{k}}{\sqrt{b_{i}}},\quad \forall k+1\leq i\leq n-1. 
\]
Since $|\nabla u(x)-\nabla u(0)|=O(|x'|)$ and $|u_{\nu}(x)-u_{n}(x)|=O(|x'|^{2})$ on $\partial\Omega\cap\{|x'|\leq r_{0}\}$ (see later \eqref{eqn4.6}, \eqref{eqn4.9}), using Assumption \ref{thm6.5}, we infer 
\begin{align*}
\left|v_{y_{n}}(y)\right|&=\frac{1}{b_{k}}\left(\left|u_{\nu}\big(x',\rho\big(x'\big)\big)\right|+O\Big(\big|x'\big|^{2}\Big)\right)\\
&=\frac{1}{b_{k}}\left|u_{n}(0)+\nabla_{x'}(u_{\nu})\big(tx',\rho\big(tx'\big)\big)\cdot x'\right|+C\theta_{k+1}^{2}\\
&=\frac{1}{b_{k}}\sum_{i=k+1}^{n-1}(u_{\nu\eta_{i}}+\partial_{x_{i}}\nu\cdot \nabla u)\big(tx',\rho\big(tx'\big)\big)x_{i}+C\theta_{k+1}^{2}\\
&\leq\frac{1}{b_{k}}\sum_{i=k+1}^{n-1}\left(C\sqrt{b_{i}}x_{i}+C\big|x'\big|^{2}\right)+C\theta_{k+1}^{2}\\
&\leq C_{4}   
\end{align*}
for some $t\in(0,1)$, where $C_{4}$ depends on $\Omega$, $\|\varphi\|_{C^{3,1}(\partial\Omega)}$, and $\|f^{3/(2n-2)}\|_{C^{2,1}(\overline{\Omega})}$. 
By the convexity of $v$, \eqref{eqn6.28}, and \eqref{eqn6.29}, we get  
\begin{align}\label{eqn6.32}
v(p^{*})&\geq v\big(p^{0}\big)+\nabla v\big(p^{0}\big)\cdot\big(p^{*}-p^{0}\big)\nonumber\\
&\geq \widetilde{\varphi}\big(p^{0}\big)-C_{4}\bar{p}_{n}\nonumber\\
&\geq\frac{1}{2}\beta^{2}-2C_{4}\varepsilon_{0}\beta^{2}+O\big(\beta^{3}\big)\nonumber\\
&\geq\frac{1}{3}\beta^{2} 
\end{align}
by taking $\varepsilon_{0}$ small. 
Furthermore, we obtain from \eqref{eqn6.28} and \eqref{eqn6.29}--\eqref{eqn6.32} that 
\begin{align}
v(\bar{p})=\widetilde{\varphi}(\bar{p})&=\frac{1}{2}\sum_{i=1}^{k}\tilde{b}_{i}\left|\bar{p}_{i}\right|^{2}+\widetilde{\varphi}\big(p^{0}\big)+O\big(\beta^{3}\big)\nonumber\\
&\leq v(p^{*})+C_{4}\bar{p}_{n}+4\varepsilon_{0}\beta^{2}\nonumber\\
&\leq v(p^{*})+(2C_{4}+4)\varepsilon_{0}\beta^{2}\nonumber\\ 
&\leq \frac{9}{8}v(p^{*})\label{eqn6.33}
\end{align}
by taking $\varepsilon_{0}$ small. 
Using the convexity of $v$, \eqref{eqn6.33}, \eqref{eqn6.30}, and \eqref{eqn6.32}, we get 
\[
v(p)\geq-\frac{|p^{*}-p|}{|\bar{p}-p^{*}|}v(\bar{p})+\frac{|\bar{p}-p|}{|\bar{p}-p^{*}|}v(p^{*})\geq\left(1-\frac{1}{8}\frac{|p^{*}-p|}{|\bar{p}-p^{*}|}\right)v(p^{*})\geq\frac{1}{4}\beta^{2}. 
\]
In conclusion, claim \eqref{eqn6.27} is valid. 

We deduce from \eqref{eqn6.27} that  
\[
\underline{v}\leq\frac{1}{2}\sigma\big|y'\big|^{2}\leq\frac{n+3}{2}\sigma\beta^{2} \leq v\quad \text{on}\ \partial'_{3}\omega\cap\left\{y_{n}<\varepsilon_{0}\beta^{2}\right\},  
\]
by taking $\sigma$ small. 
Thus \eqref{eqn6.25} holds. 
It follows from the comparison principle that 
\[
\underline{v}\leq v\quad \text{in}\ \overline{\omega}. 
\]
Since $\underline{v}(0)=v(0)=0$, 
\[
v_{n}(0)\geq\underline{v}_{n}(0).
\]
Using the convexity of $v$ and \eqref{eqn6.34}, we get 
\[
v_{n}(0)\leq\frac{v((0,\dots,0,\beta^{2}))-v(0)}{\beta^{2}}\leq \frac{C_{2}}{\beta^{2}}. 
\]
So $|v_{n}(0)|\leq C$. 
For any point $y_{0}\in\partial\omega\cap\partial\widetilde{\Omega}$, letting $y_{0}$ be the origin and $\tilde{\nu}(y_{0})=e_{n}$ by translation and rotation, letting 
$v(y_{0})=0$ and $\nabla v(y_{0})=0$ by subtracting an affine function, with similar arguments as above, we get 
\begin{equation}\label{eqn6.35}
|v_{\tilde{\nu}}(y_{0})|\leq C,\quad \forall y_{0}\in\partial\omega\cap\partial\widetilde{\Omega},  
\end{equation}
where $\tilde{\nu}=\tilde{\nu}(y)$ denotes the unit interior normal vector to $\partial\widetilde{\Omega}$ at $y$. 
From \eqref{eqn6.35} and the convexity of $v$, it follows that 
\[
|Dv|\leq C\quad \text{in}\ \overline{\omega}
\]
for a smaller $\beta$, where $C$ depends on $\Omega$, $\|\varphi\|_{C^{3,1}(\partial\Omega)}$, and $|f^{3/(2n-2)}|_{C^{2,1}(\overline{\Omega})}$.

\par
\vspace{2mm}
\textbf{Step 2.} 
Define tangential vector fields 
\[
\tilde{\eta}^{l}(y):=e_{l}+\tilde{\rho}_{l}\big(y'\big)e_{n} \quad \text{for}\ l=1,2,\dots,n-1,    
\]
in $\overline{\omega}$.
Indeed, for any $l\in\{1,2,\dots,n-1\}$, 
\begin{equation}\label{eqn9.9}
T\eta^{l}=T(e_{l}+\partial_{x_{l}}\rho e_{n})=M_{l}e_{l}+\partial_{x_{l}}\rho M_{n}e_{n}=M_{l}\left(e_{l}+\partial_{y_{l}}\tilde{\rho}e_{n}\right)=M_{l}\tilde{\eta}^{l}. 
\end{equation}
Construct functions 
\[
\Psi^{l}(y):=\pm\tilde{\eta}^{l}\cdot\nabla(v-\widetilde{\varphi})+A|y|^{2}-B_{A}\bigg(v-\frac{1}{2}\underline{v}+(C_{5}+1)y_{n}\bigg)\quad \text{in}\ \overline{\omega},  
\]
for $l=1,2,\dots,n-1$, where positive constants $A$ and $B_{A}$ are to be determined, and $C_{5}:=\sup_{\omega}|D(v-\underline{v}/2)|$.  

First, we choose $A$ large so that the functions $\Psi^{l}$ cannot attain its maximum in $\omega$ for $l=1,2,\dots,n-1$. 
Fix $l\in\{1,2,\dots,n-1\}$, and suppose, for contradiction, that $\Psi^{l}$ attains its maximum at an interior point $y_{*}\in\omega$. 
Then at the point $y_{*}$, we have 
\begin{equation}\label{eqn6.36}
\partial_{n}\Psi^{l}=0\quad \text{and}\quad F^{ij}\big[D^{2}v\big]\Psi^{l}_{ij}\leq0. 
\end{equation}
From \eqref{eqn6.36}, at the point $y_{*}$, 
\[
\pm\partial_{n}\Big(\tilde{\eta}^{l}\cdot\nabla(v-\widetilde{\varphi})\Big)=B_{A}\bigg(\bigg(v-\frac{1}{2}\underline{v}\bigg)_{n}+C_{5}+1\bigg)-2Ay_{n}\geq B_{A}-2A, 
\]
which yields $|D^{2}v(y_{*})|>1$ by taking $B_{A}$ large depending on $A$. 
Since $\lambda(D^{2}v(y_{*}))\in\Gamma_{2}$, 
\begin{equation}\label{eqn6.37}
\sigma_{1}\big[D^{2}v(y_{*})\big]\geq\left|D^{2}v(y_{*})\right|>1. 
\end{equation}
For convenience, introduce a coordinate system $z$, obtained by an orthogonal transformation of the $y$-coordinates, such that at the point $y_{*}$ in the $z$-coordinates, 
\[
D_{z}^{2}v\ \ \text{is diagonal, which implies}\ \ F^{ij}_{z}\big[D^{2}_{z}v\big]\ \ \text{is also diagonal}. 
\]
Computing \eqref{eqn6.36} at the point $y_{*}$ in the $z$-coordinates, by the concavity and homogeneity of $F$ and \eqref{eqn6.39}, we obtain 
\begin{align}
0&\geq F^{ij}_{z}\big[D^{2}_{z}v\big]\Psi^{l}_{z_{i}z_{j}}\nonumber\\  
&\geq \pm F^{ij}_{z}\big[D^{2}_{z}v\big]\left(\tilde{\eta}^{l}_{z}\cdot\nabla_{z}(v-\widetilde{\varphi})\right)_{z_{i}z_{j}}+2A\operatorname{tr}F^{ij}\big[D^{2}v\big]\nonumber\\
&\hspace{4.5mm}+B_{A}\bigg(\frac{1}{2}F\big[D^{2}\underline{v}\big]-F\big[D^{2}v\big]\bigg)\nonumber\\
&\geq\pm F^{ij}_{z}\big[D^{2}_{z}v\big]\left(\tilde{\eta}^{l}_{z}\cdot\nabla_{z} v_{z_{i}z_{j}}+2\partial_{z_{j}}\tilde{\eta}^{l}_{z}\cdot\nabla_{z} v_{z_{i}}\right)+(2A-C)\operatorname{tr}F^{ij}\nonumber\\
&=\pm\tilde{\eta}^{l}\cdot\nabla\left(\tilde{g}^{\frac{p}{n}}\right)\pm2\sum_{i=1}^{n}F^{ii}_{z}\big[D^{2}_{z}v\big]v_{z_{i}z_{i}}\times\partial_{z_{i}}\tilde{\eta}^{l,i}_{z}+(2A-C)\operatorname{tr}F^{ij}\nonumber\\
&\geq-\left|\frac{p}{n}\tilde{g}^{\frac{p}{n}-\frac{1}{2}}\frac{\tilde{\eta}^{l}\cdot\nabla\tilde{g}}{\sqrt{\tilde{g}}}\right|
-C\sum_{i=1}^{n}\left|F^{ii}_{z}\big[D^{2}_{z}v\big]v_{z_{i}z_{i}}\right|+(2A-C)\operatorname{tr}F^{ij},\label{eqn6.38}
\end{align}
where $\tilde{\eta}^{l,i}_{z}$ denotes the $i$-th component of $\tilde{\eta}^{l}$ in the $z$-coordinates, and the constant $C>0$ depends on $\tilde{\rho}$ and $\sup_{\omega}|D(v-\widetilde{\varphi})|$. 
Since $v_{z_{i}z_{i}}(y_{*})\geq0$ for all $1\leq i\leq n$ by the convexity of $v$, using the homogeneity and ellipticity of $F$, we have 
\begin{equation}\label{eqn9.8}
\tilde{f}^{\frac{1}{n}}=\sum_{j=1}^{n}F^{jj}_{z}\big[D^{2}_{z}v\big]v_{z_{j}z_{j}}\geq \left|F^{ii}_{z}\big[D^{2}_{z}v\big]v_{z_{i}z_{i}}\right|\quad \text{for}\  i=1,2,\dots,n,  
\end{equation}
at the point $y_{*}$ in the $z$-coordinates. 
Combining \eqref{eqn6.38}, \eqref{eqn9.8}, Lemma \ref{thm2.2}, \eqref{eqn2.11}, and \eqref{eqn6.37}, we derive at the point $y_{*}$, 
\begin{align*}
0&\geq-C_{6}\tilde{f}^{\frac{1}{4(n-1)}-\frac{1}{n(n-1)}}-C\tilde{f}^{\frac{1}{n}}+\frac{2A-C}{C_{7}}\sigma_{1}^{\frac{1}{n-1}}\big[D^{2}v\big]\tilde{f}^{\frac{-1}{n(n-1)}}\\
&\geq \frac{2A-C}{C_{7}}\tilde{f}^{\frac{-1}{n(n-1)}}-C_{6}\tilde{f}^{\frac{1}{4(n-1)}-\frac{1}{n(n-1)}}-C\tilde{f}^{\frac{1}{n}}\\
&>0 
\end{align*}
by taking $A$ large, which is a contradiction, where $C_{7}>0$ is universal, and $C_{6}$ depends on $\|\tilde{g}\|_{C^{1,1}(\overline{\omega})}$ .  
Thus $\Psi^{l}$ cannot attain its maximum in $\omega$ for $l=1,2,\dots,n-1$.

Second, we choose $B_{A}$ large, depending on $A$, so that  
\[
\Psi^{l}\leq0\quad \text{on}\ \partial\omega,\quad \forall1\leq l\leq n-1. 
\]
Fix $l\in\{1,2,\dots,n-1\}$, it follows from \eqref{eqn6.25} that 
\begin{equation}\label{eqn6.40}
\Psi^{l}\leq \pm\tilde{\eta}^{l}\cdot\nabla(v-\widetilde{\varphi})+A\big|y'\big|^{2}-\frac{1}{2}B_{A}(v+y_{n}).
\end{equation}
Combining \eqref{eqn6.26} and \eqref{eqn6.28} yields 
\[
v+y_{n}\geq\frac{1}{2}\big|y'\big|^{2}+O\Big(\big|y'\big|^{3}\Big)\geq\frac{1}{4}\big|y'\big|^{2}\quad \text{on}\ \partial_{1}\omega, 
\] 
which implies 
\[
\Psi^{l}\leq A\big|y'\big|^{2}-\frac{1}{8}B_{A}\big|y'\big|^{2}\leq0\quad \text{on}\ \partial_{1}\omega,  
\]
by taking $B_{A}$ large. 
It follows from \eqref{eqn6.27} that  
\begin{equation}\label{eqn6.41}
v+y_{n}\geq\min\left\{\frac{1}{4}\beta^{2},\varepsilon_{0}\beta^{2}\right\}=\varepsilon_{0}\beta^{2}\quad \text{on}\ \partial_{2}\omega\cup\partial_{3}\omega. 
\end{equation}
Then using \eqref{eqn6.40} and \eqref{eqn6.41}, we have 
\begin{align*}
\Psi^{l}\leq C+A-\frac{1}{2}B_{A}\varepsilon_{0}\beta^{2}\leq0\quad \text{on}\ \partial_{2}\omega\cup\partial_{3}\omega, 
\end{align*}
by taking $B_{A}$ large, where $C$ depends on $\tilde{\rho}$ and $\sup_{\omega}|D(v-\widetilde{\varphi})|$. 

By the comparison principle, 
\[
\Psi^{l}\leq0\quad \text{in}\ \overline{\omega},\quad \forall1\leq l\leq n-1. 
\]
Since $\Psi^{l}(0)=0$, we get $\Psi^{l}_{n}(0)\leq0$, which yields 
\[
|v_{nl}(0)|\leq C,\quad \forall1\leq l\leq n-1, 
\]
where $C$ depends on $\Omega$, $\|\varphi\|_{C^{3,1}(\partial\Omega)}$, and $|f^{3/(2n-2)}|_{C^{2,1}(\overline{\Omega})}$.  
For any point 
\[
T^{-1}y_{0}=x_{0}\in\partial\Omega\cap\left\{\big|x'\big|\leq r_{0}\right\}\cap\left\{|x|\leq\theta_{k}\sqrt{b_{k}}\right\}, 
\]
where $\theta_{k}>0$ is chosen small so that $|x_{0}|\leq\theta_{i}\sqrt{b_{i}}/2$ for all $k+1\leq i\leq n-1$, let $v(y_{0})=0$, $\nabla v(y_{0})=0$ by subtracting an affine function, and let $y_{0}$ be the origin, $\tilde{\nu}(y_{0})=e_{n}$, $\tilde{\eta}^{l}(y_{0})=e_{l}$ for all $1\leq l\leq n-1$ by translation and rotation. 
With similar arguments as above, we derive 
\[
\left|v_{\tilde{\nu}\tilde{\eta}^{l}}(y_{0})\right|\leq C,\quad \forall1\leq  l\leq n-1,  
\]
where $C$ depends on $\Omega$, $\|\varphi\|_{C^{3,1}(\partial\Omega)}$, and $|f^{3/(2n-2)}|_{C^{2,1}(\overline{\Omega})}$.
Pulling back to the $x$-coordinates, we obtain from \eqref{eqn9.9} that 
\begin{align*}
C\geq M^{2}\left|\tilde{\nu}'D^{2}_{y}u(y_{0})\tilde{\eta}^{l}\right|&=M^{2}\left|\big(T^{-1}\tilde{\nu}\big)'D^{2}_{x}u(x_{0})T^{-1}\tilde{\eta}^{l}\right|\\
&=M^{2}\left|u_{\nu\eta^{l}}(x_{0})\right|\frac{1}{|T\nu|}\frac{1}{M_{l}} 
\end{align*}
for $l=1,2,\dots,n-1$, where $\tilde{\nu}=T\nu/|T\nu|$.  
Thus by \eqref{eqn6.42}, we have 
\[
\left|u_{\nu\eta^{l}}(x_{0})\right|\leq C\frac{|T\nu|M_{l}}{M^{2}}\leq C\frac{1}{\sqrt{M}}=C\sqrt{b_{k}},\quad \forall1\leq l\leq k, 
\]
where $C$ depends on $\Omega$, $\|\varphi\|_{C^{3,1}(\partial\Omega)}$, and $|f^{3/(2n-2)}|_{C^{2,1}(\overline{\Omega})}$.
This completes the proof of Lemma \ref{thm6.6}. 
\end{proof}

\begin{theorem}\label{thm6.7}
Let $u\in C^{3,1}(\overline{\Omega})$ be a convex solution of equation \eqref{eqn1.1} with $k=n$, 
$\partial\Omega\in C^{3,1}$ be uniformly convex, 
$\varphi\in C^{3,1}(\partial\Omega)$, $\inf_{\Omega}f>0$, and $f^{3/(2n-2)}\in C^{2,1}(\overline{\Omega})$.   
Then 
\[
\sup_{\partial\Omega}|u_{\nu\nu}|\leq C,
\]
where $C$ depends on $\Omega$, $\|\varphi\|_{C^{3,1}(\partial\Omega)}$, and  $\|f^{3/(2n-2)}\|_{C^{2,1}(\overline{\Omega})}$, but is independent of $\inf_{\Omega}f$. 
\end{theorem}
\begin{proof}
At the origin, we have 
\[
u_{nn}(0)\prod_{i=1}^{n-1}b_{i}-\sum_{j=1}^{n-1}\bigg(\prod_{i\neq j}b_{i}\bigg)u_{nj}^{2}(0)=f(0). 
\]
Lemma \ref{thm6.4} ensures that Assumption \ref{thm6.5} holds for $k=n-2$. 
Thus using Lemma \ref{thm6.6}, by induction we obtain 
\begin{equation}\label{eqn6.45}
u_{ni}(0)\leq C\sqrt{b_{i}}\quad \text{for}\ i=1,2,\dots,n-2.
\end{equation}
Then it follows from Lemma \ref{thm6.3}, Lemma \ref{thm6.4}, and \eqref{eqn6.45} that 
\[
0\leq u_{nn}(0)\leq\frac{1}{\prod_{i=1}^{n-1}b_{i}}f(0)+\sum_{j=1}^{n-1}\frac{1}{b_{j}}u_{nj}^{2}(0)\leq C, 
\]
where $C$ depends on $\Omega$, $\|\varphi\|_{C^{3,1}(\partial\Omega)}$, and  $\|f^{3/(2n-2)}\|_{C^{2,1}(\overline{\Omega})}$. 
This completes the proof of Theorem \ref{thm6.7}. 
\end{proof}
\par
\vspace{2mm}

Combining Lemmas \ref{thm7.1}, \ref{thm7.2}, \ref{thm1.2}, and Theorem \ref{thm6.7}, we obtain the $C^{2}$ estimate independent of $\inf_{\Omega_{0}}f$, and then Theorem \ref{thm10.1} follows from an approximation argument.

\section{The weakly interior estimate}

We recall the following comparison principle established by Ivochkina-Trudinger-Wang \cite{Wang2004}, and provide a proof for completeness.

\begin{lemma}[Lemma 3.1 of \cite{Wang2004}]\label{thm3.1}
Let $w(x,y)$, $v(x,y)$ be functions in $\overline{\Omega}\times\mathbb{R}^{N}$ that are both $p$-homogeneous in $y\in\mathbb{R}^{N}$ for some $p\in\mathbb{R}$.
Let $w$, $v\in C^{2}(\Omega\times\mathbb{S}^{N-1})\cap C^{0}(\overline{\Omega}\times\mathbb{S}^{N-1})$ with $v>0$ in $\overline{\Omega}\times\mathbb{S}^{N-1}$.
Define the linear operator $\mathscr{L}$ by
\[
\mathscr{L}(w):=\operatorname{tr}\big(\widetilde{G}D^{2}w\big)+\langle b,Dw\rangle,
\]
where $\widetilde{G}=\widetilde{G}(x,y)$ is an $(n+N)\times(n+N)$ positive semi-definite matrix and $b=b(x,y)$ is an $(n+N)$-dimensional vector.
If there exist two positive constants $\mu_{1}$, $\mu_{2}$ such that
\begin{equation}\label{eqn3.1}
\begin{cases}
\mathscr{L}(w)\geq-\mu_{1}\quad \text{in}\ \Omega\times\mathbb{S}^{N-1},\\
\mathscr{L}(v)\leq-\mu_{2}\quad \text{in}\ \Omega\times\mathbb{S}^{N-1},
\end{cases}
\end{equation}
then 
\[
\sup_{\overline{\Omega}\times\mathbb{S}^{N-1}}\frac{w}{v}\leq\frac{\mu_{1}}{\mu_{2}}+\sup_{\partial\Omega\times\mathbb{S}^{N-1}}\frac{w}{v}.
\]
\end{lemma}
\begin{proof}
Since $w$, $v$ are $p$-homogeneous in $y\in\mathbb{R}^{N}$, we have 
\[
\frac{w}{v}(x,y)\in C^{2}(Q)\cap C^{0}(\overline{Q})\quad \text{and}\quad \frac{w}{v}(x_{0},y_{0})=\sup_{\overline{Q}}\frac{w}{v},
\]
where $Q:=\Omega\times(B_{2}\setminus B_{1/2})$ and  $(x_{0},y_{0})\in\overline{\Omega}\times\mathbb{S}^{N-1}$. 
We only need to consider the case of $(x_{0},y_{0})\in\Omega\times\mathbb{S}^{N-1}$ with $w(x_{0},y_{0})>0$.
At the maximum point $(x_{0},y_{0})$,
\begin{equation}\label{eqn3.2}
\partial_{i}\Big(\frac{w}{v}\Big)=\frac{\partial_{i}w}{v}-\frac{w\partial_{i}v}{v^{2}}=0,\quad \forall1\leq i\leq n+N,
\end{equation}
and the Hessian 
\[
\partial_{ij}\Big(\frac{w}{v}\Big)=\frac{\partial_{ij}w}{v}-\frac{\partial_{i}w\partial_{j}v+\partial_{j}w\partial_{i}v}{v^{2}}+\frac{2w\partial_{i}v\partial_{j}v}{v^{3}}-\frac{w\partial_{ij}v}{v^{2}} 
\]
is an $(n+N)\times(n+N)$ negative semi-definite matrix. 
Then at the point $(x_{0},y_{0})$,
\begin{align*}
0\geq\widetilde{G}^{ij}\partial_{ij}\Big(\frac{w}{v}\Big)&=\widetilde{G}^{ij}\left(\frac{\partial_{ij}w}{v}-\frac{w\partial_{ij}v}{v^{2}}+\left(\frac{2w\partial_{i}v}{v^{3}}-\frac{2\partial_{i}w}{v^{2}}\right)\partial_{j}v\right)\\
&\overset{\eqref{eqn3.2}}{=}\widetilde{G}^{ij}\left(\frac{\partial_{ij}w}{v}-\frac{w\partial_{ij}v}{v^{2}}\right)\\
&=\frac{\mathscr{L}(w)-b^{i}\partial_{i}w}{v}-\frac{w}{v^{2}}\left(\mathscr{L}(v)-b^{i}\partial_{i}v\right)\\
&\overset{\eqref{eqn3.2}}{=}\frac{1}{v}\mathscr{L}(w)-\frac{w}{v^{2}}\mathscr{L}(v)\\
&\overset{w,v>0}{\geq}\frac{1}{v}\left(-\mu_{1}+\frac{w}{v}\mu_{2}\right).
\end{align*}
This yields 
\[
\frac{w}{v}(x_{0},y_{0})\leq\frac{\mu_{1}}{\mu_{2}}.
\]
The proof of Lemma \ref{thm3.1} is complete. 
\end{proof}
\par
\vspace{2mm}

We now proceed to prove the \textit{weakly interior estimate}. 
\par
\vspace{2mm}

\begin{proof}[Proof of Theorem \ref{thm1.1}]
We will apply Lemma \ref{thm3.1} with $N=n+1$ and coordinates  $(x,\xi,\eta)\in\overline{\Omega}\times\mathbb{R}^{n}\times\mathbb{R}$.  
Define functions
\[
w:=u_{\xi\xi}+2\eta u_{\xi}+\eta^{2}u
\]
and
\[
v:=\frac{1}{4\alpha}(\psi+\beta)^{1-\alpha}|\xi|^{2}+\frac{\psi_{\xi}^{2}}{(\psi+\beta)^{\alpha}}+\varrho(\psi+\beta)\eta^{2},
\]
where $0<\alpha<1$, $0<\varrho<1$, and $0<\beta\ll1$ are parameters to be determined; and the function $\psi$ is given in Lemma \ref{thm2.4}.
Define a linear operator
\[
\mathscr{L}(w):=\operatorname{tr}\big(\widetilde{G}D^{2}w\big)+\langle b,Dw\rangle,
\]
where 
\begin{equation*}
\widetilde{G}:=
\begin{pmatrix}
G &rG &Gq\\
rG &r^{2}G &rGq\\
q'G &rq'G &q'Gq
\end{pmatrix}
,\quad b:=
\begin{pmatrix}
0_{n\times1}\\
-2Gq\\
0_{1\times1}
\end{pmatrix}
,
\end{equation*}
with $G$ being the $n\times n$ matrix-valued function defined in \eqref{eqn2.6}. 
The scalar function $r$ and the $n$-dimensional vector-valued function $q$ are defined by
\begin{equation}\label{eqn3.14}
r:=\frac{\alpha\psi_{\xi}}{\psi+\beta},\quad  q:=\frac{1}{\psi+\beta}\left(\frac{\xi}{4}+\frac{\alpha\psi_{\xi}}{\psi+\beta}\nabla\psi\right).
\end{equation}
It holds that the quadratic form  
\[
\left(\xi_{1},\xi_{2},\eta\right)^{T}\widetilde{G}\left(\xi_{1},\xi_{2},\eta\right)
=\left(\xi_{1}+r\xi_{2}+\eta q\right)^{T}G\left(\xi_{1}+r\xi_{2}+\eta q\right)\geq0
\]
is valid for all $(\xi_{1},\xi_{2},\eta)\in\mathbb{R}^{n}\times\mathbb{R}^{n}\times\mathbb{R}$, 
which implies that $\widetilde{G}$ is positive semi-definite. 
Due to the substantial computations involved, we divide the proof into five steps for clarity.

\par
\vspace{2mm}
\textbf{Step 1.}
For convenience, denote
\[
\hat{v}:=\frac{1}{4\alpha}(\psi+\beta)^{1-\alpha}|\xi|^{2}+\frac{\psi_{\xi}^{2}}{(\psi+\beta)^{\alpha}}.
\]
We claim that 
\begin{equation}\label{eqn3.3}
\mathscr{L}(\hat{v})\leq -\frac{1}{C_{\psi}}\frac{1}{8\alpha}|\xi|^{2} 
\end{equation}
in $\Omega\times\mathbb{R}^{n}\times\mathbb{R}$ for all $\alpha\in(0,\alpha_{0})$, where constants $C_{\psi}$ and $\alpha_{0}$ depend on  $\|\psi\|_{C^{3}(\overline{\Omega})}$ to be determined. 
It holds that 
\[
\widetilde{G}D^{2}\hat{v}=
\begin{pmatrix}
G &rG &Gq\\
rG &r^{2}G &rGq\\
q'G &rq'G &q'Gq
\end{pmatrix}
\begin{pmatrix}
\partial_{x_{i}x_{j}} &\partial_{x_{i}\xi_{j}} &0\\
\partial_{\xi_{i}x_{j}} &\partial_{\xi_{i}\xi_{j}} &0\\
0 &0 &0
\end{pmatrix}
[\hat{v}],
\]
where $1\leq i,\ j\leq n$. 
Since $G$ is symmetric, we derive 
\begin{equation}\label{eqn3.4}
\mathscr{L}(\hat{v})=G^{ij}\partial_{x_{i}x_{j}}\hat{v}+2rG^{ij}\partial_{x_{i}\xi_{j}}\hat{v}+r^{2}G^{ij}\partial_{\xi_{i}\xi_{j}}\hat{v}-2G^{ij}q_{j}\partial_{\xi_{i}}\hat{v}.
\end{equation}
Compute derivatives 
\begin{align}
\partial_{x_{i}}\hat{v}&
=\frac{1}{(\psi+\beta)^{\alpha}}\left(\frac{1-\alpha}{4\alpha}|\xi|^{2}\psi_{x_{i}}+2\psi_{\xi}\psi_{\xi x_{i}}-\alpha\frac{\psi_{\xi}^{2}\psi_{x_{i}}}{\psi+\beta}\right);\nonumber\\
\partial_{x_{i}x_{j}}\hat{v}&
=\frac{1}{(\psi+\beta)^{\alpha}}\left(\frac{1-\alpha}{4\alpha}|\xi|^{2}\psi_{x_{i}x_{j}}+2\psi_{\xi x_{i}}\psi_{\xi x_{j}}+2\psi_{\xi}\psi_{\xi x_{i}x_{j}}\right)\label{eqn3.5}\\
&\hspace{4.5mm}+\frac{1}{(\psi+\beta)^{\alpha+1}}\left(\frac{\alpha-1}{4}|\xi|^{2}\psi_{x_{i}}\psi_{x_{j}}-2\alpha\psi_{\xi}\psi_{\xi x_{i}}\psi_{x_{j}}\right)\nonumber\\
&\hspace{4.5mm}+\frac{1}{(\psi+\beta)^{\alpha+1}}\left(-2\alpha\psi_{\xi}\psi_{x_{i}}\psi_{\xi x_{j}}-\alpha\psi_{\xi}^{2}\psi_{x_{i}x_{j}}\right)\nonumber\\
&\hspace{4.5mm}+\frac{1}{(\psi+\beta)^{\alpha+2}}\big(\alpha^{2}+\alpha\big)\psi_{\xi}^{2}\psi_{x_{i}}\psi_{x_{j}};\nonumber\\
r\partial_{x_{i}\xi_{j}}\hat{v}&
=\frac{1}{(\psi+\beta)^{\alpha+1}}\left(\frac{1-\alpha}{2}\psi_{\xi}\psi_{x_{i}}\xi_{j}+2\alpha\psi_{\xi}\psi_{\xi x_{i}}\psi_{x_{j}}+2\alpha\psi_{\xi}^{2}\psi_{x_{i}x_{j}}\right)\label{eqn3.6}\\
&\hspace{4.5mm}-\frac{1}{(\psi+\beta)^{\alpha+2}}2\alpha^{2}\psi_{\xi}^{2}\psi_{x_{i}}\psi_{x_{j}};\nonumber\\
\partial_{\xi_{i}}\hat{v}&
=\frac{(\psi+\beta)^{1-\alpha}}{2\alpha}\xi_{i}+(\psi+\beta)^{-\alpha}2\psi_{\xi}\psi_{x_{i}};\nonumber\\
q_{j}\partial_{\xi_{i}}\hat{v}&
=\frac{1}{(\psi+\beta)^{\alpha}}\frac{\xi_{i}\xi_{j}}{8\alpha}+\frac{1}{(\psi+\beta)^{\alpha+1}}\left(\frac{1}{2}\psi_{\xi}\xi_{i}\psi_{x_{j}}+\frac{1}{2}\psi_{\xi}\psi_{x_{i}}\xi_{j}\right)\label{eqn3.7}\\
&\hspace{4.5mm}+\frac{1}{(\psi+\beta)^{\alpha+2}}2\alpha\psi_{\xi}^{2}\psi_{x_{i}}\psi_{x_{j}};\nonumber\\
r^{2}\partial_{\xi_{i}\xi_{j}}\hat{v}&
=\frac{1}{(\psi+\beta)^{\alpha+1}}\frac{1}{2}\alpha\psi_{\xi}^{2}\delta_{ij}
+\frac{1}{(\psi+\beta)^{\alpha+2}}2\alpha^{2}\psi_{\xi}^{2}\psi_{x_{i}}\psi_{x_{j}}.\label{eqn3.8}
\end{align}
Combining \eqref{eqn3.4}--\eqref{eqn3.8} with the symmetry of $G$ yields that 
\begin{align*}
\mathscr{L}(\hat{v})&
=\frac{G^{ij}}{(\psi+\beta)^{\alpha}}\left(\frac{1-\alpha}{4\alpha}|\xi|^{2}\psi_{x_{i}x_{j}}+2\psi_{\xi x_{i}}\psi_{\xi x_{j}}+2\psi_{\xi}\psi_{\xi x_{i}x_{j}}-\frac{1}{4\alpha}\xi_{i}\xi_{j}\right)\\
&\hspace{4.5mm}+\frac{G^{ij}}{(\psi+\beta)^{\alpha+1}}\left(\frac{\alpha-1}{4}|\xi|^{2}\psi_{x_{i}}\psi_{x_{j}}-(1+\alpha)\psi_{\xi}\psi_{x_{i}}\xi_{j}\right)\\
&\hspace{4.5mm}+\frac{G^{ij}}{(\psi+\beta)^{\alpha+1}}\left(3\alpha\psi_{\xi}^{2}\psi_{x_{i}x_{j}}+ \frac{1}{2}\alpha\psi_{\xi}^{2}\delta_{ij} \right)\\
&\hspace{4.5mm}-\frac{G^{ij}}{(\psi+\beta)^{\alpha+2}}\big(\alpha^{2}+3\alpha\big)\psi_{\xi}^{2}\psi_{x_{i}}\psi_{x_{j}}.
\end{align*}
Since $\operatorname{tr}\left(GD^{2}\psi\right)\leq-1$ from \eqref{eqn2.15},  $\operatorname{tr}G=1$, and $G$ is positive definite, we get that
\begin{align}\label{eqn3.9}
\mathscr{L}(\hat{v})&
\leq\frac{1}{(\psi+\beta)^{\alpha}}\left(-\frac{1-\alpha}{4\alpha}|\xi|^{2}+
C_{1}\|\psi\|_{C^{3}(\overline{\Omega})}^{2}|\xi|^{2}
-\frac{1}{4\alpha}G^{ij}\xi_{i}\xi_{j}\right)\nonumber\\
&\hspace{4.5mm}+\frac{1}{(\psi+\beta)^{\alpha+1}}\left(\frac{\alpha-1}{4}|\xi|^{2}G^{ij}\psi_{x_{i}}\psi_{x_{j}}-(1+\alpha)\psi_{\xi}G^{ij}\psi_{x_{i}}\xi_{j}-\frac{5}{2}\alpha\psi_{\xi}^{2} \right)\nonumber\\
&\hspace{4.5mm}-\frac{G^{ij}}{(\psi+\beta)^{\alpha+2}}\big(\alpha^{2}+3\alpha\big)\psi_{\xi}^{2}\psi_{x_{i}}\psi_{x_{j}}\nonumber\\
&\leq\frac{1}{(\psi+\beta)^{\alpha}}\left(-\frac{1-\alpha}{4\alpha}+
C_{1}\|\psi\|_{C^{3}(\overline{\Omega})}^{2}\right)|\xi|^{2}\\
&\hspace{4.5mm}+\frac{G^{ij}}{(\psi+\beta)^{\alpha}}\left[-\frac{1}{4\alpha}\xi_{i}\xi_{j}-\frac{1+\alpha}{\psi+\beta}\psi_{\xi}\psi_{x_{i}}\xi_{j}-\frac{\alpha^{2}+3\alpha}{(\psi+\beta)^{2}}\psi_{\xi}^{2}\psi_{x_{i}}\psi_{x_{j}}\right],\nonumber
\end{align}
where the constant $C_{1}>0$ is universal.
It holds that
\begin{align}
0&\leq\sum_{s=1}^{n}\left[\sum_{i=1}^{n}\Big(G^{1/2}\Big)^{si}\left(\sqrt{\alpha^{2}+3\alpha}\frac{\psi_{\xi}}{\psi+\beta}\psi_{x_{i}}+\frac{1+\alpha}{2\sqrt{\alpha^{2}+3\alpha}}\xi_{i}\right)\right]^{2}\nonumber\\
&=\sum_{i,j=1}^{n}G^{ij}\left[\frac{\alpha^{2}+3\alpha}{(\psi+\beta)^{2}}\psi_{\xi}^{2}\psi_{x_{i}}\psi_{x_{j}}+\frac{1+\alpha}{\psi+\beta}\psi_{\xi}\psi_{x_{i}}\xi_{j}+\frac{(1+\alpha)^{2}}{4\alpha(3+\alpha)}\xi_{i}\xi_{j}\right].\label{eqn3.10}
\end{align}
By \eqref{eqn3.9} and \eqref{eqn3.10}, we obtain that 
\begin{align*}
\mathscr{L}(\hat{v})&\leq\frac{1}{(\psi+\beta)^{\alpha}}\left(-\frac{1-\alpha}{4\alpha}+ C_{\psi}\right)|\xi|^{2} +\frac{G^{ij}\xi_{i}\xi_{j}}{(\psi+\beta)^{\alpha}}\left[-\frac{1}{4\alpha}+\frac{(1+\alpha)^{2}}{4\alpha(3+\alpha)}\right]\\
&\leq\frac{1}{(\psi+\beta)^{\alpha}}\left(-\frac{1-\alpha}{4\alpha}+ C_{\psi}\right)|\xi|^{2}\\
&\leq -\frac{1}{\psi+1}\frac{1}{8\alpha}|\xi|^{2} 
\end{align*}
for all $\alpha\in(0,\alpha_{0})$ with $\alpha_{0}>0$ chosen sufficiently small, where constants $C_{\psi}$ and $\alpha_{0}$ depend on  $\|\psi\|_{C^{3}(\overline{\Omega})}$. 
This establishes claim \eqref{eqn3.3}.

\par
\vspace{2mm}
\textbf{Step 2.}
It holds that 
\[
\widetilde{G}D^{2}\left[(\psi+\beta)\eta^{2}\right]=
\begin{pmatrix}
G &rG &Gq\\
rG &r^{2}G &rGq\\
q'G &rq'G &q'Gq
\end{pmatrix}
\begin{pmatrix}
\partial_{x_{i}x_{j}} &0 &\partial_{x_{i}\eta}\\
0 &0 &0\\
\partial_{\eta x_{j}} &0 &\partial_{\eta\eta}
\end{pmatrix}
\left[(\psi+\beta)\eta^{2}\right],
\]
where $1\leq i,\ j\leq n$.
By the symmetry of $G$, we compute  
\begin{align*}
&\hspace{4.5mm}\mathscr{L}\big((\psi+\beta)\eta^{2}\big)\\
&=G^{ij}\left(\partial_{x_{i}x_{j}}+2q_{j}\partial_{\eta x_{i}}+q_{i}q_{j}\partial_{\eta\eta}\right)\left[(\psi+\beta)\eta^{2}\right]\\
&=\eta^{2}G^{ij}\psi_{x_{i}x_{j}}+ G^{ij}\left(\frac{1}{\psi+\beta}\eta\psi_{x_{i}}\xi_{j}+ \frac{1}{(\psi+\beta)^{2}}4\alpha\eta\psi_{\xi}\psi_{x_{i}}\psi_{x_{j}}\right)\\
&\hspace{4.5mm}+G^{ij}\left(\frac{1}{\psi+\beta}\frac{1}{8}\xi_{i}\xi_{j}+ \frac{1}{(\psi+\beta)^{2}}\alpha\psi_{\xi}\psi_{x_{i}}\xi_{j}+ \frac{1}{(\psi+\beta)^{3}}2\alpha^{2}\psi_{\xi}^{2}\psi_{x_{i}}\psi_{x_{j}}\right).
\end{align*}
Since $\operatorname{tr}\left(GD^{2}\psi\right)\leq-1$ from \eqref{eqn2.15},  $\operatorname{tr}G=1$, and $G$ is positive definite, we get that 
\begin{align}
\mathscr{L}\big((\psi+\beta)\eta^{2}\big)&\leq -\eta^{2}+\frac{1}{\beta^{2}}C_{\psi}|\eta||\xi|+\frac{1}{\beta^{3}}C_{\psi}|\xi|^{2}\nonumber\\
&\leq-\frac{1}{2}\eta^{2}+\frac{1}{2\beta^{4}}C_{\psi}^{2}|\xi|^{2}+\frac{1}{\beta^{3}}C_{\psi}|\xi|^{2}\nonumber\\
&\leq -\frac{1}{2}\eta^{2}+\frac{1}{\beta^{4}}C_{\psi}^{2}|\xi|^{2},\label{eqn3.11}
\end{align}
where the constant $C_{\psi}>0$ depends on $\|\psi\|_{C^{3}(\overline{\Omega})}$. 
Then it follows from \eqref{eqn3.3} and \eqref{eqn3.11} that
\begin{align}
\mathscr{L}(v)&\leq-\frac{1}{C_{\psi}}\frac{1}{8\alpha}|\xi|^{2}+\varrho\left(-\frac{1}{2}\eta^{2}+\frac{1}{\beta^{4}}C_{\psi}^{2}|\xi|^{2}\right)\nonumber\\
&= \left(-\frac{1}{8C_{\psi}}\frac{1}{\alpha}+C_{\psi}^{2}\right)|\xi|^{2}-\frac{1}{2}\beta^{4}\eta^{2}\nonumber\\
&\leq-\frac{1}{2}\beta^{4}\left(|\xi|^{2}+\eta^{2}\right) \label{eqn3.12}
\end{align}
in $\Omega\times\mathbb{R}^{n}\times\mathbb{R}$ for all $\alpha\in(0,\alpha_{0})$, where we set 
\[
\varrho:=\beta^{4}
\]
and $\alpha_{0}>0$ chosen sufficiently small depending on  $\|\psi\|_{C^{3}(\overline{\Omega})}$.  
Fix $\alpha\in(0,\alpha_{0})$ depending on $\|\psi\|_{C^{3}(\overline{\Omega})}$ in subsequent arguments. 

\par
\vspace{2mm}
\textbf{Step 3.}
It holds that 
\begin{align*}
\widetilde{G}D^{2}w=
\begin{pmatrix}
G &rG &Gq\\
rG &r^{2}G &rGq\\
q'G &rq'G &q'Gq
\end{pmatrix}
\begin{pmatrix}
\partial_{x_{i}x_{j}} &\partial_{x_{i}\xi_{j}} &\partial_{x_{i}\eta}\\
\partial_{\xi_{i}x_{j}} &\partial_{\xi_{i}\xi_{j}} &\partial_{\xi_{i}\eta}\\
\partial_{\eta x_{j}} &\partial_{\eta\xi_{j}} &\partial_{\eta\eta}
\end{pmatrix}
[w], 
\end{align*}
where $1\leq i,\ j\leq n$.
By the symmetry of $G$, we compute 
\begin{align*}
&\hspace{4.5mm}\mathscr{L}(w)\\
&=G^{ij}\left(\partial_{x_{i}x_{j}}+2r\partial_{x_{i}\xi_{j}}+2q_{j}\partial_{x_{i}\eta}+ r^{2}\partial_{\xi_{i}\xi_{j}}+2rq_{j}\partial_{\xi_{i}\eta}+ q_{i}q_{j}\partial_{\eta\eta}- 2q_{j}\partial_{\xi_{i}} \right)[w]\\ 
&=G^{ij}2q_{j}(\partial_{x_{i}\eta}w-\partial_{\xi_{i}}w)
+G^{ij}\big(u_{\xi\xi x_{i}x_{j}}+2\eta u_{\xi x_{i}x_{j}}+\eta^{2}u_{x_{i}x_{j}}\big)\\
&\hspace{4.5mm}+G^{ij}2r\big(2u_{\xi x_{i}x_{j}}+2\eta u_{x_{i}x_{j}}\big)+G^{ij}r^{2}2u_{x_{i}x_{j}}+G^{ij}2rq_{j}2u_{x_{i}}+G^{ij}q_{i}q_{j}2u\\
&=G^{ij}\big(u_{\xi\xi x_{i}x_{j}}+2\eta u_{\xi x_{i}x_{j}}+\eta^{2}u_{x_{i}x_{j}}\big)
+4rG^{ij}\big(u_{\xi x_{i}x_{j}}+\eta u_{x_{i}x_{j}}\big)\\
&\hspace{4.5mm}+2r^{2}G^{ij}u_{x_{i}x_{j}}+4rG^{ij}u_{x_{i}}q_{j}+2uG^{ij}q_{i}q_{j}.
\end{align*}
It follows from \eqref{eqn2.11}--\eqref{eqn2.14}, \eqref{eqn3.14}, and $0\leq G\leq\mathbf{I}_{n}$ that 
\begin{align}
\mathscr{L}(w)&\geq -K_{1}|\xi|^{2}\sigma_{1}^{\frac{-1}{k-1}}-2|\eta|K_{1}|\xi|\sigma_{1}^{\frac{-1}{k-1}}-4|r|K_{1}|\xi|\sigma_{1}^{\frac{-1}{k-1}}-4|r||\eta|K_{1}\nonumber\\
&\hspace{4.5mm}-4|r|\|u\|_{C^{1}(\overline{\Omega})}|q|-2\|u\|_{C^{0}(\overline{\Omega})}|q|^{2}\nonumber\\
&\geq -K_{1}\sigma_{1}^{\frac{-1}{k-1}}\left(6|\xi|^{2}+\eta^{2}+r^{2}\right)-K_{1}\left(4\eta^{2}+r^{2}\right)-4r^{2}\nonumber\\
&\hspace{4.5mm}-\left(\|u\|_{C^{1}(\overline{\Omega})}^{2}+2\|u\|_{C^{0}(\overline{\Omega})}\right)|q|^{2}\nonumber\\
&\geq -K_{1}\sigma_{1}^{\frac{-1}{k-1}}\left(6|\xi|^{2}+\eta^{2}+\frac{C_{\psi}}{\beta^{2}}|\xi|^{2}\right)-K_{1}\left(4\eta^{2}+\frac{C_{\psi}}{\beta^{2}}|\xi|^{2}\right)\nonumber\\
&\hspace{4.5mm}-4\frac{C_{\psi}}{\beta^{2}}|\xi|^{2}-\left(\|u\|_{C^{1}(\overline{\Omega})}^{2}+2\|u\|_{C^{0}(\overline{\Omega})}\right)\frac{C_{\psi}}{\beta^{4}}|\xi|^{2}\nonumber\\
&\geq-K\left(1+\sigma_{1}^{\frac{-1}{k-1}}\right)\frac{1}{\beta^{4}}\left(|\xi|^{2}+\eta^{2}\right),\label{eqn3.13}
\end{align}
where $K>0$ depends on $\|\psi\|_{C^{1}(\overline{\Omega})}$, $\|u\|_{C^{1}(\overline{\Omega})}$,  $\operatorname{dist}(\Omega,\partial\Omega_{0})$, and either $\|f^{1/(k-1)}\|_{C^{1,1}(\overline{\Omega_{0}})}$ or  $\|f^{3/(2k-2)}\|_{C^{2,1}(\overline{\Omega_{0}})}$, but is independent of $\inf_{\Omega_{0}}f$. 

\par
\vspace{2mm}
\textbf{Step 4.}
It holds that 
\begin{align}
w&=\left(\theta\nu+\tau\sqrt{|\xi|^{2}-\theta^{2}}\right)^{T}D^{2}u\left(\theta\nu+\tau\sqrt{|\xi|^{2}-\theta^{2}}\right)+2\eta u_{\xi}+\eta^{2}u\nonumber\\
&\leq\theta^{2}\sup_{\partial\Omega}u_{\nu\nu}+C\label{eqn3.15}
\end{align}
for all $(x,\xi,\eta)\in\partial\Omega\times\mathbb{S}^{n}$, where $\tau$ is a unit tangential direction, $\theta:=\langle \xi,\nu\rangle$, and the constant $C>0$ depends on $\|u\|_{C^{1}(\overline{\Omega})}$, second-order tangential and tangential-normal  derivatives of $u$ on $\partial\Omega$. 
Moreover, by \eqref{eqn2.15} we get that 
\begin{align}
v&=\frac{1}{4\alpha}(\psi+\beta)^{1-\alpha}|\xi|^{2}+\frac{\psi_{\xi}^{2}}{(\psi+\beta)^{\alpha}}+\beta^{4}(\psi+\beta)\eta^{2}\nonumber\\
&=\frac{\beta^{1-\alpha}}{4\alpha}|\xi|^{2}+\frac{1}{\beta^{\alpha}}(\left|\nabla\psi\right|\langle \nu,\xi\rangle)^{2}+\beta^{5}\eta^{2}\nonumber\\
&\geq\frac{1}{\beta^{\alpha}}\theta^{2}+\beta^{5}\big(|\xi|^{2}+\eta^{2}\big)\nonumber\\
&=\frac{1}{\beta^{\alpha}}\theta^{2}+\beta^{5}\label{eqn3.16}
\end{align}
for all $(x,\xi,\eta)\in\partial\Omega\times\mathbb{S}^{n}$.
Therefore, it follows from \eqref{eqn3.15} and \eqref{eqn3.16} that 
\[
\frac{w}{v}\leq \frac{\theta^{2}\sup_{\partial\Omega}u_{\nu\nu}+C}{\theta^{2}\beta^{-\alpha}+\beta^{5}}\leq \beta^{\alpha}\sup_{\partial\Omega}u_{\nu\nu}+\frac{C}{\beta^{5}}
\]
for all $(x,\xi,\eta)\in\partial\Omega\times\mathbb{S}^{n}$.
Thus 
\begin{equation}\label{eqn3.17}
\sup_{\partial\Omega\times\mathbb{S}^{n}}\frac{w}{v}\leq\beta^{\alpha}\sup_{\partial\Omega}u_{\nu\nu}+\frac{C}{\beta^{5}}.
\end{equation}

\par
\vspace{2mm}
\textbf{Step 5.}
Let
\[
\frac{w}{v}(x_{*},\xi_{*},\eta_{*})=\sup_{\overline{\Omega}\times\mathbb{S}^{n}}\frac{w}{v},
\]
where $(x_{*},\xi_{*},\eta_{*})\in\overline{\Omega}\times\mathbb{S}^{n}$. 
Applying Lemma \ref{thm3.1} with \eqref{eqn3.12}, \eqref{eqn3.13}, and  \eqref{eqn3.17}, we derive 
\begin{equation}\label{eqn3.18}
\frac{w}{v}(x_{*},\xi_{*},\eta_{*})\leq\beta^{\alpha}\sup_{\partial\Omega}u_{\nu\nu}+\frac{2K+C}{\beta^{8}}\left(1+\sigma_{1}^{\frac{-1}{k-1}}\big[D^{2}u(x_{*})\big]\right). 
\end{equation}
Suppose  
\[
\frac{w}{v}(x_{*},\xi_{*},\eta_{*})>\frac{1}{\beta^{8}},
\]
then we have at the point $(x_{*},\xi_{*},\eta_{*})$, 
\begin{align*}
u_{\xi\xi}+2\eta u_{\xi}+\eta^{2}u&\geq \frac{1}{\beta^{8}}\left(\frac{1}{4\alpha}(\psi+\beta)^{1-\alpha}|\xi|^{2}+\frac{\psi_{\xi}^{2}}{(\psi+\beta)^{\alpha}}+\beta^{4}(\psi+\beta)\eta^{2}\right)\\
&\geq \frac{1}{\beta^{8}}\left(\frac{1}{4}\beta|\xi|^{2}+\beta^{5}\eta^{2}\right)\\
&\geq \frac{1}{\beta^{3}}\left(|\xi|^{2}+\eta^{2}\right)\\
&=\frac{1}{\beta^{3}}. 
\end{align*}
Since $\lambda(D^{2}u(x_{*}))\in\Gamma_{2}$, one gets at the point $x_{*}$, 
\begin{equation}\label{eqn3.19}
\sigma_{1}\big[\lambda\big(D^{2}u\big)\big]\geq\lambda_{\operatorname{max}}\big(D^{2}u\big)\geq u_{\xi_{*}\xi_{*}}\geq \frac{1}{\beta^{3}}-3\|u\|_{C^{1}(\overline{\Omega})}>1, 
\end{equation}
where $0<\beta\ll1$.
Combining \eqref{eqn3.18} and \eqref{eqn3.19} yields  
\[
\sup_{\overline{\Omega}\times\mathbb{S}^{n}}\frac{w}{v}\leq \beta^{\alpha}\sup_{\partial\Omega}u_{\nu\nu}+\frac{4K+2C}{\beta^{8}}. 
\]
Therefore, for any $x\in\Omega$ and $\xi\in\mathbb{S}^{n-1}$, 
\begin{align}
w(x,\xi,0)&\leq v(x,\xi,0)\sup_{\overline{\Omega}\times\mathbb{S}^{n}}\frac{w}{v}\nonumber\\
&=\left(\frac{1}{4\alpha}(\psi+\beta)^{1-\alpha}+\frac{\psi_{\xi}^{2}}{(\psi+\beta)^{\alpha}}\right) \left(\beta^{\alpha}\sup_{\partial\Omega}u_{\nu\nu}+ \frac{4K+2C}{\beta^{8}}\right)\nonumber\\
&\leq \frac{C_{\psi}}{\psi^{\alpha}}\left(\beta^{\alpha}\sup_{\partial\Omega}u_{\nu\nu}+ \frac{4K+2C}{\beta^{8}}\right),\label{eqn3.20}
\end{align}
where the constant $C_{\psi}>0$ depends on $\|\psi\|_{C^{3}(\overline{\Omega})}$. 
From the construction of $\psi^{*}=-t^{-1}\psi$ in Lemma \ref{thm2.4}, there exist  small positive constants $\delta$ and $\delta^{*}$ such that
\begin{equation}\label{eqn3.21}
\psi^{*}\leq-\delta^{*}d(x)\quad \text{in}\ \Omega_{\delta}, 
\end{equation}
and since $\psi^{*}$ is subharmonic, 
\begin{equation}\label{eqn3.22}
\psi^{*}\leq-\delta^{*}\delta\quad \text{in}\ \Omega\setminus\Omega_{\delta}, 
\end{equation}
where $t$, $\delta$, $\delta^{*}$ depend on $\Omega$.  
It follows from \eqref{eqn3.20}--\eqref{eqn3.22} that 
\[
u_{\xi\xi}(x)\leq C_{*}\frac{1}{d(x)}\left(\beta^{\alpha}\sup_{\partial\Omega}u_{\nu\nu}+ \frac{1}{\beta^{8}}\right),\quad \forall0<\beta\ll1,
\]
for all $x\in\Omega$ and $\xi\in\mathbb{S}^{n-1}$, where the constant $C_{*}>0$ depends on $\Omega$, $\|u\|_{C^{1}(\overline{\Omega})}$, second-order tangential and tangential-normal  derivatives of $u$ on $\partial\Omega$,  $\operatorname{dist}(\Omega,\partial\Omega_{0})$, and either $\|f^{1/(k-1)}\|_{C^{1,1}(\overline{\Omega_{0}})}$ or  $\|f^{3/(2k-2)}\|_{C^{2,1}(\overline{\Omega_{0}})}$. 
Finally, taking
\[
\varepsilon:=C_{*}\beta^{\alpha}\quad \text{and}\quad C_{\varepsilon}:=C_{*}\beta^{-8},
\]
using the results in Section 3, we complete the proof of Theorem \ref{thm1.1}. 
\end{proof}

\section{The boundary estimate in terms of the interior one}

Let $A=(A_{ij})$ be a skew-symmetric matrix, which implies that the matrix  $\exp{A}$ is orthogonal. 
Define a vector field $\bar{\tau}:\ \overline{\Omega}\rightarrow\mathbb{R}^{n}$ by  
\[
\bar{\tau}(x):=Ax.
\]
We adopt notations
\[
u_{\bar{\tau}\bar{\tau}}:=u_{ij}\bar{\tau}_{i}\bar{\tau}_{j}\quad \text{and}\quad  u_{(\bar{\tau})(\bar{\tau})}:=\big(u_{\bar{\tau}}\big)_{\bar{\tau}}=u_{ij}\bar{\tau}_{i}\bar{\tau}_{j}+u_{i}\bar{\tau}_{j}\partial_{x_{j}}(\bar{\tau}_{i}).
\]
The following derivative exchange formula is due to Ivochkina-Trudinger-Wang \cite{Wang2004}.

\begin{lemma}[Lemma 2.1 of \cite{Wang2004}]\label{thm4.1}
For any function $u\in C^{4}(\overline{\Omega})$, it holds that 
\[
F^{ij}\big[D^{2}u\big]\left(u_{(\tau)(\tau)}\right)_{ij}=\left(F\big[D^{2}u\big]\right)_{(\tau)(\tau)}-F^{ij,st}\big[D^{2}u\big](u_{\tau})_{ij}(u_{\tau})_{st},
\]
where $\tau(x):=\vec{a}+\bar{\tau}(x)$ for any fixed vector $\vec{a}\in\mathbb{R}^{n}$. 
\end{lemma}

Let $u\in C^{3,1}(\overline{\Omega})$ be a $k$-admissible solution to the Dirichlet problem \eqref{eqn1.1}.
Let the origin be a boundary point satisfying $\nu(0)=e_{n}$, and $\partial\Omega$ be represented by
\begin{equation}\label{eqn4.2}
x_{n}=\rho\big(x'\big)
\end{equation}
on $\partial\Omega\cap\{|x'|\leq r_{0}\}$ for some $C^{3,1}$ function $\rho$ and constant $r_{0}>0$. 
It is clear that $\nabla_{x'}\rho(0)=0$ and $x_{n}=O(|x'|^{2})$. 
Let $\varphi(0)=0$ and $\nabla\varphi(0)=0$ by subtracting an affine function. 
It follows from \eqref{eqn4.2} that 
\begin{equation}\label{eqn4.31}
\nu(x)=\frac{\left(-\rho_{1}(x'),\dots,-\rho_{n-1}(x'),1\right)}{\sqrt{1+\sum_{i=1}^{n-1}\rho_{i}^{2}(x')}} 
\end{equation}
for all $x\in\partial\Omega\cap\{|x'|\leq r_{0}\}$. 
We introduce a tangential vector field $\eta$ defined by 
\begin{equation}\label{eqn4.30}
\eta(x):=\left(\frac{x'}{|x'|},\sum_{i=1}^{n-1}\frac{x_{i}}{|x'|}\rho_{i}\big(x'\big)\right)\quad \text{in}\ B_{r_{0}}\setminus\{0\}. 
\end{equation}
We construct a vector field 
\begin{equation}\label{eqn4.3}
\tau=\tau(x):=
\begin{pmatrix}
1\\
0\\
\vdots\\
0
\end{pmatrix}
+
\begin{pmatrix}
~ &~ &~ &-\rho_{11}(0)\\
~ &0_{(n-1)\times(n-1)}\hspace{-10mm} &~ &\vdots\\
~ &~ &~ &-\rho_{1,n-1}(0)\\
\rho_{11}(0)\hspace{-7mm} &\cdots\hspace{-6mm} &\rho_{1,n-1}(0)\hspace{-3mm} &0
\end{pmatrix}
\begin{pmatrix}
x_{1}\\
x_{2}\\
\vdots\\
x_{n}
\end{pmatrix}
\end{equation}
in $\overline{\Omega}$. 
Define
\[
w(x):= u_{(\tau)(\tau)}(x)-u_{(\tau)(\tau)}(0)\quad \text{in}\ \overline{\Omega},
\]
and 
\[
M:=\sup_{x\in\partial\Omega}u_{\nu\nu}(x). 
\]
Without loss of generality, we assume $M\gg1$, and $C_{\varepsilon}>1$ for all $\varepsilon$ in \eqref{eqn1.7}.  

\begin{lemma}\label{thm4.7}
There exist a positive constant $K_{0}$ and a linear function $h$ of $x'$ such that 
\begin{equation}\label{eqn4.10}
w(x)-h\big(x'\big)\leq K_{0}\left(\big|x'\big|^{2}+M\big|x'\big|^{4}\right)
\end{equation}
for all $x\in\partial\Omega\cap\{|x'|\leq r_{0}\}$, where $K_{0}$ and the coefficients of $h$ depend on $\Omega$, $\|\varphi\|_{C^{3,1}(\partial\Omega)}$, $\operatorname{dist}(\Omega,\partial\Omega_{0})$, and either $\|f^{1/(k-1)}\|_{C^{1,1}(\overline{\Omega_{0}})}$ or  $\|f^{3/(2k-2)}\|_{C^{1,1}(\overline{\Omega_{0}})}$, but are independent of $\inf_{\Omega_{0}}f$. 
\end{lemma}
\begin{proof}  
For any point $x\in\partial\Omega\cap\{|x'|\leq r_{0}\}$, let $\xi(x)$ be the projection of $\tau(x)$ onto the tangent plane of $\partial\Omega$ at $x$.   
By \eqref{eqn4.31} and \eqref{eqn4.3}, one has  
\begin{equation}\label{eqn4.4}
\langle \tau,\nu\rangle=O\Big(\big|x'\big|^{2}\Big)
\end{equation}
on $\partial\Omega\cap\{|x'|\leq r_{0}\}$. 
It is clear that for any $i\in\{1,\dots,n\}$, 
\begin{align}
u_{i}(x)-u_{i}(0)&=u_{i}\big(x',\rho\big(x'\big)\big)-u_{i}(0,\rho(0))\nonumber\\
&=e_{i}^{T}D^{2}u\big(tx',\rho\big(tx'\big)\big)\eta\big(tx',\rho\big(tx'\big)\big)\times \left|x'\right|\nonumber\\
&= O\big(\big|x'\big|\big),\label{eqn4.6}
\end{align}
for all $x\neq0$ on $\partial\Omega\cap\{|x'|\leq r_{0}\}$, where $t\in(0,1)$ depends on $x$, and the tangential vector field $\eta$ is defined in \eqref{eqn4.30}.   
Since $\nabla_{x'}u(0)=\nabla_{x'}\varphi(0)=0$, using \eqref{eqn4.6} one gets  
\begin{align}
|u_{\nu}(x)-u_{n}(x)|&\leq\sum_{i=1}^{n-1}\left|\left(u_{i}(x)-u_{i}(0)\right)\rho_{i}\big(x'\big)\right|+\frac{|u_{n}(x)|\sum_{i=1}^{n-1}\rho_{i}^{2}(x')}{1+\sqrt{1+\sum_{i=1}^{n-1}\rho_{i}^{2}}}\nonumber\\
&=O\Big(\big|x'\big|^{2}\Big)\label{eqn4.9}
\end{align}
on $\partial\Omega\cap\{|x'|\leq r_{0}\}$. 

Combining \eqref{eqn4.4}, \eqref{eqn4.3}, and \eqref{eqn4.6}, one derives  
\begin{align}
u_{(\tau)(\tau)}(x)&=(\xi+\langle\tau,\nu\rangle\nu)^{T}D^{2}u\left(\xi+\langle\tau,\nu\rangle\nu\right)+u_{i}\tau_{j}\partial_{x_{j}}(\tau_{i})\nonumber\\
&=u_{\xi\xi}+u_{i}\tau_{j}\partial_{x_{j}}(\tau_{i})+O\Big(\big|x'\big|^{2}+M\big|x'\big|^{4}\Big)\nonumber\\
&=u_{\xi\xi}+ \rho_{11}(0)u_{n}(x)+h_{1}\big(x'\big)+O\Big(\big|x'\big|^{2}+M\big|x'\big|^{4}\Big) \label{eqn4.7}
\end{align}
on $\partial\Omega\cap\{|x'|\leq r_{0}\}$, where 
\[
h_{1}\big(x'\big):=-\sum_{i=1}^{n-1}\bigg(\rho_{1i}(0)\sum_{k=1}^{n-1}\rho_{1k}(0)u_{k}(0)\bigg)x_{i}.
\]
Using $\partial\Omega\in C^{3,1}$ and $u=\varphi$ on $\partial\Omega$, one gets   
\begin{align}\label{eqn4.8}
u_{\xi\xi}(x)&=\varphi_{\xi\xi}(x)+\kappa(x,\xi(x))\left(\varphi_{\nu}(x)-u_{\nu}(x)\right)\nonumber\\
&=[\varphi_{11}+\rho_{11}\varphi_{n}](0)+h_{2}\big(x'\big)+O\Big(\big|x'\big|^{2}\Big)-\kappa(x,\xi(x))u_{\nu}(x)\nonumber\\
&=[u_{11}+\rho_{11}u_{n}](0)+h_{2}\big(x'\big)+O\Big(\big|x'\big|^{2}\Big)-\kappa(x,\xi(x))u_{\nu}(x)
\end{align}
on $\partial\Omega\cap\{|x'|\leq r_{0}\}$, where $\kappa(x,\xi)$ is the curvature of $\partial\Omega$ along the direction $\xi$ at $x\in\partial\Omega$, and $h_{2}(x')$ is the linear part in the Taylor expansion of $\varphi_{\xi\xi}(x)+\kappa(x,\xi(x))\varphi_{\nu}(x)$ with respect to $x'$. 
By \eqref{eqn4.7}, \eqref{eqn4.8}, \eqref{eqn4.9}, and \eqref{eqn4.6}, we have 
\begin{align}
w(x)&=u_{(\tau)(\tau)}(x)-[u_{11}+\rho_{11}u_{n}](0)\nonumber\\
&=\rho_{11}(0)u_{n}(x)-\kappa(x,\xi(x))u_{\nu}(x)+h_{1}+h_{2}+O\Big(\big|x'\big|^{2}+M\big|x'\big|^{4}\Big) \nonumber\\
&=\left[\rho_{11}(0)-\kappa(x,\xi(x))\right]u_{n}(x)+h_{1}+h_{2}+O\Big(\big|x'\big|^{2}+M\big|x'\big|^{4}\Big) \nonumber\\
&=\left[\rho_{11}(0)-\kappa(x,\xi(x))\right]u_{n}(0)+h_{1}+h_{2}+O\Big(\big|x'\big|^{2}+M\big|x'\big|^{4}\Big) \nonumber\\ 
&=h_{3}\big(x'\big)+h_{1}+h_{2}+O\Big(\big|x'\big|^{2}+M\big|x'\big|^{4}\Big), \nonumber
\end{align}
on $\partial\Omega\cap\{|x'|\leq r_{0}\}$, where $h_{3}(x')$ is the linear part in the Taylor expansion of $-\kappa(x,\xi(x))u_{n}(0)$ with respect to $x'$.

Consequently, there exists a constant $K_{0}>0$ such that 
\[
w(x)-h\big(x'\big)\leq K_{0}\left(\big|x'\big|^{2}+M\big|x'\big|^{4}\right)
\]
for all $x\in\partial\Omega\cap\{|x'|\leq r_{0}\}$, where $h(x'):=h_{1}(x')+h_{2}(x')+h_{3}(x')$; $K_{0}$ and the coefficients of $h$ depend on $\Omega$, $\|\varphi\|_{C^{3,1}(\partial\Omega)}$, $\|u\|_{C^{1}(\overline{\Omega})}$, second-order tangential and tangential-normal  derivatives of $u$ on $\partial\Omega$. 
Using the results in Section 3, we complete the proof of Lemma \ref{thm4.7}. 
\end{proof}
\par
\vspace{2mm}

Define the linear elliptic operator $\mathcal{L}$ by 
\[
\mathcal{L}(w):=F^{ij}\big[D^{2}u\big]w_{ij}.
\]
Define  
\[
\omega_{r}:=\left\{x\in\Omega:\ \rho\big(x'\big)<x_{n}<\rho\big(x'\big)+r^{4},\ \big|x'\big|<r\right\},\quad \forall r\leq r_{0}.
\]

\subsection{The case under condition $(\romannumeral1)$ or $(\romannumeral2)$}

In this subsection, we assume condition \eqref{eqn1.6} is satisfied, and either $f^{1/(k-1)}\in C^{1,1}(\overline{\Omega_{0}})$ holds or $f^{3/(2k-2)}\in C^{2,1}(\overline{\Omega_{0}})$ and $k\geq5$ hold.  

\begin{lemma}\label{thm4.2}
Let $\inf_{\Omega_{0}}f>0$ where  $\Omega\Subset\Omega_{0}$. 
We have two conclusions:

$(\romannumeral1)$ If $f^{1/(k-1)}\in C^{1,1}(\overline{\Omega_{0}})$, then 
\[
\mathcal{L}(w)\geq-C_{1}f^{\frac{-1}{k(k-1)}}\quad \text{in}\ \Omega,
\]
where $C_{1}>0$ depends on $\Omega$, $\operatorname{dist}(\Omega,\partial\Omega_{0})$, and  $\|f^{1/(k-1)}\|_{C^{1,1}(\overline{\Omega_{0}})}$, but is independent of $\inf_{\Omega_{0}}f$. 

$(\romannumeral2)$ If $f^{3/(2k-2)}\in C^{2,1}(\overline{\Omega_{0}})$ and $k\geq5$, then 
\[
\mathcal{L}(w)\geq-C_{2}f^{\frac{-1}{k(k-1)}}\quad \text{in}\ \Omega, 
\]
where $C_{2}>0$ depends on $\Omega$, $\operatorname{dist}(\Omega,\partial\Omega_{0})$, and 
$\|f^{3/(2k-2)}\|_{C^{2,1}(\overline{\Omega_{0}})}$, but is independent of $\inf_{\Omega_{0}}f$.
\end{lemma}
\begin{proof}
From Lemma \ref{thm4.1}, we have
\begin{align}
\mathcal{L}(w)&=F^{ij}\big[D^{2}u\big]\left(u_{(\tau)(\tau)}\right)_{ij}\nonumber\\
&\geq \left(g^{\frac{p}{k}}\right)_{\tau\tau}+\left(g^{\frac{p}{k}}\right)_{i}\tau_{j}\partial_{x_{j}}(\tau_{i})\nonumber\\
&=\frac{p}{k}g^{\frac{p}{k}-1}\left(\partial_{\tau\tau}g-\frac{k-p}{k}\frac{\left|\partial_{\tau}g\right|^{2}}{g}\right)+ \left(\frac{p}{k}g^{\frac{p}{k}-\frac{1}{2}}\frac{\partial_{x_{i}}g}{\sqrt{g}}\right)\tau_{j}\partial_{x_{j}}(\tau_{i})\label{eqn7.10}
\end{align}
in $\Omega$. 
When $f^{1/(k-1)}\in C^{1,1}(\overline{\Omega_{0}})$, it follows from \eqref{eqn7.10} and \eqref{eqn7.1} that  
\begin{align*}
\mathcal{L}(w)&\geq-\frac{p}{k}g^{-\frac{1}{k}}\left|\partial_{\tau\tau}g-\frac{k-p}{k}\frac{\left|\partial_{\tau}g\right|^{2}}{g}\right|- g^{-\frac{1}{k}}\left|\frac{p}{k}\partial_{x_{i}}g\tau_{j}\partial_{x_{j}}(\tau_{i})\right|\\
&\geq-C_{1}f^{\frac{-1}{k(k-1)}}
\end{align*}
in $\Omega$, where $C_{1}>0$ depends on $\Omega$, $\operatorname{dist}(\Omega,\partial\Omega_{0})$, and  $\|f^{1/(k-1)}\|_{C^{1,1}(\overline{\Omega_{0}})}$. 

When $f^{3/(2k-2)}\in C^{2,1}(\overline{\Omega_{0}})$ and $k\geq5$, which implies $(k+2)/3k<1/2$, it follows from \eqref{eqn7.10}, \eqref{eqn7.2}, and  \eqref{eqn7.1} that 
\begin{align*}
\mathcal{L}(w)&\geq\frac{p}{k}g^{\frac{-p}{k(k-1)}-\frac{1}{3}}\left(\partial_{\tau\tau}g-\frac{k+2}{3k}\frac{\left|\partial_{\tau}g\right|^{2}}{g}\right)+ g^{\frac{-p}{k(k-1)}}\left(\frac{p}{k}g^{\frac{1}{6}}\frac{\partial_{x_{i}}g}{\sqrt{g}}\right)\tau_{j}\partial_{x_{j}}(\tau_{i})\\
&\geq-C_{2}f^{\frac{-1}{k(k-1)}}
\end{align*}
in $\Omega$, where $C_{2}>0$ depends on $\Omega$, $\operatorname{dist}(\Omega,\partial\Omega_{0})$, and 
$\|f^{3/(2k-2)}\|_{C^{2,1}(\overline{\Omega_{0}})}$. 
This completes the proof of Lemma \ref{thm4.2}. 
\end{proof}

\begin{lemma}\label{thm4.3}
There exist positive constants $r_{2}$ and $K_{2}$ such that 
\[
w(x)-h\big(x'\big)\leq K_{2}\left(\big|x'\big|^{2}+M\big|x'\big|^{4}\right)+K_{2}M\left(x_{n}-\rho\big(x'\big)\right)\quad \text{in}\ \omega_{r_{2}}, 
\]
where $r_{2}$ depends on $\Omega$; and the constant $K_{2}$ depends on $\delta_{0}$, $\Omega$, $\|\varphi\|_{C^{3,1}(\partial\Omega)}$, $\operatorname{dist}(\Omega,\partial\Omega_{0})$, and either $\|f^{1/(k-1)}\|_{C^{1,1}(\overline{\Omega_{0}})}$ or  $\|f^{3/(2k-2)}\|_{C^{2,1}(\overline{\Omega_{0}})}$, but is independent of $\inf_{\Omega_{0}}f$. 
\end{lemma}
\begin{proof} 
Define a function 
\[
v:=\big(x_{n}-\rho\big(x'\big)\big)^{2}-\beta_{1}\big(x_{n}-\rho\big(x'\big)\big)-\frac{K_{0}}{M}\big|x'\big|^{2}-K_{0}\big|x'\big|^{4}
\]
in $\omega_{r_{0}}$, where the constant $\beta_{1}>0$ is to be determined.  
It holds that 
\begin{align*}
D^{2}v&=D^{2}\big[x_{n}^{2}+\beta_{1}\rho\big(x'\big)\big]+D^{2}\big[\rho^{2}\big(x'\big)-2x_{n}\rho\big(x'\big)\big]\\
&\hspace{4.5mm}-\frac{K_{0}}{M}D^{2}\Big[\big|x'\big|^{2}\Big]-D^{2}\Big[K_{0}\big|x'\big|^{4}\Big].
\end{align*}
Since $\partial\Omega\in C^{3,1}$ is strictly $(k-1)$-convex, one can fix $\beta_{1}$  sufficiently small such that $\lambda(D^{2}[x_{n}^{2}+\beta_{1}\rho](0))$ belongs to a compact subset of $\Gamma_{k}$ independent of $0\in\partial\Omega$.  
By 
\[
|x|\rightarrow0,\ \rho\big(x'\big)\rightarrow0,\ \nabla_{x'}\rho\big(x'\big)\rightarrow0,\ D^{2}_{x'x'}\rho\big(x'\big)\rightarrow D^{2}_{x'x'}\rho(0)\quad \text{as}\ r\rightarrow0,
\]
there exist two small positive constants $r_{2}\leq\min\{r_{0},\beta_{1}\}$, $\delta_{2}$, and a large constant $C_{1}>0$ such that 
\[
\left\{\lambda\big(D^{2}v-\delta_{2}\mathbf{I}_{n}\big):\ x\in\omega_{r_{2}}\right\}\ \ \text{is a perturbation of}\ \  \lambda\big(D^{2}\big[x_{n}^{2}+\beta_{1}\rho\big](0)\big)  
\]
for all $M\geq C_{1}$, and is thus contained in a compact subset of $\Gamma_{k}$ for all $M\geq C_{1}$.
We emphasize that $\beta_{1}$, $r_{2}$, $\delta_{2}$, $C_{1}$ all depend only on $\Omega$.   
Using the concavity of $F$, we have  
\[
F^{ij}\big[D^{2}u\big]\left(C_{2}(v_{ij}-\delta_{2}\delta_{ij})-u_{ij}\right)\geq C_{2}\sigma_{k}^{\frac{1}{k}}\big[D^{2}v-\delta_{2}\mathbf{I}_{n}\big]-f^{\frac{1}{k}}>0
\]  
in $\omega_{r_{2}}$, where $C_{2}$ is sufficiently large. 
Then by virtue of \eqref{eqn2.12} and \eqref{eqn2.11}, we derive 
\begin{equation}\label{eqn4.13}
\mathcal{L}(v)\geq \delta_{2}\operatorname{tr}F^{ij}\big[D^{2}u\big]\geq\frac{\delta_{2}}{C}\sigma_{1}^{\frac{1}{k-1}}\big[D^{2}u\big]f^{\frac{-1}{k(k-1)}}\overset{\eqref{eqn1.6}}{\geq} \frac{\delta_{2}}{C}\delta_{0}^{\frac{1}{k-1}}f^{\frac{-1}{k(k-1)}}
\end{equation}
in $\omega_{r_{2}}$, where $C>0$ is universal. 
Thus using \eqref{eqn4.13} and Lemma \ref{thm4.2}, there exists a large constant $K>0$ such that 
\begin{equation}\label{eqn4.14}
\mathcal{L}(-KMv)<\mathcal{L}(-Kv)<\mathcal{L}(w-h)\quad \text{in}\ \omega_{r_{1}},
\end{equation}
where $K$ depends on $\delta_{0}$, $\Omega$, $\operatorname{dist}(\Omega,\partial\Omega_{0})$, and 
either $\|f^{1/(k-1)}\|_{C^{1,1}(\overline{\Omega_{0}})}$ or $\|f^{3/(2k-2)}\|_{C^{2,1}(\overline{\Omega_{0}})}$. 

The boundary $\partial\omega_{r_{2}}$ consists of three parts: $\partial_{1}\omega_{r_{2}}\cup\partial_{2}\omega_{r_{2}}\cup\partial_{3}\omega_{r_{2}}$, where $\partial_{1}\omega_{r_{2}}$ and $\partial_{2}\omega_{r_{2}}$ are respectively the graph parts of $\rho$ and $\rho+r_{2}^{4}$, and $\partial_{3}\omega_{r_{2}}$ is the boundary part on $\{|x'|=r_{2}\}$. 
On $\partial_{1}\omega_{r_{2}}$, using \eqref{eqn4.10} one gets 
\[
-KMv=KK_{0}\left(\big|x'\big|^{2}+M\big|x'\big|^{4}\right)\geq w-h.
\]
It follows from the results in Section 3 that 
\[
w-h\leq C_{3}M+C_{4}\quad \text{on}\ \partial_{2}\omega_{r_{2}}\cup\partial_{3}\omega_{r_{2}}, 
\]
where $C_{3}$ depends on $\Omega$; $C_{4}$ depends on $\Omega$, $\|\varphi\|_{C^{3,1}(\partial\Omega)}$,  $\operatorname{dist}(\Omega,\partial\Omega_{0})$, and either $\|f^{1/(k-1)}\|_{C^{1,1}(\overline{\Omega_{0}})}$ or  $\|f^{3/(2k-2)}\|_{C^{2,1}(\overline{\Omega_{0}})}$. 
It holds that
\[
-KMv\geq KM\left(\beta_{1}r_{2}^{4}-r_{2}^{8}\right)\geq\frac{\beta_{1}}{2}KMr_{2}^{4}\quad \text{on}\ \partial_{2}\omega_{r_{2}}, 
\]
and
\[
-KMv\geq KK_{0}\left(r_{2}^{2}+Mr_{2}^{4}\right)\quad \text{on}\ \partial_{3}\omega_{r_{2}}.
\] 
Then we obtain $-KMv\geq w-h$ on $\partial_{2}\omega_{r_{2}}\cup\partial_{3}\omega_{r_{2}}$ by taking $K$ large, which depends on $\Omega$, $\|\varphi\|_{C^{3,1}(\partial\Omega)}$,  $\operatorname{dist}(\Omega,\partial\Omega_{0})$, and either $\|f^{1/(k-1)}\|_{C^{1,1}(\overline{\Omega_{0}})}$ or  $\|f^{3/(2k-2)}\|_{C^{2,1}(\overline{\Omega_{0}})}$. 

Therefore, it follows from the maximum principle that 
\[
-KMv\geq w-h\quad \text{in}\ \omega_{r_{2}}, 
\]
which implies
\[
w-h\leq KM\left(\beta_{1}\big(x_{n}-\rho\big(x'\big)\big)+\frac{K_{0}}{M}\big|x'\big|^{2}+K_{0}\big|x'\big|^{4}\right)
\]
in $\omega_{r_{2}}$. 
This completes the proof of Lemma \ref{thm4.3}. 
\end{proof}

\begin{lemma}\label{thm4.4}
For any $\sigma>0$, there exists a positive constant $K_{\sigma}$ such that 
\[
\left(u_{(\tau)(\tau)}\right)_{n}(0)\leq\sigma M+K_{\sigma}, 
\]
where constant $K_{\sigma}$ depends on $\sigma$, $\Omega$, $\|\varphi\|_{C^{3,1}(\partial\Omega)}$,  $\operatorname{dist}(\Omega,\partial\Omega_{0})$, and either $\|f^{1/(k-1)}\|_{C^{1,1}(\overline{\Omega_{0}})}$ or  $\|f^{3/(2k-2)}\|_{C^{2,1}(\overline{\Omega_{0}})}$, but is independent of $\inf_{\Omega_{0}}f$.  
\end{lemma}
\begin{proof}
Let $v$ be the function defined in the proof of Lemma \ref{thm7.2}, which is  
\[
v=\big(x_{n}-\rho\big(x'\big)\big)^{2}-\beta_{1}\big(x_{n}-\rho\big(x'\big)\big)-\beta_{2}\big|x'\big|^{2}\quad \text{in}\ \omega_{r_{1}}.
\] 
Recalling \eqref{eqn7.22}, we derive 
\begin{equation}\label{eqn4.18}
\mathcal{L}(v)\geq\frac{\delta_{1}}{C}\sigma_{1}^{\frac{1}{k-1}}\big[D^{2}u\big]f^{\frac{-1}{k(k-1)}}\overset{\eqref{eqn1.6}}{\geq}\frac{\delta_{1}}{C}\delta_{0}^{\frac{1}{k-1}}f^{\frac{-1}{k(k-1)}}\quad \text{in}\ \omega_{r_{1}},
\end{equation}
where $C>0$ is universal.  
We emphasize that $\beta_{1}$, $\beta_{2}$, $r_{1}\leq\min\{r_{0},\beta_{1}\}$, $\delta_{1}$ all depend only on $\Omega$.
It follows from Lemma \ref{thm4.2} and \eqref{eqn4.18} that there exists a large constant $K>0$ such that 
\begin{equation}\label{eqn4.19}
\mathcal{L}(-Kv)<\mathcal{L}(w-h)\quad \text{in}\ \omega_{r_{1}}, 
\end{equation}
where constant $K$ depends on $\Omega$, $\operatorname{dist}(\Omega,\partial\Omega_{0})$, and either $\|f^{1/(k-1)}\|_{C^{1,1}(\overline{\Omega_{0}})}$ or $\|f^{3/(2k-2)}\|_{C^{2,1}(\overline{\Omega_{0}})}$. 

For any $r\leq \min\{r_{1},r_{2}\}$, the boundary of $\omega_{r}$ consists of three parts: $\partial_{1}\omega_{r}\cup\partial_{2}\omega_{r}\cup\partial_{3}\omega_{r}$, where $\partial_{1}\omega_{r}$ and $\partial_{2}\omega_{r}$ are respectively the graph parts of $\rho$ and $\rho+r^{4}$, and $\partial_{3}\omega_{r}$ is the boundary part on $\{|x'|=r\}$. 
Using \eqref{eqn4.10}, one gets 
\begin{equation}\label{eqn4.22}
-K(rM+1)v=K\beta_{2}(rM+1)\big|x'\big|^{2}\geq K\beta_{2}\left(\big|x'\big|^{2}+M\big|x'\big|^{4}\right)\geq w-h
\end{equation}
on $\partial_{1}\omega_{r}$, by taking $K$ large, which depends on $\Omega$, $\|\varphi\|_{C^{3,1}(\partial\Omega)}$, $\operatorname{dist}(\Omega,\partial\Omega_{0})$, and either $\|f^{1/(k-1)}\|_{C^{1,1}(\overline{\Omega_{0}})}$ or  $\|f^{3/(2k-2)}\|_{C^{1,1}(\overline{\Omega_{0}})}$. 
It follows from Lemma \ref{thm4.3} that 
\[
w-h\leq 2K_{2}Mr^{4}+K_{2}r^{2}\quad \text{on}\ \partial_{3}\omega_{r}. 
\]
We have 
\begin{equation}\label{eqn4.23}
-K(rM+1)v\geq K(rM+1)\beta_{2}r^{2}\geq K\beta_{2}Mr^{4}+K\beta_{2}r^{2}\geq w-h
\end{equation}
on $\partial_{3}\omega_{r}$, by taking $K$ large, which depends on $\delta_{0}$, $\Omega$, $\|\varphi\|_{C^{3,1}(\partial\Omega)}$, $\operatorname{dist}(\Omega,\partial\Omega_{0})$, and either $\|f^{1/(k-1)}\|_{C^{1,1}(\overline{\Omega_{0}})}$ or  $\|f^{3/(2k-2)}\|_{C^{2,1}(\overline{\Omega_{0}})}$. 

Without loss of generality, we assume $|\nabla_{x'}\rho|\leq1$ on $\{|x'|\leq r_{0}\}$. 
The distance function satisfies  
\[
\operatorname{dist}(x,\partial\Omega)\geq\inf_{|y'|\leq r_{0}}\left(\big|x'-y'\big|+r^{4}-\big|\nabla_{x'}\rho\cdot\big(x'-y'\big)\big|\right)
\geq r^{4},\quad \forall x\in\partial_{2}\omega_{r}. 
\] 
By virtue of the weakly interior estimate, namely Theorem \ref{thm1.1}, we infer  
\begin{align*}
w-h\leq\left|u_{ij}\tau_{i}\tau_{j}\right|+C_{1}&\leq2\sup_{\Omega_{r^{4}}}\left|D^{2}u\right|+C_{1}\\
&\leq \frac{2}{r^{4}}\left(r^{9}\sup_{\partial\Omega}\left|D^{2}u\right|+C_{r^{9}}\right)+C_{1}\\
&\leq2Mr^{5}+\frac{2C_{r^{9}}}{r^{4}}+C_{1}
\end{align*}
on $\partial_{2}\omega_{r}$, where $C_{r^{9}}$ is given in \eqref{eqn1.7}; and  $C_{1}>0$ depends on $\Omega$, $\|\varphi\|_{C^{3,1}(\partial\Omega)}$, $\operatorname{dist}(\Omega,\partial\Omega_{0})$, and either $\|f^{1/(k-1)}\|_{C^{1,1}(\overline{\Omega_{0}})}$ or  $\|f^{3/(2k-2)}\|_{C^{1,1}(\overline{\Omega_{0}})}$. 
Then we obtain  
\begin{align}\label{eqn4.24}
-K\left(rM+\frac{C_{r^{9}}}{r^{8}}\right)v\geq K\left(rM+\frac{C_{r^{9}}}{r^{8}}\right)\frac{\beta_{1}}{2}r^{4}&= \frac{K\beta_{1}}{2}Mr^{5}+\frac{K\beta_{1}}{2}\frac{C_{r^{9}}}{r^{4}}\nonumber\\
&\geq w-h 
\end{align}
on $\partial_{2}\omega_{r}$, by taking  $K$ large, which depends on $\Omega$, $\|\varphi\|_{C^{3,1}(\partial\Omega)}$, $\operatorname{dist}(\Omega,\partial\Omega_{0})$, and either $\|f^{1/(k-1)}\|_{C^{1,1}(\overline{\Omega_{0}})}$ or  $\|f^{3/(2k-2)}\|_{C^{1,1}(\overline{\Omega_{0}})}$. 

Note that $\mathcal{L}(v)>0$ and $v<0$ in $\omega_{r}$.
Applying the maximum principle with \eqref{eqn4.19}--\eqref{eqn4.24} yields  
\[
-K\left(rM+\frac{C_{r^{9}}}{r^{8}}\right)v\geq w-h\quad \text{in}\ \omega_{r},\quad \forall r\leq\min\{r_{1},r_{2}\}, 
\]
where the constant $K$ depends on $\delta_{0}$, $\Omega$, $\|\varphi\|_{C^{3,1}(\partial\Omega)}$,  $\operatorname{dist}(\Omega,\partial\Omega_{0})$, and 
either $\|f^{1/(k-1)}\|_{C^{1,1}(\overline{\Omega_{0}})}$ or $\|f^{3/(2k-2)}\|_{C^{2,1}(\overline{\Omega_{0}})}$. 
Therefore,  
\[
\left(u_{(\tau)(\tau)}\right)_{n}(0)=\lim_{x_{n}\rightarrow0^{+}}\frac{w(0,x_{n})-h(0)}{x_{n}}\leq K\beta_{1}\left(rM+\frac{C_{r^{9}}}{r^{8}}\right)
\]
for all $r\leq\min\{r_{1},r_{2}\}$. 
Let
\[
\sigma:=K\beta_{1} r\quad \text{and}\quad  K_{\sigma}:=K\beta_{1}\frac{C_{r^{9}}}{r^{8}}.
\]
This completes the proof of Lemma \ref{thm4.4}. 
\end{proof}

\begin{lemma}\label{thm4.5} 
It holds that 
\[
\left(u_{\nu ij}\eta_{i}\eta_{j}+u_{\nu i}\eta_{j}\partial_{x_{j}}(\eta_{i})\right)(x)
\leq\sigma M+K_{\sigma}, \quad\forall\sigma>0, 
\]
for all points $x\neq0$ on $\partial\Omega\cap\{|x'|\leq r_{0}\}$, where $K_{\sigma}$ is given in Lemma \ref{thm4.4}.  
\end{lemma}
\begin{proof}
It is clear that the expression $(u_{\nu ij}\eta_{i}\eta_{j}+u_{\nu i}\eta_{j}\partial_{x_{j}}(\eta_{i}))(x)$ is invariant under orthogonal transformations.  

\par
\vspace{2mm}
\textbf{Step 1.} 
Let $\tilde{\eta}(x)$ be an arbitrary tangential vector field defined in $B_{r}(0)$, where the radius $r>0$ can be arbitrarily small depending on the origin $0\in\partial\Omega$. 
Furthermore, let $\tilde{\eta}$ satisfy  
\[
\tilde{\eta}(0)=e_{1}\quad \text{and}\quad \left|\partial_{x_{1}}(\tilde{\eta}_{i})(0)\right|\leq C,\quad \forall1\leq i\leq n-1, 
\]
for some constant $C$. 

Differentiating $\nu\cdot\tilde{\eta}(x',\rho(x'))=0$ with respect to $x_{1}$ at the origin yields  
\[
\partial_{x_{1}}(\tilde{\eta}_{n})(0)-\rho_{11}(0)=0.
\]
Then we have 
\[
u_{n(\tilde{\eta})(\tilde{\eta})}(0)
=u_{n11}(0)+u_{ni}(0)\partial_{x_{1}}(\tilde{\eta}_{i})(0)\leq u_{n11}(0)+u_{nn}(0)\rho_{11}(0)+K_{1}, 
\]
where $K_{1}$ depends on $C$ and tangential-normal derivatives of $u$ on $\partial\Omega$. 
It holds that 
\begin{align*}
\left(u_{(\tau)(\tau)}\right)_{n}(0)&=u_{n11}(0)+u_{in}(0)\tau_{j}\partial_{x_{j}}(\tau_{i})(0)\\
&\hspace{4.5mm}+u_{ij}\partial_{x_{n}}[\tau_{i}\tau_{j}]+u_{i}\partial_{x_{n}}\big[\tau_{j}\partial_{x_{j}}(\tau_{i})\big]\\
&\geq u_{n11}(0)+u_{nn}(0)\rho_{11}(0)-K_{2},
\end{align*}
where $K_{2}$ depends on $\Omega$, $\|u\|_{C^{1}(\overline{\Omega})}$, second-order tangential and tangential-normal derivatives of $u$ on $\partial\Omega$.  
As a conclusion,
\begin{equation}\label{eqn7.12}
u_{n(\tilde{\eta})(\tilde{\eta})}(0)\leq \left(u_{(\tau)(\tau)}\right)_{n}(0)+K_{1}+K_{2}\leq\sigma M+K_{\sigma},\quad \forall\sigma>0, 
\end{equation}
by Lemma \ref{thm4.4} and increasing $K_{\sigma}$. 

\par
\vspace{2mm}
\textbf{Step 2.}
For any $x_{0}\neq0$ on $\partial\Omega\cap\{|x'|\leq r_{0}\}$, 
by rotation and translation, we introduce a new coordinate system $y$ such that  $x_{0}$ is the origin $0\in\partial\Omega$, $e_{n}=\nu(0)$, and $\eta(x_{0})$ is in the direction $e_{1}$. 
Denote the expression of vector field $\eta(x)$ in the new coordinate system by $\bar{\eta}(y)$.
It is clear that $\bar{\eta}$ is well-defined in $B_{|x_{0}|/2}(0)$. 
Recalling \eqref{eqn4.30}, by calculation, one can verify 
\[
\left|\partial_{y_{1}}(\bar{\eta}_{i})(0)\right|\leq C,\quad \forall1\leq i\leq n-1, 
\]
where $C$ depends only on the $C^{2}$-norm of $\partial\Omega$.  

Define $\tilde{\eta}(y):=\bar{\eta}(y)/|\eta(x_{0})|$.
Then   
\[
\tilde{\eta}(0)=e_{1}\quad \text{and}\quad \left|\partial_{y_{1}}(\tilde{\eta}_{i})(0)\right|\leq C,\quad \forall1\leq i\leq n-1. 
\]
By \eqref{eqn7.12} and $|\eta(x_{0})|\leq2$, we get 
\[
\left(u_{nij}\bar{\eta}_{i}\bar{\eta}_{j}+u_{ni}\bar{\eta}_{j}\partial_{y_{j}}(\bar{\eta}_{i})\right)(0)\leq 4u_{n(\tilde{\eta})(\tilde{\eta})}(0)\leq\sigma M+4K_{\sigma/4}
\]
for all $\sigma>0$. 
Coming back to the original coordinate system \eqref{eqn4.2}, by increasing $K_{\sigma}$, we derive 
\[
\left(u_{\nu ij}\eta_{i}\eta_{j}+u_{\nu i}\eta_{j}\partial_{x_{j}}(\eta_{i})\right)(x_{0})
\leq\sigma M+K_{\sigma},\quad \forall\sigma>0, 
\]
for all $x_{0}\neq0$ on $\partial\Omega\cap\{|x'|\leq r_{0}\}$. 
The proof of Lemma \ref{thm4.5} is complete.  
\end{proof}
\par
\vspace{2mm}

We now proceed to establish the boundary second-order normal derivative estimate.
\par
\vspace{2mm}
\begin{proof}[Proof of Theorem \ref{thm10.3}]
Without loss of generality, let $M$ be attained at the origin $0\in\partial\Omega$, $e_{n}=\nu(0)$, \eqref{eqn4.2} hold, $\varphi(0)=0$, and $\nabla\varphi(0)=0$.  

Applying Taylor expansion to $h(t):=u_{n}(tx',\rho(tx'))$, we obtain that for any $x\neq0$ on $\partial\Omega\cap\{|x'|\leq r_{0}\}$,   
\begin{equation}\label{eqn4.20}
u_{n}\big(x',\rho\big(x'\big)\big)-u_{n}(0)-\sum_{i=1}^{n-1}u_{ni}(0)x_{i}
=\frac{1}{2}u_{n(\eta)(\eta)}\big(tx',\rho\big(tx'\big)\big)\left|x'\right|^{2},
\end{equation}
where $t\in(0,1)$ depends on $x$, and $\eta$ is defined in \eqref{eqn4.30}.  
Differentiating $u=\varphi$ on $\partial\Omega$ three times, yields that the third-order tangential derivatives of $u$ are bounded. 
Using Lemma \ref{thm4.5}, we derive  
\begin{align}\label{eqn4.21}
u_{n(\eta)(\eta)}(y)&=\left[u_{n\eta\eta}+u_{ni}\eta_{j}\partial_{x_{j}}(\eta_{i})\right](y)\nonumber\\
&=\langle e_{n},\nu\rangle\left[ u_{\nu ij}\eta_{i}\eta_{j}+ u_{\nu i}\eta_{j}\partial_{x_{j}}(\eta_{i})\right](y)+C_{1}\nonumber\\
&\leq\sigma M+K_{\sigma}+C_{1} 
\end{align}
for all $y\neq0$ on $\partial\Omega\cap\{|x'|\leq r_{0}\}$ and for all $\sigma>0$, where $C_{1}$ depends on $\Omega$, $\|\varphi\|_{C^{3,1}(\partial\Omega)}$, $\operatorname{dist}(\Omega,\partial\Omega_{0})$, and either $\|f^{1/(k-1)}\|_{C^{1,1}(\overline{\Omega_{0}})}$ or  $\|f^{3/(2k-2)}\|_{C^{1,1}(\overline{\Omega_{0}})}$. 
It follows from \eqref{eqn4.20} and \eqref{eqn4.21} that 
\begin{equation}\label{eqn4.29}
u_{n}(x)-u_{n}(0)-\sum_{i=1}^{n-1}u_{ni}(0)x_{i}
\leq (\sigma M+K_{\sigma})\big|x'\big|^{2}
\end{equation}
for all $x\in\partial\Omega\cap\{|x'|\leq r_{0}\}$, by increasing $K_{\sigma}$. 

Let $v$ be the function defined in the proof of Lemma \ref{thm7.2}, which is 
\[
v=\big(x_{n}-\rho\big(x'\big)\big)^{2}-\beta_{1}\big(x_{n}-\rho\big(x'\big)\big)-\beta_{2}\big|x'\big|^{2}\quad \text{in}\ \omega_{r_{1}}.
\] 
Recalling \eqref{eqn7.22}, we derive 
\begin{equation}\label{eqn4.28}
\mathcal{L}(v)\geq\frac{\delta_{1}}{C}\sigma_{1}^{\frac{1}{k-1}}\big[D^{2}u\big]f^{\frac{-1}{k(k-1)}}\overset{\eqref{eqn1.6}}{\geq}\frac{\delta_{1}}{C}f^{\frac{-1}{k(k-1)}}\quad \text{in}\ \omega_{r_{1}},
\end{equation}
where $C>0$ is universal.  
We emphasize that $\beta_{1}$, $\beta_{2}$, $r_{1}\leq\min\{r_{0},\beta_{1}\}$, $\delta_{1}$ all depend only on $\Omega$.
By \eqref{eqn2.7} and \eqref{eqn7.1}, we infer  
\begin{equation}\label{eqn8.1}
\mathcal{L}(u_{n})=\left(\frac{p}{k}g^{\frac{p}{k-1}-\frac{1}{2}}\frac{\partial_{n}g}{\sqrt{g}}\right)f^{\frac{-1}{k(k-1)}}\geq-C_{2}f^{\frac{-1}{k(k-1)}}\quad \text{in}\ \Omega,
\end{equation}
where the constant $C_{2}$ depends on $\operatorname{dist}(\Omega,\partial\Omega_{0})$ and either $\|f^{1/(k-1)}\|_{C^{1,1}(\overline{\Omega_{0}})}$ or  $\|f^{3/(2k-2)}\|_{C^{1,1}(\overline{\Omega_{0}})}$. 
Define
\[
\tilde{v}:=-K(\sigma M+K_{\sigma})v\quad \text{in}\ \overline{\omega_{r_{1}}}, 
\]
where $K>0$ is a constant to be determined. 
By \eqref{eqn4.28} and \eqref{eqn8.1}, we obtain  
\[
\mathcal{L}(\tilde{v})\leq \mathcal{L}\left(u_{n}-u_{n}(0)-\sum_{i=1}^{n-1}u_{ni}(0)x_{i}\right)\quad \text{in}\ \omega_{r_{1}},  
\]
by taking the constant $K$ large, which depends on $\Omega$,  $\operatorname{dist}(\Omega,\partial\Omega_{0})$, and either $\|f^{1/(k-1)}\|_{C^{1,1}(\overline{\Omega_{0}})}$ or  $\|f^{3/(2k-2)}\|_{C^{1,1}(\overline{\Omega_{0}})}$. 
 
The boundary $\partial\omega_{r_{1}}$ consists of three parts: $\partial_{1}\omega_{r_{1}}\cup\partial_{2}\omega_{r_{1}}\cup\partial_{3}\omega_{r_{1}}$, defined in the proof of Lemma \ref{thm7.2}.
Using \eqref{eqn4.29}, one has 
\[
\tilde{v}= K\beta_{2}(\sigma M+K_{\sigma})\big|x'\big|^{2}
\geq u_{n}-u_{n}(0)-\sum_{i=1}^{n-1}u_{ni}(0)x_{i}\quad \text{on}\ \partial_{1}\omega_{r_{1}},  
\]
by taking $K$ large, which depends on $\Omega$. 
It is clear that 
\[
\tilde{v}\geq K(\sigma M+K_{\sigma})\min\left\{\frac{\beta_{1}}{2}r_{1}^{4},\beta_{2}r_{1}^{2}\right\}
\]
on $\partial_{2}\omega_{r_{1}}\cup\partial_{3}\omega_{r_{1}}$. 
According to the results in Section 3, we get
\[
\tilde{v}\geq u_{n}-u_{n}(0)-\sum_{i=1}^{n-1}u_{ni}(0)x_{i}\quad \text{on}\ \partial_{2}\omega_{r_{1}}\cup\partial_{3}\omega_{r_{1}}, 
\] 
by taking $K$ large, which depends on $\Omega$, $\|\varphi\|_{C^{3,1}(\partial\Omega)}$, $\operatorname{dist}(\Omega,\partial\Omega_{0})$, and either $\|f^{1/(k-1)}\|_{C^{1,1}(\overline{\Omega_{0}})}$ or  $\|f^{3/(2k-2)}\|_{C^{1,1}(\overline{\Omega_{0}})}$. 
Hence applying the comparison principle yields 
\[
\tilde{v}\geq u_{n}-u_{n}(0)-\sum_{i=1}^{n-1}u_{ni}(0)x_{i}\quad \text{in}\ \omega_{r_{1}},
\]
where the constant $K$ depends on $\Omega$, $\|\varphi\|_{C^{3,1}(\partial\Omega)}$,  $\operatorname{dist}(\Omega,\partial\Omega_{0})$, and either $\|f^{1/(k-1)}\|_{C^{1,1}(\overline{\Omega_{0}})}$ or  $\|f^{3/(2k-2)}\|_{C^{1,1}(\overline{\Omega_{0}})}$.

Therefore, 
\[
M=u_{nn}(0)=\lim_{x_{n}\rightarrow0^{+}}\frac{u_{n}(0,x_{n})-u_{n}(0)}{x_{n}}\leq\partial_{x_{n}}\tilde{v}(0)=K\beta_{1}(\sigma M+K_{\sigma}), 
\] 
where $K_{\sigma}$ is given in Lemma \ref{thm4.4}. 
Taking $\sigma=1/(2K\beta_{1})$, we obtain the upper bound of $M$. 
Then the $C^{2}$ estimate independent of $\inf_{\Omega_{0}}f$ is established, and  Theorem \ref{thm10.3} follows from an approximation argument.
\end{proof}

\subsection{The case under condition $(\romannumeral3)$}

In this subsection, we assume $f^{3/(2k)}\in C^{2,1}(\overline{\Omega_{0}})$. 
In fact, Lemmas \ref{thm2.5} and \ref{thm4.2} can be improved.

\begin{lemma}\label{thm7.4}
Let $\inf_{\Omega_{0}}f>0$ where $\Omega\Subset\Omega_{0}$, and  $f^{3/(2k)}\in C^{2,1}(\overline{\Omega_{0}})$.  
Then   
\[
\operatorname{tr}\left(GD^{2}u_{\xi\xi}\right)\geq-K|\xi|^{2},\quad \forall \xi\in\mathbb{R}^{n}, 
\]
in $\Omega$, where $K$ depends on  $\operatorname{dist}(\Omega,\partial\Omega_{0})$ and  $\|f^{3/(2k)}\|_{C^{2,1}(\overline{\Omega_{0}})}$, but is independent of $\inf_{\Omega_{0}}f$. 
\end{lemma}
\begin{proof} 
It follows from \eqref{eqn2.16}, \eqref{eqn7.2}, and \eqref{eqn2.11} that 
\[
\operatorname{tr}\left(GD^{2}u_{ee}\right)\geq\frac{1}{\operatorname{tr}F^{ij}}\frac{2}{3}g^{-\frac{1}{3}}\left(\partial_{ee}g-\frac{1}{3}\frac{\left|\partial_{e}g\right|^{2}}{g}\right)
\geq-\frac{K}{\operatorname{tr}F^{ij}}\geq-CK
\]
for all $e\in\mathbb{S}^{n-1}$, where $g=f^{3/(2k)}$, $C>0$ is universal, and  $K>0$ depends on $\operatorname{dist}(\Omega,\partial\Omega_{0})$ and $\|f^{3/(2k)}\|_{C^{2,1}(\overline{\Omega_{0}})}$. 
\end{proof}

\begin{lemma}\label{thm7.3}
Let $\inf_{\Omega_{0}}f>0$ where $\Omega\Subset\Omega_{0}$, and $f^{3/(2k)}\in C^{2,1}(\overline{\Omega_{0}})$. 
Then 
\[
\mathcal{L}(w)\geq-C\quad \text{in}\ \Omega, 
\]
where $C$ depends on $\Omega$, $\operatorname{dist}(\Omega,\partial\Omega_{0})$, and 
$\|f^{3/(2k)}\|_{C^{2,1}(\overline{\Omega_{0}})}$, but is independent of $\inf_{\Omega_{0}}f$. 
\end{lemma}
\begin{proof}
It follows from \eqref{eqn7.10} that
\[
\mathcal{L}(w)\geq\frac{2}{3}g^{-\frac{1}{3}}\left(\partial_{\tau\tau}g-\frac{1}{3}\frac{\left|\partial_{\tau}g\right|^{2}}{g}\right)+ \left(\frac{2}{3}g^{\frac{1}{6}}\frac{\partial_{x_{i}}g}{\sqrt{g}}\right)\tau_{j}\partial_{x_{j}}(\tau_{i})
\]
in $\Omega$, where $g=f^{3/(2k)}$. 
Using Lemma \ref{thm2.2}, we obtain
\[
\mathcal{L}(w)\geq -C,
\] 
where $C$ depends on $\|\rho\|_{C^{3}}$, $\operatorname{dist}(\Omega,\partial\Omega_{0})$, and  
$\|f^{3/(2k)}\|_{C^{2,1}(\overline{\Omega_{0}})}$.
\end{proof}

\par
\vspace{2mm}
Using Lemmas \ref{thm7.4} and \ref{thm7.3}, with the same proof as in subsection 6.1, we prove Theorem \ref{thm10.3}.

\end{document}